\documentclass{amsart}
\usepackage[fontsize=11pt]{scrextend}
\usepackage{amsmath,graphicx,amsfonts,amssymb,amsthm,hyperref,amscd,esint,bbm,verbatim,mathabx,abstract}
\usepackage[margin=1in]{geometry}
\usepackage{mathtools}
\mathtoolsset{showonlyrefs}
\newtheorem{theorem}{Theorem}[section]

\newtheorem{lemma}[theorem]{Lemma}
\numberwithin{equation}{section}
\usepackage[style=numeric,backend=bibtex]{biblatex}
\title{Global, Non-scattering solutions to the quintic, focusing semilinear wave equation on $\mathbb{R}^{1+3}$}
\date{}
\author{Mohandas Pillai} 
\numberwithin{equation}{section}
\begin{document}
\maketitle
\begin{abstract} \noindent We consider the quintic, focusing semilinear wave equation on $\mathbb{R}^{1+3}$, in the radially symmetric setting, and construct infinite time blow-up, relaxation, and intermediate types of solutions. More precisely, we first define an admissible class of time-dependent length scales, which includes a symbol class of functions. Then, we construct solutions which can be decomposed, for all sufficiently large time, into an Aubin-Talentini (soliton) solution, re-scaled by an admissible length scale, plus radiation (which solves the free 3 dimensional wave equation), plus corrections which decay as time approaches infinity. The solutions include infinite time blow-up and relaxation with rates including, but not limited to, positive and negative powers of time, with exponents sufficiently small in absolute value. We also obtain solutions whose soliton component has oscillatory length scales, including ones which converge to zero along one sequence of times approaching infinity, but which diverge to infinity along another such sequence of times. The method of proof is similar to a recent wave maps work of the author, which is itself inspired by matched asymptotic expansions.
 \end{abstract}
\tableofcontents
\section{Introduction}
We consider the quintic, focusing semilinear wave equation on $\mathbb{R}^{1+3}$, in the radially symmetric setting, namely
\begin{equation}\label{slw}-\partial_{t}^{2}u+\partial_{r}^{2}u+\frac{2}{r}\partial_{r}u+u^{5}=0,\quad  r>0\end{equation}
The following energy is conserved for sufficiently regular solutions to \eqref{slw}.
\begin{equation}\label{enslw} E_{SLW}(u,\partial_{t}u) = \int_{0}^{\infty} r^{2} dr\left(\frac{(\partial_{t}u)^{2}}{2}+\frac{(\partial_{r}u)^{2}}{2}-\frac{u^{6}}{6}\right)\end{equation}
If $u$ is a solution to \eqref{slw}, then, so is the function $u_{\lambda}(t,r):= \frac{1}{\sqrt{\lambda}} u(\frac{t}{\lambda},\frac{r}{\lambda}), \lambda>0$. We remark that the equation \eqref{slw} is energy critical, since, if $u$ is a sufficiently regular solution to \eqref{slw}, then, $E_{SLW}(u,\partial_{t}u) = E_{SLW}(u_{\lambda},\partial_{t}u_{\lambda})$ (since $E_{SLW}(u,\partial_{t}u)$ is independent of time). We denote the Aubin-Talentini soliton by
$$Q_{1}(r) = (\frac{r^2}{3}+1)^{-1/2}$$
For $\lambda>0$, we let
\begin{equation}\label{qlambdadef}Q_{\lambda}(r) = \lambda^{-1/2} Q_{1}(\frac{r}{\lambda})\end{equation}
The Cauchy problem associated to \eqref{slw} is locally well-posed in $\dot{H}^{1}(\mathbb{R}^{3}) \times L^{2}(\mathbb{R}^{3})$ (see, for instance, \cite{kcbms}, \cite{soggebook}, and references therein). 
The free wave equation \begin{equation}\label{free3dwaveeqn}-\partial_{t}^{2}u+\partial_{r}^{2}u+\frac{2}{r}\partial_{r}u=0\end{equation} will also be of use to us, and we recall its formally conserved energy.
\begin{equation}\label{en}E(u,\partial_{t}u) =\frac{1}{2}\left(||\partial_{t}u||_{L^{2}(r^{2} dr)}^{2} + ||\partial_{r}u||_{L^{2}(r^{2} dr)}^{2}\right)  \end{equation}
Part of the work \cite{dkm}, of Duyckaerts, Kenig, and Merle, showed that if $u$ is a global solution to \eqref{slw} with $E_{SLW}(u,\partial_{t}u) < 2 E_{SLW}(Q_{1},0)$, then, $u$ either scatters, or, there exists $v$ solving \eqref{free3dwaveeqn}, $n=\pm 1$, and $\lambda(t)$ defined for $t$ sufficiently large, so that
\begin{equation}\label{dkmasymp}\lim_{t \rightarrow \infty}E(u-v- n Q_{\lambda(t)},\partial_{t}u-\partial_{t}v) =0\end{equation}
Another characterization of solutions satisfying \eqref{dkmasymp} was provided in a special case of one of the results of \cite{knscenterstablemanifold}. If $(u,\partial_{t}u) \in C([0,\infty), \dot{H}^{1}(\mathbb{R}^{3}) \times L^{2}(\mathbb{R}^{3}))$ satisfies \eqref{dkmasymp}, then, $(u(0),\partial_{t}u(0)) \in M_{D}$, where $M_{D}$ is a $C^{1}$ connected manifold constructed in \cite{knscenterstablemanifold}.\\
\\
The work \cite{dk}, of Donninger and Krieger constructed infinite time blow-up and relaxation solutions to \eqref{slw}, with soliton length scale $\lambda(t) = t^{-\mu}$, for $|\mu|$ sufficiently small.\\ 
\\
For each positive $\lambda(t) \in C^{\infty}((50,\infty))$ satisfying the following inequalities for all sufficiently large $t$, and sufficiently small $C_{u},C_{l},C_{2} \geq 0$ (the precise constraints are in \eqref{cofconstr0} through \eqref{cofconstr2})
$$\frac{-C_{l}}{t} \leq \frac{\lambda'(t)}{\lambda(t)} \leq \frac{C_{u}}{t}, \quad \frac{|\lambda^{(k)}(t)|}{\lambda(t)} \leq \frac{C_{k}}{t^{k}}, \quad k \geq 2$$
the current work constructs a finite energy solution to \eqref{slw} which can be decomposed as 
$$u(t,r) = Q_{\lambda(t)}(r) + v_{rad}(t,r) + u_{e}(t,r)$$
where
$v_{rad}$ solves \eqref{free3dwaveeqn}, $E(u_{e},\partial_{t}(Q_{\lambda(t)}+u_{e})) \rightarrow 0$ as $t \rightarrow \infty$. (We remark that, aside from the precise smallness constraints on $C_{u},C_{l},C_{2}$, the  admissible class of $\lambda$ of this work is the same as that of a wave maps work of the author, \cite{wm2}.) To the knowledge of the author, the existence of all of the solutions constructed in this work, aside from those with asymptotically constant $\lambda(t)$, or $\lambda(t)$ equal to precisely a power of $t$ (as in \cite{dk}) is new. This includes new infinite time blow-up and relaxation solutions (see Remark 2 after Theorem \ref{mainthm}), along with solutions that have oscillatory soliton length scales satisfying any combination of 
\begin{equation}\label{lambdaosc}\begin{split}&\liminf_{t \rightarrow \infty} \lambda(t) \in \{0 , \lambda_{0}\} \quad \text{   and  }
\limsup_{t \rightarrow \infty} \lambda(t) \in\{ \lambda_{1} ,\infty\}, \quad 0< \lambda_{0} \leq \lambda_{1}\end{split}\end{equation}
(in addition to oscillatory $\lambda$ such that $\lim_{t \rightarrow \infty} \lambda(t) = 0 \text{  or  } \infty$) (see Remark 3 after Theorem \ref{mainthm}). In addition, the procedure used to construct the ansatz of this work, and many aspects of the procedure used to complete the ansatz to an exact solution, are quite different than that used in \cite{dk}.

We will now briefly summarize those aspects of the work \cite{kstslw} of Krieger, Schlag, and Tataru, which are necessary in order to precisely describe the admissible class of $\lambda$ of the current work. In the process of completing our approximate solution to an exact solution, we use the distorted Fourier transform, $\mathcal{F}$ from Proposition 4.3 of \cite{kstslw}. This is a Fourier transform associated to $L$, which is the elliptic part of a conjugation of the linearization of \eqref{slw} around $Q_{1}$. We also use the transference operator, $\mathcal{K}$, defined in $(2.7)$ of \cite{kstslw}. In the interest of brevity, we will not recall the definition of $\mathcal{K}$ until it is used in the proof of this work. What is important to note here is that,  by Proposition 5.2 of \cite{kstslw}, for all $\alpha \in \mathbb{R}$, $\mathcal{K}$ and $[\xi \partial_{\xi},\mathcal{K}]$ are bounded on $L^{2,\alpha}_{\rho}$, which has the norm
$$||f||_{L^{2,\alpha}_{\rho}}^{2} = |f(\xi_{d})|^{2} +||f(\xi) \langle \xi \rangle^{\alpha}||_{L^{2}((0,\infty),\rho(\xi) d\xi)}^{2} $$
Here, $\xi_{d}$ is the negative eigenvalue of $L$, and $\rho$ is the density of the continuous part of the spectral measure. Finally, by Lemma 4.6 of \cite{kstslw}, there exists $C_{\rho}>0$ so that
$$\frac{\rho(y)}{\rho(z)} \leq C_{\rho} \left(\sqrt{\frac{y}{z}}+\sqrt{\frac{z}{y}}\right), \quad y,z >0$$
Now, we can define our admissible class of $\lambda$, which we denote $\Lambda$. Let $\Lambda$ be the set of positive functions $\lambda \in C^{\infty}((50,\infty))$ such that there exists $T_{\lambda}>50, C_{u},C_{l},C_{k}\geq 0$ so that
\begin{equation}\label{lambdasymb} \frac{-C_{l}}{t} \leq \frac{\lambda'(t)}{\lambda(t)} \leq \frac{C_{u}}{t}, \quad \frac{|\lambda^{(k)}(t)|}{\lambda(t)} \leq \frac{C_{k}}{t^{k}}, \quad k \geq 2, t \geq T_{\lambda}\end{equation}
where, if $M:= \text{max}\{C_{u},C_{l}\}$,
\begin{equation}\label{cofconstr0}   C_{l}+\frac{163}{2352} C_{u} < \frac{1}{588}\end{equation}
\begin{equation}\label{cofconstr1}\begin{split}&40\left(\frac{3 M^{2}}{16} (111-45 \log(4))+\frac{15 C_{2}}{16}(-12+12 \sqrt{3}+15 \pi^{2}-4 \log(4))\right)\\
&+2(C_{2}(1+||\mathcal{K}||_{\mathcal{L}(L^{2,0}_{\rho})})+M^{2}(||\mathcal{K}||_{\mathcal{L}(L^{2,0}_{\rho})}+2||[\mathcal{K},\xi\partial_{\xi}]||_{\mathcal{L}(L^{2,0}_{\rho})}+||\mathcal{K}||_{\mathcal{L}(L^{2,0}_{\rho})}^{2})+2M(1+||\mathcal{K}||_{\mathcal{L}(L^{2,0}_{\rho})})) < \frac{1}{12 \sqrt{C_{\rho}}}\end{split}\end{equation}
and
\begin{equation}\label{cofconstr2} M (1+||\mathcal{K}||_{\mathcal{L}(L^{2,\frac{1}{2}}_{\rho})}) < \frac{1}{48\sqrt{C_{\rho}}}\end{equation}
\emph{Remark} (Smallness constraints \eqref{cofconstr0} through \eqref{cofconstr2} can be achieved by re-scaling the exponent of $t$)  If $f \in C^{\infty}((0,\infty))$ is any positive function satisfying, for all $k \geq 1$,
$$\frac{|f^{(k)}(t)|}{f(t)} \leq \frac{D_{k}}{t^{k}}, \quad t \geq T_{f}$$
(for some $T_{f}>50$) then, the function defined for $t > 50$ by 
\begin{equation}\label{flambda}\lambda(t) = f(t^{1/c})\end{equation} is in $\Lambda$ for $c>0$ sufficiently large. The main theorem of this work is the following.
\begin{theorem}\label{mainthm} For each $\lambda \in \Lambda$, there exists $T_{0}>0$ and a finite energy solution, $u$, to \eqref{slw}, for $t > T_{0}$, satisfying
$$u(t,r) = Q_{\lambda(t)}(r) + v_{rad}(t,r)+u_{e}(t,r)$$
where
$$(-\partial_{t}^{2}+\partial_{r}^{2}+\frac{2}{r}\partial_{r})v_{rad} = 0, \quad E(v_{rad},\partial_{t}v_{rad}) < \infty$$
and
$$E(u_{e},\partial_{t}(Q_{\lambda(t)}+u_{e})) \leq \frac{C \left(\sup_{x \in [T_{\lambda},t]} \sqrt{\lambda(x)}\right)^{2}}{t}$$
\end{theorem}
\emph{Remark 1}. As in \cite{wm2}, the Cauchy data of $v_{rad}$ is explicit in terms of $\lambda$. In fact, 
$$v_{rad}(0,r)=0, \quad \partial_{t}v_{rad}(0,r) = \frac{\psi(r)}{r} \left(\frac{-15 \pi \sqrt{\lambda(r)} \lambda''(r)}{8} -\frac{\sqrt{3}\lambda'(r)}{2 \sqrt{\lambda(r)}}\right) + v_{3,0}(r)+v_{4,0}(r) $$
where $v_{3,0}$ is defined in \eqref{v30def}, $v_{4,0}$ is defined in \eqref{v40def} (the function $N_{0}$ appearing in \eqref{v40def} is defined in \eqref{n0def}), and $\psi$ is a relatively unimportant cutoff, defined in \eqref{psidef} ($\psi(r) =0$ near $r=0$, and $\psi(r)=1$ for large $r$).\\
\\
\emph{Remark 2}. The solutions in Theorem \ref{mainthm} include infinite time blow-up and relaxation solutions including, but not limited to $\lambda(t)$ being a power of $t$, with exponent sufficiently small in absolute value. Firstly, if $T_{f}>0$, and $f \in C^{\infty}((T_{f},\infty))$ is a positive function satisfying
$$\frac{|f^{(k)}(t)|}{f(t)} \leq C_{f,k}t^{-k}, \quad k \geq 1, \quad t > T_{f}$$
then, if $g(t) = \log(f(t))$, we have
$$|g^{(n)}(t)| \leq C_{n} t^{-n}, \quad n \geq 1, \quad t > T_{f}$$
Therefore, a simple induction argument shows that for all $k \in \mathbb{N}$, there exists $T_{k}>0$ so that, if 
$$g_{k}(t) = \left(\text{k-th fold composition of }\log\right)(t), \quad t > T_{k}$$ 
then,  $g_{k}\in C^{\infty}((T_{k},\infty))$, and, for $t >T_{k}$,
$$g_{k}(t) >1, \quad |g_{k}^{(n)}(t)| \leq C_{n,k} t^{-n}, \quad n \geq 1$$
Then, if $k \in \mathbb{N}$ and $\alpha \in \mathbb{R}$ Faa di Bruno's formula, for instance, shows that if
$$h_{k,\alpha}(t) = (g_{k}(t+T_{k}))^{\alpha}, \quad t>50$$
we have $h_{k,\alpha} \in C^{\infty}((50,\infty))$, and for $t >50$,
$$ h_{k,\alpha}(t) >0, \quad \frac{|h_{k,\alpha}^{(n)}(t)|}{h_{k,\alpha}(t)} \leq \frac{C_{n,k,\alpha}}{t^{n} g_{k}(t+T_{k})}, \quad n\geq 1$$
(recall that $g_{k}(x)>1$ if $x >T_{k}$). Then, an induction argument shows that, for each $N \in \mathbb{N}$, if $\beta_{k} \in \mathbb{N}$, and $\alpha_{k} \in \mathbb{R}$, for all $1 \leq k \leq N$, and $\epsilon\geq 0$ is sufficiently small, 
$$ \lambda(t):= t^{\pm \epsilon} \prod_{w=1}^{N} h_{\beta_{w},\alpha_{w}}(t) \in \Lambda, \quad \lambda(t):= t^{\pm \epsilon} \in \Lambda$$
This gives infinite time blow-up and relaxation solutions with a variety of rates, including powers of $t$.\\
\\
\emph{Remark 3}. The solutions in Theorem \ref{mainthm} include  solutions with $\lambda(t)$ as in \eqref{lambdaosc}. It is relatively straightforward to verify that the following functions are in $\Lambda$ (and the complete details of this verification for the following functions is given in the remarks following Theorem 1.1 of \cite{wm2}). For  $0< \lambda_{0} \leq \lambda_{1} \in \mathbb{R}$, $0 \leq \alpha_{0},\alpha_{1}$, $\alpha_{0}+\alpha_{1} <1$
$$ \lambda(t) = \frac{\lambda_{0} \log^{-\alpha_{0}}(t) + \lambda_{1} \log^{\alpha_{1}}(t)}{2}+\frac{\left(\lambda_{1}\log^{\alpha_{1}}(t) - \lambda_{0}\log^{-\alpha_{0}}(t)\right)}{2} \sin(\log(\log(t)))$$
For $c>0$ sufficiently small, and $|a|$ sufficiently small, depending on $c$, 
$$\lambda(t) = t^{a}\left(2+c \sin(\log(t))\right)$$
As per \eqref{flambda}, there are many more possible forms of $\lambda \in \Lambda$. \\
\\
\emph{Remark 4}. The function $u$ from Theorem \ref{mainthm} is of the form $$u=Q_{\lambda(t)}(r) + v_{rad}(t,r)+v_{e}(t,r)+v_{f}(t,r)$$
where $v_{rad},v_{e}$ are fairly explicit, and 
$$(t,x) \mapsto v_{f}(t,|x|) \in C^{0}_{t}([T_{0},\infty),H^{2}(\mathbb{R}^{3})), \quad (t,x) \mapsto \partial_{1}v_{f}(t,|x|) \in C^{0}([T_{0},\infty),H^{1}(\mathbb{R}^{3}))$$
This is due to the definition of the space $Z$ (see \eqref{znormdef}), the fact that dilation is continuous on $L^{2}$ and Lemma 2.7 of \cite{kstslw}, which relates $H^{k}(\mathbb{R}^{3})$ to $L^{2,\alpha}_{\rho}$ spaces.\\
\\
As discussed previously, the work \cite{dk}, of Donninger and Krieger constructed infinite time blow-up and relaxation solutions to \eqref{slw}, with soliton length scale $\lambda(t) = t^{-\mu}$, for $|\mu|$ sufficiently small. The aforementioned work \cite{kstslw}, of Krieger, Schlag, and Tataru constructed finite time blow-up solutions to \eqref{slw}, with soliton length scale $\lambda(t) = t^{1+\nu}$, for $\nu > \frac{1}{2}$. The subsequent work of Krieger and Schlag, \cite{ksfullrangeslw}, extended the range of $\nu$, to all $\nu>0$. The stability of the solutions in \cite{kstslw} and \cite{ksfullrangeslw} was studied in the work  \cite{knstability}, of Krieger and Nahas. Work of Donninger, Huang, Krieger, and Schlag, \cite{exoticslw}, constructed other finite time blow-up solutions to \eqref{slw}, with soliton length scale given by $\lambda(t) = t^{1+\nu} e^{\epsilon_{0} \sin(\log(t))}$, with $\nu>3, |\epsilon_{0}| \ll 1$.\\
\\
We also mention some constructions of solutions to the critical semilinear wave equation in other dimensions. The work \cite{hrtypeii} constructed finite time blow-up solutions to the critical semilinear wave equation in $4+1$ dimensions.  Multisoliton solutions to the $5+1$ dimensional critical wave equation were constructed in \cite{mmmultisolitons} \cite{ymultisoliton}. Infinite time blow-up solutions involving multiple dynamically rescaled solitons for the $5+1$ dimensional critical wave equation were constructed in \cite{jmmultisolitons}. Two bubble solutions to the $6+1$ dimensional critical wave equation were constructed in \cite{j}.\\
\\
The work \cite{kmthresholdslw} studied the possible dynamical behavior of solutions to the critical wave equation in dimensions 3,4, and 5, which have energy strictly less than that of the ground state soliton. The work \cite{dmthresholddynamics} classified the possible dynamical behavior of solutions to the critical wave equation in spatial dimensions $3,4,$ and $5$, which have energy exactly equal to that of the ground state soliton. The work \cite{knsglobaldynamics} studied the dynamics of solutions with energy bounded above by a quantity slightly larger than the ground state energy. Stable manifolds for the critical wave equation in 3 spatial dimensions were constructed in  \cite{ksonthefocusingslw} (studied further in \cite{knsthreshold}), \cite{bcenterstablemanifold}, and \cite{knscenterstablemanifold}. Modulated soliton solutions emerging from randomized initial data were studied in \cite{kmrandominitdat}.
\subsection{Acknowledgments}This material is based upon work partially supported by the National Science Foundation under Grant No. DMS-2103106.
\section{Notation}

The equation \eqref{slw}, linearized around $Q_{1}(R)$ is 
\begin{equation}\label{linearizedeqn} -\partial_{t}^{2}u+\partial_{R}^{2}u+\frac{2}{R}\partial_{R}u+\frac{45 u}{(3+R^{2})^{2}}=0\end{equation}
Two linearly independent solutions to \eqref{linearizedeqn} are the following
\begin{equation}\label{phi0e2def}\phi_{0}(R) = -\frac{\sqrt{3} \left(R^2-3\right)}{2 \left(R^2+3\right)^{3/2}}, \quad e_{2}(R) = \frac{2 \left(R^4-18 R^2+9\right)}{\sqrt{3} R \left(R^2+3\right)^{3/2}}\end{equation}
In particular, the function denoted by $\phi_{0}$ in \cite{kstslw}, which we denote in this paper by $\widetilde{\phi_{0}}$ satisfies
$$\widetilde{\phi_{0}}(R) = 2 R \phi_{0}(R)$$
The following function, which is (negative of) the potential obtained by linearizing \eqref{slw} around $Q_{\lambda(t)}(r)$, will appear in many of our computations.
\begin{equation}\label{Vdef}V(t,r)=\frac{-45}{\lambda(t)^2 \left(\frac{r^2}{\lambda(t)^2}+3\right)^2}\end{equation}
\section{Summary of the Proof}
The method of proof of Theorem \ref{mainthm} is similar to that used in a wave maps work of the author, \cite{wm2}. For completeness, we summarize the argument here. We choose $\lambda \in \Lambda$, restrict our attention to sufficiently large $t$, and start our ansatz with the soliton $Q_{\lambda(t)}(r)$. Inserting $Q_{\lambda(t)}(r)$ into \eqref{slw} produces the error term $\partial_{t}^{2}Q_{\lambda(t)}(r)$. We improve this error term by finding an approximate solution to
\begin{equation}\label{introuceqn}\left(-\partial_{t}^{2}+\partial_{r}^{2}+\frac{2}{r}\partial_{r}+\frac{45 \lambda(t)^{2}}{(r^{2}+3\lambda(t)^{2})^{2}}\right)v = \partial_{t}^{2}Q_{\lambda(t)}(r)\end{equation}
(The operator on the left hand side is the linearization of \eqref{slw} around $Q_{\lambda(t)}$). Our starting point is a function of the form
\begin{equation}\label{introucdef}u_{c}(t,r) = \chi_{\leq 1}(\frac{r}{h(t)})u_{e}(t,r) + \chi_{\geq 1}(\frac{r}{h(t)}) u_{w}(t,r)\end{equation}
where $\chi_{\leq 1}, \chi_{\geq 1}$ are cutoffs satisfying \eqref{chiprops}, and we define a length scale $h(t) = \lambda(t) \left(\frac{t}{\lambda(t)}\right)^{a}$,  where $a$ satisfies \eqref{hdef}. (The function $h(t)$ is defined with the constraints \eqref{hdef} at the beginning of the argument, for convenience. One may not know at the beginning of the argument why $a$ is restricted to satisfy \eqref{hdef}, and one could keep $a$ fairly general, and then impose the constraints \eqref{hdef} at various stages in the argument, as needed. We will give some intuition about the choice of $a$ in this section).\\
\\
We will construct $u_{e}$ and $u_{w}$ to be good approximate solutions to \eqref{introuceqn} for small $r$ and large $r$, respectively, satisfying the following matching property: When one takes the asymptotic expansion of $u_{e}(t,r)$ and $u_{w}(t,r)$ for large $t$, in the region $r \sim h(t)$, one obtains functions of the form $$\sum_{j=-n}^{n} \sum_{l=0}^{m} \frac{c_{j,l}(t) \log^{l}(r)}{r^{j}}$$
for various $n,m \geq 0$  (see sections \ref{firstordermatching} \ref{secondordermatching}, and \ref{thirdordermatching}). We choose $u_{e}$ and $u_{w}$ so that sufficiently many of the $c_{j,l}(t)$ involved in the expansions of $u_{e}$ and $u_{w}$ are equal to each other. As we explain in more detail shortly, this is accomplished by essentially keeping enough degrees of freedom in $u_{e}$ and $u_{w}$, which are then fixed when imposing the matching conditions. We start with $u_{ell}$, the first term in $u_{e}$, which solves the following equation.
\begin{equation}\label{introuelleqn} \left(\partial_{r}^{2}+\frac{2}{r}\partial_{r}+\frac{45 \lambda(t)^{2}}{(r^{2}+3\lambda(t)^{2})^{2}}\right)u_{ell}(t,r) = \partial_{t}^{2}Q_{\lambda(t)}(r)\end{equation}
(In other words, we remove the $\partial_{t}^{2}$ term from the left hand side of \eqref{introuceqn} when considering the small $r$ corrections). The most general solution to \eqref{introuelleqn} which is not singular at the origin depends on one function of $t$, say $c_{1}(t)$ (see \eqref{uellexp}). Recall that we must use sufficiently general solutions to the equations that we consider in order to be able to do the matching procedure, up to the order required, and we can regard $c_{1}(t)$ as a degree of freedom in $u_{ell}$. \\
\\
While the linear error term of $u_{ell}$, which is $\partial_{t}^{2}u_{ell}$, is better than that of the soliton for small $r$, it is not sufficiently small for our purposes. So, we iterate the above procedure, obtaining first $u_{ell,2}$ solving \eqref{introuelleqn}, with $u_{ell}$ on the left hand side replaced by $u_{ell,2}$, and $Q_{\lambda(t)}(r)$ on the right-hand side replaced by $u_{ell}$, see \eqref{uell2def}. Then, we obtain $u_{ell,3}$ by replacing $u_{ell,2}$ with $u_{ell,3}$, and replacing $u_{ell}$ by $u_{ell,2}$, see \eqref{uell3def}. We remark that $u_{ell,2}$ and $u_{ell,3}$ are not the most general solutions to their equations which are not singular at the origin, but, are sufficient for our purposes.\\
\\
Next, we consider $u_{w}$, the large $r$ approximate solution. To understand the first few terms in $u_{w}$, one can start with a general (finite energy) solution to
\begin{equation}\label{uwintrofirsteqn}\left(-\partial_{t}^{2}+\partial_{r}^{2}+\frac{2}{r}\partial_{r}\right)u = \partial_{t}^{2}Q_{\lambda(t)}(r)\end{equation} 
(note that the operator on the left-hand side is obtained by removing the potential term from \eqref{introuceqn}). This general solution can be expressed as  $w_{1}(t,r)+v_{ex}(t,r)+v_{2}(t,r)$ where
$w_{1}$ is a particular solution to \eqref{uwintrofirsteqn}, with $ \partial_{t}^{2}Q_{\lambda(t)}(r)$ replaced by its leading term for large $r$ (see \eqref{w1eqn}), $v_{ex}$ is a particular solution to \eqref{uwintrofirsteqn}, with the rest of $\partial_{t}^{2}Q_{\lambda(t)}(r)$ on the right-hand side (see \eqref{vexeqn}, and $v_{2}$ is a free wave (with Cauchy data $v_{2}(0,r)=0, \quad \partial_{t}v_{2}(0,r) = v_{2,0}(r)$ where $v_{2,0} \in L^{2}(r^{2} dr)$ is not yet specified, in this discussion). $v_{2}$ can be regarded as a degree of freedom in the general solution to \eqref{uwintrofirsteqn}. (In the actual proof, for convenience, we define $v_{2}$ with a fixed explicit $v_{2,0}$ which leads to an eventual matching between $u_{w}$ and $u_{e}$ to first order. In order to explain why the particular choice of $v_{2,0}$ is made in the proof, we leave it general in this discussion). Also, when $\lambda'(t) \neq 0$, $Q_{\lambda(t)}$ and $w_{1}(t)$ both have infinite kinetic energy, but, an exact cancellation of the large $r$ behavior of each of these functions shows that $\partial_{t}(Q_{\lambda(t)}(r) + \chi_{\geq 1}(\frac{r}{h(t)}) w_{1}(t,r)) \in L^{2}(r^{2} dr)$, see \eqref{dtw10qlambdacancel}.\\
\\
We can now do a first order matching of $u_{e}$ and $u_{w}$. (This is actually only one part of what is called ``first order matching'' in  section \ref{firstordermatching}).  Note that the leading behavior of $u_{ell}$, $w_{1}$ and $v_{2}$ in the matching region is given by $u_{ell,main}$, $w_{1,lm}$ and $v_{2,lm}$.
\begin{equation}\begin{split}&u_{ell,main}(t,r) \\
&= -\frac{\sqrt{3}c_{1}(t) \lambda(t)}{2 r}+\sqrt{\lambda(t)} \lambda''(t) \left(\frac{675 \pi  \lambda(t)^2}{16 r^2}-\frac{3 \sqrt{3} \lambda(t) \left(37-20 \log \left(\frac{r^2}{3 \lambda(t)^2}\right)\right)}{8 r}+\frac{\sqrt{3} r}{4 \lambda(t)}-\frac{15 \pi }{8}\right)\\
&+\frac{\left(\frac{15 \sqrt{3} \lambda(t)}{16 r}-\frac{\sqrt{3} r}{8 \lambda(t)}\right) \lambda'(t)^2}{\sqrt{\lambda(t)}}\end{split}\end{equation}
$$w_{1,lm}(t,r) = \frac{1}{2} r \left(\frac{\sqrt{3} \lambda''(t)}{2 \sqrt{\lambda(t)}}-\frac{\sqrt{3} \lambda'(t)^2}{4 \lambda(t)^{3/2}}\right)$$
$$v_{2,lm}(t,r) =t v_{2,0}(t)$$
Here, we use the fact that $\lambda(t) \ll h(t) \ll t$, and work under the assumption of sufficiently regular $v_{2,0}$ (we will end up choosing $v_{2,0}$  to be a smooth function). Note that our freedom to choose $v_{2,0}$ only affects terms of the form $f_{0}(t)$, while there are larger terms of the form $r f_{1}(t)$ involved in both $u_{ell}$ and $w_{1}$. On the other hand, the terms of the form $r f_{1}(t)$ in $u_{ell}$ and $w_{1,lm}$ automatically match (in other words, they are the same). A similar kind of automatic matching of terms that are larger than those involving the degrees of freedom is observed in \cite{wm2}. Therefore, we can choose $v_{2,0}$ so that the terms of the form $f_{0}(t)$ in the difference $u_{ell,main}-w_{1,lm}-v_{2,lm}$ vanish. Doing so requires that 
$$v_{2,0}(r) = \frac{-15 \pi \sqrt{\lambda(r)} \lambda''(r)}{8 r}, \quad r \geq T_{0}$$
(since we only consider $t \geq T_{0}$). This explains why, in the proof, $v_{2}$ is defined in \eqref{v20def} with the particular Cauchy data $v_{2,0}(r)=\frac{\psi(r)}{r}\left(\frac{-15}{8} \pi \sqrt{\lambda(r)} \lambda''(r)\right)$, where $\psi$ is defined in \eqref{psidef}. In this discussion, we now fix $v_{2,0}$ by the above formula.\\
\\
We provide some intuition regarding the constraints on $a$, \eqref{hdef}. The linear error term of any of our small $r$ corrections, say $u_{small}$ is $\partial_{t}^{2}u_{small}$, while that of a large $r$ correction is $V(t,r) u_{large}(t,r)$. In order to identify a region of the form $r \sim h(t)$ where $u_{small}$ and $u_{large}$, as well as their linear error terms, are of comparable size, we note that $u_{small}$ is a symbol in $t$, and recall the definition of $V$, given in \eqref{Vdef}. Therefore, the error terms of $u_{small}$ and $u_{large}$ are of comparable size when  $h(t) \gg \lambda(t)$, and $t^{-2} \sim \lambda(t)^{2} h(t)^{-4}$. This would suggest that $h(t) \sim \sqrt{t \lambda(t)}$. Our actual definition of $h(t)$ is larger than $\sqrt{t\lambda(t)}$ simply because increasing $h(t)$ improves the error terms associated to $u_{w}$ in the matching region (though it does make the error terms associated to $u_{e}$ larger) and it turns out that some error terms associated to $u_{w}$ are more delicate than those associated to $u_{e}$. On the other hand, $a$ can not be too large, since this would make some of the error terms of $u_{e}$ too large.\\
\\ 
Next, we consider the higher order terms in $u_{w}$. The linear error term associated to $w_{1}+v_{2}$ is  $V(t,r) (w_{1}+v_{2}):= RHS_{2}(t,r)$, where we recall the notation \eqref{Vdef}. Therefore, we consider a general, finite energy solution to \eqref{uwintrofirsteqn} (with the right-hand side replaced by $RHS_{2}$) given by $u_{w,2}+v_{3}$, where $v_{3}$ is a finite energy free wave with Cauchy data not yet specified, and $u_{w,2}$ is the particular solution with zero Cauchy data at infinity. As before, $v_{3}$ can be regarded as a degree of freedom in our second correction in $u_{w}$.\\
\\
Then, we compute the leading behavior of $v_{ex}$ and $u_{w,2}$ in the matching region (see Section \ref{firstordermatching}). As shown in Section \ref{firstordermatching}, $u_{ell,main}$ as well as the leading parts of $v_{ex}$ and $u_{w,2}$ contain terms of the form $f_{0}(t) r^{-2}$ and $f_{1}(t) \log(r) r^{-1}$, which can not be fixed by an appropriate choice of the degrees of freedom of our solutions. However, exactly as before, and analogously to \cite{wm2}, these terms happen to exactly match each other, so that it is possible to choose $c_{1}(t)$, the degree of freedom in $u_{ell}$, so that $u_{ell}$ matches $w_{1}+v_{2}+v_{ex}+u_{w,2}$ to an appropriate order (see again Section \ref{firstordermatching}). \\
\\
Higher order matching, done in Section \ref{secondordermatching}, fixes the Cauchy data of $v_{3}$. In Section \ref{thirdordermatching}, we verify an ``automatic'' matching of certain higher order terms. At this stage, we have fully constructed $u_{c}$, given in \eqref{introucdef}. Recalling that $u_{w}$ appears in \eqref{introucdef} with a factor of $\chi_{\geq 1}(\frac{r}{h(t)})$, we add another term to the ansatz, $\psi_{2}(\frac{r}{h(t)}) u_{3}$, in order to eliminate the linear error term $V(t,r) v_{3}(t,r) \chi_{\geq 1}(\frac{r}{h(t)})$. ($\psi_{2}$ is defined in \eqref{psi2def}, see also \eqref{u3formula}).\\
\\
At this stage, we have the preliminary ansatz $ Q_{\lambda(t)}(r)+ u_{c}(t,r)+\psi_{2}(\frac{r}{h(t)}) u_{3}(t,r)$ whose linear error term is small (see Lemmas \ref{eu3lemma}, \ref{ematch0estlemma}, \ref{ew2lemma}, \ref{eexlemma}, \ref{eell3lemma}). The nonlinear interactions of our preliminary ansatz are decomposed into a perturbative term $N_{1}$, and a non-perturbative part, $N_{0}$, see section \ref{nonlinear1section}. We eliminate $N_{0}$ by first considering $u_{N_{0}}$, the solution to 
$$-\partial_{t}^{2}u_{N_{0}}+\partial_{r}^{2}u_{N_{0}}+\frac{2}{r}\partial_{r}u_{N_{0}}=N_{0}$$
with zero Cauchy data at infinity. The linear error term of $u_{N_{0}}$, namely $V(t,r) u_{N_{0}}$, is worst near the origin, due to the fast decay of $V(t,r)$ for large $r$, and turns out to be too large near the origin. We therefore add a free wave, $v_{4}$, to $u_{N_{0}}$, which cancels the leading part of $u_{N_{0}}(t,r)$ in the region $r \leq \frac{t}{2}$, thereby allowing the linear error term of $u_{N_{0}}+v_{4}$, which is $e_{1} := V(t,r) (u_{N_{0}}+v_{4})(t,r)$ to be much smaller than $V(t,r) u_{N_{0}}(t,r)$, for instance, in the region $r \leq \frac{t}{2}$. There are still two parts of $e_{1}$ which are not quite perturbative. These are eliminated by $u_{4}$ \eqref{u4eqn} and $u_{N_{0},ell}$\eqref{un0elleqn}, see also \eqref{eun0ev4defs}. At this stage, we have the improved ansatz\\
$Q_{\lambda(t)}(r)+ u_{c}(t,r)+\psi_{2}(\frac{r}{h(t)}) u_{3}(t,r)+u_{N_{0}}(t,r) + v_{4}(t,r) + u_{4}(t,r) + u_{N_{0},ell}(t,r) \psi_{1}(\frac{r}{t})$. The self interactions of the new free wave, $v_{4}(t,r)$, are not quite perturbative for large $r$, though the rest of the nonlinear interactions of our ansatz are (see \eqref{n20def}). We therefore add a final term, $u_{N_{2}}$, to the ansatz in order to eliminate the $v_{4}$ self-interactions for large $r$, see \eqref{un2eqn}.\\
\\
Denoting our final ansatz by $Q_{\lambda(t)}(r) + u_{ansatz}$, we are now ready to substitute $u=Q_{\lambda(t)}(r) + u_{ansatz} +  v$ into \eqref{slw}, and solve the resulting equation for $v$ (\eqref{veqn}) using a fixed point argument. More precisely, we formally derive the equation \eqref{yeqn} solved by $y$ defined by
\begin{equation}y(t,\xi) = (y_{0}(t),y_{1}(t,\xi \lambda(t)^{-2})) = \mathcal{F}\left(\left(\cdot\right) v(t,\cdot\lambda(t))\right)(\xi)\end{equation}
where $\mathcal{F}$ is the distorted Fourier transform from \cite{kstslw}, which we regard as a two-component vector, as in Proposition 4.3 of \cite{kstslw}. Then, we prove that \eqref{yeqn} has a solution, say $(y_{6,0},y_{6,1})$, which is regular enough to allow us to justify the statement that the function $v$ defined by the following expression, with $y_{0}=y_{6,0}, y_{1}=y_{6,1}$
\begin{equation}v(t,r) = \frac{\lambda(t)}{r}\left(y_{0}(t) \phi_{d}(\frac{r}{\lambda(t)})+\int_{0}^{\infty} \phi(\frac{r}{\lambda(t)},\xi) y_{1}(t,\frac{\xi}{\lambda(t)^{2}}) \rho(\xi) d\xi\right)\end{equation}
is a solution to \eqref{veqn}. The equation \eqref{yeqn} is solved with a fixed point argument in the space $Z$, defined in Section \ref{iterationspacesection}. (The map $T$, whose fixed point is a solution to \eqref{yeqn}, is defined in \eqref{tdef}). The fixed point argument essentially only uses Minkowski's inequality, and energy-type estimates. The full details are given just after Lemma \ref{rhslemma}.
\section{Construction of the Ansatz}
Let $\lambda \in \Lambda$. By definition of $\Lambda$ and \eqref{cofconstr0}, we have
\begin{equation}\label{lambdacomparg} \left(\frac{x}{t}\right)^{-C_{l}} \leq \frac{\lambda(x)}{\lambda(t)} \leq \left(\frac{x}{t}\right)^{C_{u}}, x \geq t \geq T_{\lambda}, \quad x \mapsto \frac{\lambda(x)^{5/2}}{x} \text{ is decreasing on }[T_{\lambda},\infty)\end{equation}
In particular, by \eqref{cofconstr0}, there exists $T_{\lambda,0}>0$ so that, for all $t \geq T_{\lambda,0}$, $t-\lambda(t) \geq 2T_{\lambda}$.  Let $T_{0}>0$ satisfy the following, but be otherwise arbitrary.  
\begin{equation}\label{t0init}T_{0}\geq e^{200}(1+T_{\lambda}+T_{\lambda,0}):=T_{0,1}\end{equation}
For the rest of the paper, we restrict $t$ to satisfy $t \geq T_{0}$. Throughout the proof, the letter $C$ will denote a constant, whose value may change from line to line, but which is \emph{independent} of $T_{0}$, unless otherwise stated. Define
\begin{equation}h(t)=\lambda(t)\left(\frac{t}{\lambda(t)}\right)^{a}\end{equation}
where $a$ is any real number satisfying
\begin{equation}\label{hdef}\frac{2(2+\frac{5}{2}C_{u}+84C_{l})}{7(1-C_{u})} < a < \text{min}\{\frac{2(4-7C_{u}-12C_{l})}{13(1-C_{u})},\frac{2(1-4C_{u}-12C_{l})}{3(1-C_{u})}\}\end{equation}
In particular, $a < \frac{2}{3}$, and thus the lower bound on $a$ implies that  $x \mapsto h(x)$ is increasing, and that there exists $C>0$ so that $h(t) \geq C \sqrt{t}, \quad t \geq T_{\lambda}$. It is possible to choose $a$ as in \eqref{hdef}, by \eqref{cofconstr0}.  We start with the near origin corrections.
\subsection{First small $r$ correction}
We define $u_{ell}$ to be the following solution to 
$$\partial_{r}^{2}u_{ell}+\frac{2}{r}\partial_{r}u_{ell} +\frac{45 u_{ell}}{\lambda(t)^2 \left(\frac{r^2}{\lambda(t)^2}+3\right)^2} = \partial_{t}^{2}Q_{\lambda(t)}(r) $$
\begin{equation}\label{uellexp}\begin{split}u_{ell}(t,R \lambda(t)) &= \phi_{0}(R) \int_0^{R} s^2 \lambda(t)^2 f(s \lambda(t)) e_{2}(s) \, ds-e_{2}(R) \int_0^{R} s^2 \lambda(t)^2 f(s \lambda(t)) \phi_{0}(s) \, ds+c_{1}(t)\phi_{0}(R)\\
&=\frac{f_{1}(R) \lambda'(t)^2}{\sqrt{\lambda(t)}}+f_{2}(R) \sqrt{\lambda(t)} \lambda''(t)+c_{1}(t)\phi_{0}(R)\end{split}\end{equation}
where
$$f_{1}(R) = -\frac{\sqrt{3} R^2 \left(R^2-3\right)}{8 \left(R^2+3\right)^{3/2}}$$
\begin{equation}\begin{split}f_{2}(R) &=\frac{\sqrt{3} R \left(R^4-66 R^2+30 \left(R^2-3\right) \log \left(\frac{1}{3} \left(R^2+3\right)\right)+45\right)-15 \left(R^4-18 R^2+9\right) \tan ^{-1}\left(\frac{R}{\sqrt{3}}\right)}{4 R \left(R^2+3\right)^{3/2}}\end{split}\end{equation}
and $c_{1}$ will be chosen later. Let 
$$f_{1,0}(R) =\frac{15 \sqrt{3}}{16 R}-\frac{\sqrt{3} R}{8}, \quad f_{1,1}(R) = f_{1}(R)-f_{1,0}(R)$$ 
$$f_{2,0}(R) =\frac{675 \pi }{16 R^2}-\frac{3 \sqrt{3} \left(37-20 \log \left(\frac{R^2}{3}\right)\right)}{8 R}+\frac{\sqrt{3} R}{4}-\frac{15 \pi }{8}, \quad f_{2,1}(R) = f_{2}(R)-f_{2,0}(R)$$
Then, for $k=0,1$,
$$\lim_{s \rightarrow \infty}\left(\frac{s^{3}}{\log(s)}f_{k,1}(s)\right) < \infty$$
Then, the leading part of $u_{ell}$ in the matching region $r \sim h(t)$ will turn out to be:
\begin{equation}\label{uellmaindef}\begin{split}u_{ell,main}(t,r) &= -\frac{\sqrt{3} c_{1}(t) \lambda(t)}{2 r}+\frac{f_{1,0}(\frac{r}{\lambda(t)}) \lambda'(t)^2}{\sqrt{\lambda(t)}}+f_{2,0}(\frac{r}{\lambda(t)}) \sqrt{\lambda(t)} \lambda''(t)\end{split}\end{equation}
\subsection{Second small $r$ correction}
The second correction for small $r$ eliminates the linear error term of $u_{ell}$. In particular, we define $u_{ell,2}$ to be the following solution to
$$\partial_{r}^{2}u_{ell,2}+\frac{2}{r}\partial_{r}u_{ell,2}+\frac{45 u_{ell,2}}{\lambda(t)^2 \left(\frac{r^2}{\lambda(t)^2}+3\right)^2} = \partial_{t}^{2}u_{ell}(t,r) $$
\begin{equation}\label{uell2def}u_{ell,2}(t,r) = \phi_{0}(\frac{r}{\lambda(t)}) \int_0^{\frac{r}{\lambda(t)}} s^2 \lambda(t)^2 \partial_{1}^{2}u_{ell}(t,s\lambda(t)) e_{2}(s)  ds-e_{2}(\frac{r}{\lambda(t)}) \int_0^{\frac{r}{\lambda(t)}} s^2 \lambda(t)^2\partial_{1}^{2}u_{ell}(t,s\lambda(t)) \phi_{0}(s)  ds\end{equation}
In order to compute the relevant terms associated to $u_{ell,2}$ for the eventual matching process (see Section \ref{secondordermatching}) we first decompose $u_{ell}$ as follows.
\begin{equation}\begin{split} u_{ell}(t,r) &= \frac{f_{1}(\frac{r}{\lambda(t)}) \lambda'(t)^2}{\sqrt{\lambda(t)}}+f_{2}(\frac{r}{\lambda(t)}) \sqrt{\lambda(t)} \lambda''(t)+c_{1}(t)\phi_{0}(\frac{r}{\lambda(t)})=u_{ell,main}(t,r)+e_{ell,1}(t,r)\end{split}\end{equation}
where
\begin{equation}\label{eell1def}e_{ell,1}(t,r) = \frac{f_{1,1}(\frac{r}{\lambda(t)}) \lambda'(t)^2}{\sqrt{\lambda(t)}}+f_{2,1}(\frac{r}{\lambda(t)}) \sqrt{\lambda(t)} \lambda''(t)+c_{1}(t)\left(\phi_{0}(\frac{r}{\lambda(t)})+\frac{\sqrt{3}\lambda(t)}{2 r}\right)\end{equation}
Then, for $r \geq \lambda(t)$, we re-write the formula for $u_{ell,2}$ as follows (the first four integrals contain the terms which will be required in the matching process).
\begin{equation}\label{uell2formatching}\begin{split} &u_{ell,2}(t,r)\\
&=\phi_{0}(\frac{r}{\lambda(t)}) \int_{1}^{\frac{r}{\lambda(t)}} s^{2} \lambda(t)^{2} \partial_{1}^{2}u_{ell,main}(t,s\lambda(t)) e_{2}(s) ds- e_{2}(\frac{r}{\lambda(t)}) \int_{1}^{\frac{r}{\lambda(t)}} s^{2}\lambda(t)^{2} \partial_{1}^{2}u_{ell,main}(t,s\lambda(t)) \phi_{0}(s) ds\\
&- e_{2}(\frac{r}{\lambda(t)}) \int_{1}^{\infty} s^{2}\lambda(t)^{2} \partial_{1}^{2}e_{ell,1}(t,s\lambda(t)) \phi_{0}(s) ds- e_{2}(\frac{r}{\lambda(t)}) \int_{0}^{1} s^{2}\lambda(t)^{2} \partial_{1}^{2}u_{ell}(t,s\lambda(t)) \phi_{0}(s) ds\\
&+\phi_{0}(\frac{r}{\lambda(t)}) \int_{1}^{\frac{r}{\lambda(t)}} s^{2} \lambda(t)^{2} \partial_{1}^{2}e_{ell,1}(t,s\lambda(t)) e_{2}(s) ds +\phi_{0}(\frac{r}{\lambda(t)}) \int_{0}^{1} s^{2} \lambda(t)^{2} \partial_{1}^{2}u_{ell}(t,s\lambda(t)) e_{2}(s) ds\\
&+ e_{2}(\frac{r}{\lambda(t)}) \int_{\frac{r}{\lambda(t)}}^{\infty} s^{2}\lambda(t)^{2} \partial_{1}^{2}e_{ell,1}(t,s\lambda(t)) \phi_{0}(s) ds\end{split}\end{equation}
\subsection{Third small $r$ correction}
We eliminate the error term of $u_{ell,2}$ with the correction $u_{ell,3}$, which the following solution to
$$\partial_{r}^{2}u_{ell,3}+\frac{2}{r}\partial_{r}u_{ell,3}+\frac{45 u_{ell,3}}{\lambda(t)^2 \left(\frac{r^2}{\lambda(t)^2}+3\right)^2}=\partial_{t}^{2}u_{ell,2}(t,r)$$
\begin{equation}\label{uell3def}u_{ell,3}(t,r) = \phi_{0}(\frac{r}{\lambda(t)}) \int_0^{\frac{r}{\lambda(t)}} s^2 \lambda(t)^2 \partial_{1}^{2}u_{ell,2}(t,s\lambda(t)) e_{2}(s)  ds-e_{2}(\frac{r}{\lambda(t)}) \int_0^{\frac{r}{\lambda(t)}} s^2 \lambda(t)^2\partial_{1}^{2}u_{ell,2}(t,s\lambda(t)) \phi_{0}(s)  ds\end{equation}
\subsection{First large $r$ correction}
We first note that
\begin{equation}\label{dttqlgrcancel}\partial_{t}^{2}Q_{\lambda(t)}(r) =\frac{\sqrt{3} \left(2 \lambda(t) \lambda''(t)-\lambda'(t)^2\right)}{4 r \lambda(t)^{3/2}} + O\left(\frac{1}{r^{3}}\right), \quad r \rightarrow \infty\end{equation}
Then, we let
\begin{equation}\label{w1def}w_{1}(t,r) = w_{1,0}(t,r) + \frac{t}{2} \int_{0}^{\pi} \sin(\theta) \widetilde{v_{2,0}}(\sqrt{r^{2}+t^{2}+2 r t \cos(\theta)}) d\theta:=w_{1,0}(t,r) + \widetilde{v_{2}}(t,r)\end{equation}
where
\begin{equation}\label{v20hattildedef}\widetilde{v_{2,0}}(r) = \frac{\psi(r)}{r} \left(\frac{-\sqrt{3} \lambda'(r)}{2 \sqrt{\lambda(r)}}\right), \quad w_{1,0}(t,r) = \frac{\sqrt{3} \left(\sqrt{\lambda(r+t)}-\sqrt{\lambda(t)}\right)}{r}\end{equation}
and
$\psi \in C^{\infty}(\mathbb{R})$, 
\begin{equation}\label{psidef}\psi(x)=\begin{cases} 0, \quad x \leq T_{\lambda}\\
1, \quad x \geq 2T_{\lambda}\end{cases}\end{equation}
By direct computation, $w_{1}$ solves
\begin{equation}\label{w1eqn}-\partial_{t}^{2}w_{1}+\partial_{r}^{2}w_{1}+\frac{2}{r}\partial_{r}w_{1} = \frac{\sqrt{3} \left(2 \lambda(t) \lambda''(t)-\lambda'(t)^2\right)}{4 r \lambda(t)^{3/2}}\end{equation}
The main contribution of $w_{1}$ in the matching region is
\begin{equation}\label{w1maindef}\begin{split} w_{1,main}(t,r) &=\sum _{j=0}^5 \frac{r^{j} \partial_{t}^{j}\left(\frac{\sqrt{3} \lambda'(t)}{2 \sqrt{\lambda(t)}}\right)}{(j+1)!}+t \widetilde{v_{2,0}}(t)+\frac{1}{6} r^2 \partial_{t}^{2}\left(t \widetilde{v_{2,0}}(t)\right) +\frac{1}{120} r^4  \partial_{t}^{4}\left(t \widetilde{v_{2,0}}(t)\right) \\
&:=\frac{r^{5}}{720} \partial_{t}^{5}\left(\frac{\sqrt{3} \lambda'(t)}{2 \sqrt{\lambda(t)}}\right)+w_{1,cm}(t,r)\end{split}\end{equation}
We also define
\begin{equation}\label{w1linmaindef}w_{1,lm}(t,r):= \sum _{j=0}^1 \frac{r^{j} \partial_{t}^{j}\left(\frac{\sqrt{3} \lambda'(t)}{2 \sqrt{\lambda(t)}}\right)}{(j+1)!}+t \widetilde{v_{2,0}}(t)\end{equation}
A direct application of the fundamental theorem of calculus (or Taylor's theorem) gives
\begin{lemma}\label{w1minusmainest} For $0 \leq j,k \leq 2$ or $j=3,k=0$, $$|\partial_{t}^{j}\partial_{r}^{k}(w_{1}-w_{1,main})(t,r)| \leq \frac{C r^{7-k}}{t^{8+j}} \sqrt{\lambda(t)}, \quad r \leq \frac{t}{2}$$
For any $j \geq 0$,
$$|\partial_{t}^{j}w_{1,lm}(t,r)| \leq \frac{C_{j} r \sqrt{\lambda(t)}}{t^{j+2}}, \quad |\partial_{t}^{j}(w_{1,cm}-w_{1,lm})(t,r)| \leq \frac{C_{j} r^{3} \sqrt{\lambda(t)}}{t^{4+j}}$$
$$|\partial_{t}^{j}(w_{1}-w_{1,cm})(t,r)| \leq \begin{cases} \frac{C_{j}r^{5} \sqrt{\lambda(t)}}{t^{6+j}}, \quad r \leq \frac{t}{2}\\
\frac{C_{j}\sup_{x \in [t,t+r]} \sqrt{\lambda(x)} \log(t+r)}{r}\left(\frac{1}{t^{j}}+\frac{1}{\langle t-r \rangle^{j}}\right) + \frac{C_{j} r^{3}\sqrt{\lambda(t)}}{t^{4+j}}, \quad r \geq \frac{t}{2}\end{cases}$$
\end{lemma}
Recalling \eqref{v20hattildedef}, we also have the following estimates.
\begin{lemma}\label{w1estlemma} For all $j,k \geq 0$, there exists $C_{j,k}>0$ such that for $r \leq \frac{t}{2}$,
$$|\partial_{t}^{j}\partial_{r}^{k} w_{1}(t,r)| \leq \frac{C_{j,k} \sqrt{\lambda(t)}}{t^{2+j}} \begin{cases} r, \quad k=0\\
\frac{1}{t^{k-1}}, \quad k \geq 1\end{cases}$$
and for $r \geq \frac{t}{2}$,
$$|\partial_{t}^{j}\partial_{r}^{k} w_{1,0}(t,r)| \leq \frac{C_{j,k} \sup_{x \in [t,t+r]}\sqrt{\lambda(x)}}{t^{j}r^{k+1}}, \quad |\partial_{t}^{j}\partial_{r}^{k} \widetilde{v_{2}}(t,r)| \leq \frac{C_{j,k} \left(\sup_{x \in [T_{\lambda},t+r]}\sqrt{\lambda(x)}\right) \log(t+r)}{r\langle r-t\rangle^{j+k}}$$
\begin{equation}\label{dtpdr}|\left(\partial_{t}+\partial_{r}\right)\widetilde{v_{2}}| \leq \frac{C \left(\sup_{x \in [T_{\lambda},t+r]}\sqrt{\lambda(x)}\right)\log(t+r)}{t r}, \quad r \geq \frac{t}{2}\end{equation}
\begin{equation}\label{v2tildenearorigin}t(|\partial_{t}\widetilde{v_{2}}(t,r)| + |\partial_{r}\widetilde{v_{2}}(t,r)|) + |\widetilde{v_{2}}(t,r)| \leq \frac{C \left(\sup_{x \in [T_{\lambda},t]}\sqrt{\lambda(x)}\right)}{t}, \quad r \leq \frac{t}{2}\end{equation}
\end{lemma}
\begin{proof} For $r \leq \frac{t}{2}$, we start with the definition of $w_{1}$, \eqref{w1def}, and get
\begin{equation}\label{w1forest}\begin{split}w_{1}(t,r)&=\frac{\sqrt{3}}{2} \int_{0}^{1} \frac{\lambda'(t+r y)}{\sqrt{\lambda(t+r y)}} dy + \frac{t}{2}\int_{-1}^{1} \widetilde{v_{2,0}}(\sqrt{r^{2}+t^{2}+2 r t u}) du\\
&=\frac{\sqrt{3}}{2} \int_{0}^{1} dy \int_{t}^{t+r y} ds \partial_{s}\left(\frac{\lambda'(s)}{\sqrt{\lambda(s)}}\right)+\frac{t}{2}\int_{-1}^{1}du \int_{0}^{r} dy \frac{\widetilde{v_{2,0}}'(\sqrt{y^{2}+t^{2}+2y t u}) (y+ tu)}{\sqrt{y^{2}+t^{2}+2 y t u}}\\
&=\frac{\sqrt{3}}{2}\int_{0}^{1} r dw (1-w)\partial_{s}\left(\frac{\lambda'(s)}{\sqrt{\lambda(s)}}\right)\Bigr|_{s=t+r w}+\frac{t}{2}\int_{0}^{r} dy \int_{-1}^{1} \frac{du \widetilde{v_{2,0}}'(\sqrt{y^{2}+t^{2}+2 y t u})}{\sqrt{y^{2}+t^{2}+2 y t u}}(y+t u)\end{split}\end{equation}
The symbol-type estimates on $\lambda$, \eqref{lambdasymb}, allow for the differentiation under the integral sign in $t$ and $r$, for the first term of the last line of the equation above, and give, for $r \leq \frac{t}{2}$,
$$|\partial_{t}^{j}\partial_{r}^{k} \left(\frac{\sqrt{3}}{2} \int_{0}^{1} r dy (1-y) \partial_{s}\left(\frac{\lambda'(s)}{\sqrt{\lambda(s)}}\right)\Bigr|_{s=t+r y}\right)| \leq C_{j,k} \frac{\sqrt{\lambda(t)}}{t^{2+j}} \begin{cases}r, \quad k=0\\
\frac{1}{t^{k-1}}, \quad k \geq 1\end{cases}$$ 
where we used \eqref{lambdacomparg}. To estimate the second term on the last line of \eqref{w1forest} in the region $r \leq \frac{t}{2}$, we first note that \eqref{v20hattildedef} implies that
\begin{equation}\label{v20tildeests}|\widetilde{v_{2,0}}^{(k)}(r)| \leq \frac{C_{k}\sqrt{\lambda(r)}}{r^{2+k}} \mathbbm{1}_{\{r \geq T_{\lambda}\}}\end{equation}
Direct estimation gives, for $r \leq \frac{t}{2}$,
\begin{equation}|\partial_{t}^{j}\partial_{r}^{k}\left(\frac{t}{2}\int_{0}^{r} dy\int_{-1}^{1} \frac{du \widetilde{v_{2,0}}'(\sqrt{y^{2}+t^{2}+2 y t u})}{\sqrt{y^{2}+t^{2}+2 y t u}} (y+t u)\right)| \leq \frac{C_{j,k} \sqrt{\lambda(t)}}{t^{2+j}}\begin{cases} r, \quad k=0\\
\frac{1}{t^{k-1}}, \quad k \geq 1\end{cases}\end{equation}
Finally, \eqref{v2tildenearorigin} is obtained by directly estimating \eqref{w1def}. For $w_{1,0}$ in the region $r \geq \frac{t}{2}$, we simply directly estimate \eqref{v20hattildedef}. For $\widetilde{v_{2}}$, we use the following.  If $f \in C^{\infty}([0,\infty))$, then, for $k \geq 1$, and $u \in \mathbb{R}$,
\begin{equation}\label{derivofsqrtcomp}\partial_{r}^{k}\left(f(\sqrt{r^{2}+t^{2}+2 r u})\right) = \sum_{l=0}^{k}\sum_{n=1}^{k} (r+u)^{l} c_{l,k,n} \frac{f^{(n)}(\sqrt{r^{2}+t^{2}+2 r u})}{\left(r^{2}+t^{2}+2 r u\right)^{\frac{l+k-n}{2}}}\end{equation}
Then, directly differentiating the definition of $\widetilde{v_{2}}$ in \eqref{w1def}, we get
\begin{equation}\label{drkv2tilde}|\partial_{r}^{k}\widetilde{v_{2}}(t,r)| \leq \frac{C_{k} \left(\sup_{x \in [T_{\lambda},t+r]}\sqrt{\lambda(x)}\right) \log(r+t)}{\langle r-t \rangle^{k} r}, \quad r \geq \frac{t}{2}\end{equation}
Next, we let $X = t \partial_{t}+r \partial_{r}$, and get
$$X(\widetilde{v_{2}})(t,r) = \widetilde{v_{2}}(t,r) + \frac{t}{2} \int_{0}^{\pi} \sin(\theta) \sqrt{r^{2}+t^{2}+2 r t \cos(\theta)} \widetilde{v_{2,0}}'(\sqrt{r^{2}+t^{2}+2 r t \cos(\theta)}) d\theta$$
Using $X(\widetilde{v_{2}}) = t(\partial_{t}+\partial_{r})\widetilde{v_{2}}+(r-t)\partial_{r}\widetilde{v_{2}}$, we get \eqref{dtpdr}. Next, again using \eqref{derivofsqrtcomp} and \eqref{drkv2tilde}, we get a first estimate:
\begin{equation}\label{drkdtv2tilde}|\partial_{r}^{k}\partial_{t}\widetilde{v_{2}}| \leq \frac{C_{k} \left(\sup_{x \in [T_{\lambda},t+r]}\sqrt{\lambda(x)}\right) \log(r+t)}{r \langle r-t \rangle^{k}} \left(\frac{1}{t} + \frac{r}{t \langle t-r \rangle}\right), \quad r \geq \frac{t}{2}\end{equation}
Next, we use
$$\widetilde{v_{2}}(t,r) = \frac{1}{2}\int_{-t}^{t} du \widetilde{v_{2,0}}(\sqrt{r^{2}+t^{2}+2 r t u})$$
and directly differentiate to get
$$|\partial_{r}^{k}\partial_{t}\widetilde{v_{2}}| \leq \frac{C_{k} \left(\sup_{x \in [T_{\lambda},t+r]}\sqrt{\lambda(x)}\right)}{\langle r-t \rangle^{2+k}}, \quad r \geq \frac{t}{2}$$
Combining this with \eqref{drkdtv2tilde}, we get
$$|\partial_{r}^{k}\partial_{t}\widetilde{v_{2}}(t,r)| \leq \frac{C_{k} \left(\sup_{x \in [T_{\lambda},t+r]}\sqrt{\lambda(x)}\right) \log(t)}{r \langle t-r \rangle^{k+1}}, \quad r \geq \frac{t}{2}$$
Finally, we differentiate the equation solved by $\widetilde{v_{2}}$, namely
$$\left(-\partial_{t}^{2}+\partial_{r}^{2}+\frac{2}{r}\partial_{r}\right)\widetilde{v_{2}} =0$$
to get
$$|\partial_{r}^{k}\partial_{t}^{2}\widetilde{v_{2}}| \leq \frac{C_{k,j} \left(\sup_{x \in [T_{\lambda},t+r]}\sqrt{\lambda(x)}\right) \log(t+r)}{r \langle t-r \rangle^{2+k+j}}, \quad r \geq \frac{t}{2}$$
An induction argument then finishes the proof of the lemma \end{proof}
Next, let
\begin{equation}\label{rhsdef}RHS(t,r) = \partial_{t}^{2}Q_{\lambda(t)}(r) - \left(\frac{\sqrt{3} \left(2 \lambda(t) \lambda''(t)-\lambda'(t)^2\right)}{4 r \lambda(t)^{3/2}}\right)\end{equation}
Note that the solution to
$$\begin{cases}-\partial_{t}^{2}v_{ex,s}+\partial_{r}^{2}v_{ex,s}+\frac{2}{r}\partial_{r}v_{ex,s}=0\\
v_{ex,s}(s,r) =0\\
\partial_{1}v_{ex,s}(s,r) = RHS(s,r)\end{cases}$$
is
$$v_{ex,s}(t,r) = \frac{f_{3}(s,r+s-t)-f_{3}(s,|r-(s-t)|)}{\pi  r}$$
where
\begin{equation}\label{vexf3def}\begin{split}f_{3}(s,x)&=\frac{\pi}{2} \int_{x}^{\infty} r RHS(s,r) dr\\
&=-\frac{\sqrt{3} \pi  \left(3 x \lambda(s)^2 \left(\sqrt{3 \lambda(s)^2+x^2}+6 x\right)+x^3 \left(\sqrt{3 \lambda(s)^2+x^2}-x\right)+27 \lambda(s)^4\right) \lambda'(s)^2}{8 \left(\lambda(s) \left(3 \lambda(s)^2+x^2\right)\right)^{3/2}}\\
&-\frac{\pi  \left(x \left(x-\sqrt{3 \lambda(s)^2+x^2}\right)+9 \lambda(s)^2\right) \lambda''(s)}{4 \sqrt{\frac{1}{3} x^2 \lambda(s)+\lambda(s)^3}}\end{split}\end{equation}
Then, we let $v_{ex}$ denote the following solution to
\begin{equation}\label{vexeqn}-\partial_{t}^{2}v_{ex}+\partial_{r}^{2}v_{ex}+\frac{2}{r}\partial_{r}v_{ex} = RHS(t,r)\end{equation}
\begin{equation}\label{vexdef}v_{ex}(t,r) = \int_{t}^{\infty}v_{ex,s}(t,r)ds\end{equation}
\begin{lemma} \label{vexestlemma}
For all $j \geq 0$, there exists $C_{j}>0$ such that
\begin{equation}\label{vexests}|\partial_{t}^{j}\partial_{r}^{k}v_{ex}(t,r)| \leq \begin{cases} \frac{C_{j} \lambda(t)^{3/2-k} \log(t)}{t^{2+j}}, \quad r \leq \lambda(t)\\
\frac{C_{j} \sup_{x \in [t,t+r]}\left(\lambda(x)^{5/2}\right) \log(t+r)}{r^{1+k}t^{2+j}}, \quad r \geq \lambda(t)\end{cases},\quad  k=0,1\end{equation}
In addition,
\begin{equation}\label{vexenest}||\partial_{t}v_{ex}||_{L^{2}(r^{2} dr)} + ||\partial_{r}v_{ex}||_{L^{2}(r^{2} dr)} \leq \frac{C \lambda(t)}{t}\end{equation}\end{lemma}
\begin{proof} 
By directly estimating \eqref{vexf3def} and its derivatives, we get
\begin{equation}\label{f3ests}|\partial_{s}^{j}\partial_{x}^{k}f_{3}(s,x)| \leq \frac{C_{j,k} \lambda(s)^{5/2}}{s^{2+j}} \frac{1}{\left(\text{max}\{x,\lambda(s)\}\right)^{k+1}}\end{equation}
Then, from the definition of $v_{ex}$ (\eqref{vexdef}), we get
$$v_{ex}(t,r) = \int_{r}^{\infty}dy \frac{ f_{3}(t+y-r,y)}{\pi r} - \int_{0}^{r} dy \frac{f_{3}(r+t-y,y)}{\pi r} - \int_{0}^{\infty} \frac{dy}{\pi r} f_{3}(t+r+y,y)$$
Differentiation under the integral sign is possible by \eqref{f3ests}, and direct estimation of the integrals gives \eqref{vexests} for $r \geq \lambda(t)$. Direct estimation of the above integral also gives \eqref{vexests} for $r \leq \lambda(t)$, and $k=0$. Finally, direct estimation of the $r$ derivative of 
\begin{equation} \partial_{t}^{j}v_{ex}(t,r) =  -\frac{1}{\pi} \int_{0}^{1} dz \partial_{1}^{j}f_{3}(t+r(1-z),rz) -\frac{1}{\pi} \int_{0}^{1} dz \partial_{1}^{j} f_{3}(t+r(1+z),rz) +\int_{r}^{\infty} \frac{dy}{\pi} \int_{-1}^{1} \partial_{1}^{j+1}f_{3}(r z+t+y,y) dz\end{equation} finishes the proof of \eqref{vexests}. For the energy estimate, \eqref{vexenest}, we note that
$$\partial_{r}v_{ex}(t,r)+\frac{v_{ex}(t,r)}{r} = \int_{t}^{\infty} ds \left(\frac{-1}{2r}\right)\begin{aligned}[t]&\left((r+s-t) RHS(s,r+s-t)+(s-t-r)RHS(s,|s-t-r|)\right)\end{aligned}$$
$$\partial_{t}v_{ex}(t,r) = \int_{t}^{\infty} ds \left(\frac{-1}{2r}\right)\left(-(r+s-t) RHS(s,r+s-t)+(s-t-r)RHS(s,|s-t-r|)\right)$$
Then, Minkowski's inequality and direct estimation of \eqref{rhsdef} give
$$||\partial_{r}v_{ex}+\frac{v_{ex}}{r}||_{L^{2}(r^{2} dr)}+ ||\partial_{t}v_{ex}(t,r)||_{L^{2}(r^{2} dr)} \leq \frac{C \lambda(t)}{t}$$
Finally, since $||\partial_{r}v_{ex}+\frac{v_{ex}}{r}||_{L^{2}(r^{2} dr)} < \infty$, the dominated convergence theorem shows that
$$||\partial_{r}v_{ex}+\frac{v_{ex}}{r}||^{2}_{L^{2}(r^{2} dr)} = \lim_{M \rightarrow \infty} \left(\int_{\frac{1}{M}}^{M} r^{2} dr (\partial_{r}v_{ex})^{2} + \int_{\frac{1}{M}}^{M} dr \partial_{r}(r v_{ex}^{2})\right)$$
Then, \eqref{vexests} and the monotone convergence theorem, show that
$$||\partial_{r}v_{ex}+\frac{v_{ex}}{r}||^{2}_{L^{2}(r^{2} dr)} =\int_{0}^{\infty} r^{2} dr (\partial_{r}v_{ex})^{2} $$
which completes the proof of the lemma.
\end{proof}
We define $v_{ex,ell}$ to be the following solution to
$$\partial_{r}^{2}v_{ex,ell}+\frac{2}{r}\partial_{r}v_{ex,ell}=RHS$$
\begin{equation}\begin{split}v_{ex,ell}(t,r) &=-\frac{1}{r}\int_{0}^{r} y^{2}RHS(t,y)dy - \int_{r}^{\infty} y RHS(t,y) dy\\
&=\frac{\sqrt{3} \lambda'(t)^2 \left(r \left(r-\sqrt{3 \lambda(t)^2+r^2}\right)+\lambda(t)^2 \left(45 \sinh ^{-1}\left(\frac{r}{\sqrt{3} \lambda(t)}\right)-\frac{24 r}{\sqrt{3 \lambda(t)^2+r^2}}\right)\right)}{8 r \lambda(t)^{3/2}}\\
&+\frac{\sqrt{3} \lambda''(t) \left(r \left(\sqrt{3 \lambda(t)^2+r^2}-r\right)+15 \lambda(t)^2 \sinh ^{-1}\left(\frac{r}{\sqrt{3} \lambda(t)}\right)\right)}{4 r \sqrt{\lambda(t)}}\end{split}\end{equation}
It will be convenient to define the following principal part of $v_{ex,ell}$ in the matching region:
\begin{equation}\label{vexell0def}\begin{split} &v_{ex,ell,0}(t,r)\\
&:=\frac{3 \sqrt{3} \sqrt{\lambda(t)} \left(\lambda'(t)^2 \left(15 \log \left(\frac{4}{3 \lambda(t)^2}\right)+30 \log (r)-17\right)+2 \lambda(t) \lambda''(t) \left(5 \log \left(\frac{4}{3 \lambda(t)^2}\right)+10 \log (r)+1\right)\right)}{16 r}\end{split}\end{equation}
A direct computation gives
\begin{lemma}\label{vexellminusell0estlemma} For $0 \leq j+k \leq 5$,
$$|\partial_{t}^{j}\partial_{r}^{k}(v_{ex,ell}(t,r)-v_{ex,ell,0}(t,r))| \leq \frac{C \lambda(t)^{3/2}}{t^{2+j}} \begin{cases} \frac{(1+|\log(\frac{r}{\lambda(t)})|) \lambda(t)}{r^{1+k}}, \quad r \leq \lambda(t)\\
\frac{\lambda(t)^{3}}{r^{3+k}}, \quad r \geq \lambda(t)\end{cases}$$\end{lemma}
$v_{ex,sub}(t,r):=v_{ex}(t,r)-v_{ex,ell}(t,r)$ solves
\begin{equation}\label{vexsubeqn}-\partial_{t}^{2}v_{ex,sub}+\partial_{r}^{2}v_{ex,sub}+\frac{2}{r}\partial_{r}v_{ex,sub}=\partial_{t}^{2}v_{ex,ell}\end{equation} 
In fact, we have the following formula for $v_{ex,sub}$:
\begin{lemma} \label{vexsublemma}
$$v_{ex,sub}(t,r) = \int_{0}^{\infty} dx \int_{|r-x|}^{r+x} ds \left(-\frac{x}{2r}\right) \partial_{1}^{2}v_{ex,ell}(t+s,x)$$\end{lemma}
\begin{proof}
\begin{equation}\begin{split} v_{ex}(t,r) &= \int_{t}^{\infty} ds \left(\frac{-1}{2r} \int_{|r-(s-t)|}^{r+s-t} y RHS(s,y) dy\right)= \int_{0}^{\infty} dy \int_{|r-y|}^{r+y} dw \left(\frac{-y}{2r} RHS(t+w,y)\right) \end{split}\end{equation}
Integrating by parts in $w$, we get
\begin{equation}\begin{split} v_{ex}(t,r) &= \int_{0}^{\infty} dy \left(\frac{-y}{2r}\right)\left(RHS(t+r+y,y)(r+y)-RHS(t+|r-y|,y) |r-y|\right)\\
&+\frac{1}{2r} \int_{0}^{\infty} dy y \int_{|r-y|}^{r+y} w \partial_{1}RHS(t+w,y) dw\\
&=\int_{0}^{\infty} dy \left(\frac{-y}{2r}\right)\left(RHS(t,y)(r+y)-RHS(t,y) |r-y|\right)\\
&+\int_{0}^{\infty} dy \left(\frac{-y}{2r}\right)\begin{aligned}[t]&\left(\left(RHS(t+r+y,y)-RHS(t,y)\right)(r+y)\right.\\
&\left.-\left(RHS(t+|r-y|,y)-RHS(t,y)\right) |r-y|\right)\end{aligned}\\
&+\frac{1}{2r}\int_{0}^{\infty} dy y \int_{|r-y|}^{r+y} w \partial_{1}RHS(t+w,y) dw\end{split}\end{equation}
The first term on the right-hand side of the last equality above is $v_{ex,ell}$. So,
\begin{equation}\begin{split}&v_{ex,sub}(t,r)\\
&= \int_{0}^{\infty} dy \left(\frac{-y}{2r}\right)\left(\int_{0}^{r+y} \partial_{1}RHS(t+w,y)(r+y-w) dw -\int_{0}^{|r-y|} \partial_{1}RHS(w+t,y)(|r-y|-w)dw\right)\\
&= \int_{0}^{\infty} dq \int_{0}^{\infty} dy \left(\frac{y}{4r}\right) \left(\partial_{1}RHS(t+r+q,y)-\partial_{1}RHS(t+|r-q|,y)\right)(q+y-|q-y|)\\
&=\int_{0}^{\infty} dx \int_{|r-x|}^{r+x} ds \left(-\frac{x}{2r}\int_{0}^{\infty} dy \left(\frac{-y}{2x}\right)\partial_{1}^{2}RHS(t+s,y)(x+y-|x-y|)\right)\\
&=\int_{0}^{\infty} dx \int_{|r-x|}^{r+x} ds \left(-\frac{x}{2r}\right) \partial_{1}^{2}v_{ex,ell}(t+s,y)\end{split}\end{equation}\end{proof}
The leading part of $v_{ex,sub}$ in the matching region is $v_{ex,sub,ell}$ (as the next lemma shows), given by
\begin{equation}\label{vexsubelldef}v_{ex,sub,ell}(t,r) = -\frac{1}{r} \int_{\lambda(t)}^{r} x^{2}\partial_{1}^{2}v_{ex,ell,0}(t,x) dx+ \int_{\lambda(t)}^{r} x \partial_{1}^{2}v_{ex,ell,0}(t,x) dx\end{equation}
In particular, $v_{ex,sub,ell}$ solves
$$\partial_{r}^{2}v_{ex,sub,ell}+\frac{2}{r}\partial_{r}v_{ex,sub,ell}=\partial_{t}^{2}v_{ex,ell,0}$$

\begin{lemma}\label{vexsubminussub0estlemma} For $j=0$, $0 \leq k \leq 2$ and $j=1$, $0 \leq k \leq 1$, and in the region $\lambda(t) \leq r \leq t$,
\begin{equation}\begin{split} &|\partial_{t}^{j}\partial_{r}^{k}\begin{aligned}[t]&\left(v_{ex,sub}(t,r)-v_{ex,sub,ell}(t,r)+\int_{t}^{\infty} ds g_{3}(s) + \int_{t}^{\infty} ds g_{2}''(s) \log\left(\frac{4 (s-t)^{2}}{3 \lambda(s)^{2}}\right)\right.\\
&\left.+\int_{t}^{\infty} ds (s-t)\partial_{s}^{2}\left(v_{ex,ell}(s,y)-v_{ex,ell,0}(s,y)\right)\Bigr|_{y=s-t}\right)|\end{aligned}\\
&\leq \begin{cases} \frac{C r^{2}\lambda(t)^{5/2}\log(t)}{t^{5}}+\frac{C\lambda(t)^{7/2} \log^{2}(t)}{t^{4}}, \quad j=k=0\\
\frac{C \lambda(t)^{5/2}\log^{2}(t)}{t^{4}}, \quad j=1,k=0\\
\frac{C r \lambda(t)^{5/2} \log(t)}{t^{5}} + \frac{C \lambda(t)^{7/2} \log^{2}(t)}{r t^{4}}, \quad j=0,k=1\\
\frac{C \lambda(t)^{5/2} \log^{2}(t)}{r t^{4}}, \quad j+k=2\end{cases}\end{split}\end{equation}
where
$$g_{1}(s) = \frac{3 \sqrt{3} \sqrt{\lambda(s)}((-17)\lambda'(s)^{2}+2 \lambda(s)\lambda''(s))}{16}, \quad g_{2}(s) = \frac{3 \sqrt{3}\sqrt{\lambda(s)} (15 \lambda'(s)^{2}+10 \lambda(s)\lambda''(s))}{16}$$
$$g_{3}(s) = \frac{2 g_{2}(s) ((\lambda'(s))^{2}-\lambda(s)\lambda''(s))+\lambda(s)(-4 \lambda'(s)g_{2}'(s)+\lambda(s)g_{1}''(s))}{\lambda(s)^{2}}$$\end{lemma}
\begin{proof}We start with the definition (recall \eqref{vexell0def})
$$v_{ex,sub,0}(t,r):= \int_{t}^{\infty}ds \int_{|s-t-r|}^{s-t+r} dy \left(\frac{-y}{2r}\right) \partial_{s}^{2}v_{ex,ell,0}(s,y)= \int_{t}^{\infty} ds \left(\frac{-1}{2r}\right) \int_{|s-t-r|}^{s-t+r}dy(g_{3}(s)+\log(\frac{4 y^{2}}{3\lambda(s)^{2}})g_{2}''(s))$$
where $g_{1},g_{2},g_{3}$ are as in the lemma statement. We compute the $y$ integrals, and decompose the $s$ integrals in the following form, which makes it clear which terms will cancel when $v_{ex,sub,ell}$ is subtracted. (Whenever $y$ appears in integrals below, it arises from the change of variable $s=t+r y$).
\begin{equation}\begin{split}v_{ex,sub,0}(t,r)&= \frac{r}{4}\left(2+\log(\frac{16 r^{4}}{9 \lambda(t)^{4}})\right)g_{2}''(t)+ \frac{r}{2}g_{3}(t) - \int_{t}^{\infty} ds g_{3}(s)+\int_{t}^{t+r}ds (g_{3}(s)-g_{3}(t))\left(1-\frac{(s-t)}{r}\right) \\
&-\frac{1}{2}\int_{0}^{1}dy (g_{2}''(t+r y)-g_{2}''(t)) r \begin{aligned}[t]&\left((1+y)(-2+\log(\frac{4}{3}\frac{r^{2}}{\lambda(t)^{2}}(1+y)^{2}))\right.\\
&+\left.(1-y)(2+\log(\frac{3 \lambda(t)^{2}}{4 r^{2}(1-y)^{2}}))\right)\end{aligned}\\
&-\frac{1}{2}\int_{0}^{1} dy g_{2}''(t+r y) r (2y)\log(\frac{\lambda(t)^{2}}{\lambda(t+r y)^{2}}) - \int_{t}^{\infty} ds g_{2}''(s) (2+\log(\frac{4 (s-t)^{2}}{3\lambda(s)^{2}}))\\
&+\int_{t}^{t+r} ds (g_{2}''(s)-g_{2}''(t)) \log(\frac{4 (s-t)^{2}}{3\lambda(s)^{2}}) -2 \int_{t}^{t+r} ds g_{2}''(t) \log(\frac{\lambda(s)}{\lambda(t)})\\
&+r(g_{2}''(t+r)-g_{2}''(t))(-1+\log(4))\\
&+r^{2}\int_{1}^{\infty} \left(y \log(1-\frac{1}{y^{2}}) + coth^{-1}(y)-y+\frac{y^{2}}{2} \log(1+\frac{2}{y-1})\right)g_{2}'''(t+r y) dy -2 (g_{2}'(t)+r g_{2}''(t))\end{split}\end{equation}
We remark that the last line of the above expression was obtained using integration by parts. We note that
\begin{equation}\begin{split}&\int_0^r \partial_{t}^{2}v_{ex,ell,0}(t,x) x \, dx-\frac{1}{r}\int_0^r \partial_{t}^{2}v_{ex,ell,0}(t,x) x^2 \, dx = \frac{1}{4} r g_{2}''(t) \left(\log \left(\frac{16 r^4}{9 \lambda(t)^4}\right)+2\right)-2 r g_{2}''(t)+\frac{1}{2} r g_{3}(t)\end{split}\end{equation}
Therefore, recalling \eqref{vexsubelldef}, we have
\begin{equation}\begin{split}&v_{ex,sub,0}-v_{ex,sub,ell}+\int_{t}^{\infty} ds g_{3}(s) + \int_{t}^{\infty} ds g_{2}''(s)\log(\frac{4 (s-t)^{2}}{3\lambda(s)^{2}})\\
&=\int_{t}^{t+r} ds (g_{3}(s)-g_{3}(t))(1-\frac{(s-t)}{r})\\
&-\frac{r}{2}\int_{0}^{1} dy (g_{2}''(t+r y)-g_{2}''(t))\left((1+y)(-2+\log(\frac{4 r^{2}}{3 \lambda(t)^{2}}(1+y)^{2}))+(1-y)(2+\log(\frac{3 \lambda(t)^{2}}{4 r^{2}(1-y)^{2}}))\right)\\
&-r \int_{0}^{1} dy g_{2}''(t+r y) y \log(\frac{\lambda(t)^{2}}{\lambda(t+r y)^{2}})+\int_{t}^{t+r} ds (g_{2}''(s)-g_{2}''(t)) \log(\frac{4 (s-t)^{2}}{3 \lambda(s)^{2}})\\
&-2 \int_{t}^{t+r} ds g_{2}''(t) \log(\frac{\lambda(s)}{\lambda(t)}) + r(g_{2}''(t+r) - g_{2}''(t)) (-1+\log(4))\\
&+r^{2}\int_{1}^{\infty} \left(y\log(1-\frac{1}{y^{2}}) + coth^{-1}(y) - y+\frac{y^{2}}{2} \log(1+\frac{2}{y-1})\right) g_{2}'''(t+r y) dy\\
&+\int_0^{\lambda(t)} \partial_{t}^{2}v_{ex,ell,0}(t,x) x \, dx-\frac{1}{r}\int_0^{\lambda(t)} \partial_{t}^{2}v_{ex,ell,0}(t,x) x^2 \, dx\end{split}\end{equation}
From Lemma \ref{vexellminusell0estlemma}, we thus get for $j=0, 0 \leq k \leq 2$ and $j=1, 0 \leq k \leq 1$,  
\begin{equation}\begin{split}&|\partial_{t}^{j}\partial_{r}^{k} \left(v_{ex,sub,0}-v_{ex,sub,ell}+\int_{t}^{\infty} ds g_{3}(s) + \int_{t}^{\infty} ds g_{2}''(s)\log(\frac{4 (s-t)^{2}}{3\lambda(s)^{2}})\right)|\\
&\leq \frac{C r^{2-k} \lambda(t)^{5/2} \log(t)}{t^{5+j}} + \frac{C \lambda(t)^{7/2} \log(t)}{r^{k}t^{4+j}}, \quad \lambda(t) \leq r \leq t \end{split}\end{equation}
Next, for $F(s,y) = \partial_{s}^{2}(v_{ex,ell}(s,y)-v_{ex,ell,0}(s,y))$, we have
\begin{equation}\label{vexsubminussub0}\begin{split} v_{ex,sub}(t,r)-v_{ex,sub,0}(t,r)&=\int_{t}^{\infty} ds \int_{|s-t-r|}^{s-t+r} dy \left(\frac{-y}{2r}\right) \partial_{s}^{2}(v_{ex,ell}(s,y)-v_{ex,ell,0}(s,y))\\
&=\int_{0}^{r} dy \frac{y}{r}\left(-\frac{y}{2}\right) F(t,y) + \int_{r}^{2r} dy \left(2-\frac{y}{r}\right)\left(\frac{-y}{2}\right) F(t,y)\\
&-\frac{1}{2}\int_{0}^{1}dz \int_{r(1-z)}^{r(1+z)} dy y (F(t+r z,y)-F(t,y))\\
&-\frac{1}{2r}\int_{t+r}^{\infty} ds \left(\int_{s-t-r}^{s-t+r} dy y F(s,y) -2r(s-t) F(s,s-t)\right)\\
&-\int_{t}^{\infty} ds(s-t) F(s,s-t) + \int_{t}^{t+r} ds (s-t) F(s,s-t)\end{split}\end{equation}
Direct estimation gives, in the region $ \lambda(t) \leq r \leq t$,
$$|v_{ex,sub}(t,r)-v_{ex,sub,0}(t,r)+\int_{t}^{\infty} ds(s-t) F(s,s-t)| \leq \frac{C r^{2} \lambda(t)^{5/2}\log(t)}{t^{5}} + \frac{C \lambda(t)^{7/2} \log^{2}(t)}{t^{4}}$$
We now consider the derivatives of $v_{ex,sub}(t,r)-v_{ex,sub,0}(t,r)+\int_{t}^{\infty} ds(s-t) F(s,s-t)$. For the first $r$ derivative, we directly differentiate \eqref{vexsubminussub0}. To estimate the first time derivative, we start with
$$v_{ex,sub}(t,r) - v_{ex,sub,0}(t,r) = \int_{0}^{\infty}dw \int_{|w-r|}^{w+r} dy \left(\frac{-y}{2r}\right) F(t+w,y)$$
and note that Lemma \ref{vexellminusell0estlemma} justifies the differentiation under the integral sign in $t$, leading to
$$\partial_{t}(v_{ex,sub}-v_{ex,sub,0}) = \int_{t}^{\infty} ds \int_{|s-t-r|}^{s-t+r} dy \left(\frac{-y}{2r}\right) \partial_{1}F(s,y)$$
This expression, as well as its $r$ and $t$ derivatives are then directly estimated. Finally, we can infer an estimate on  $\partial_{r}^{2}\left(v_{ex,sub}-v_{ex,sub,0}\right)$ from our work up to this point, as soon as we justify the statement that
\begin{equation}\label{vexsubminussub0eqn}\left(-\partial_{t}^{2}+\partial_{r}^{2}+\frac{2}{r}\partial_{r}\right)\left(v_{ex,sub}-v_{ex,sub,0}\right) = \partial_{t}^{2}(v_{ex,ell}-v_{ex,ell,0})\end{equation}
From \eqref{vexsubeqn}, it suffices to study $v_{ex,sub,0}$. We note that 
\begin{equation}\begin{split} v_{ex,sub,0}(t,r) &= \frac{-1}{2} \int_{0}^{1} dy (g_{3}(t+r y) \cdot 2 r y + g_{2}''(t+r y)\begin{aligned}[t]&\left(r(1+y)(-2+\log(\frac{4 r^{2}(1+y)^{2}}{3 \lambda(t+r y)^{2}}))\right.\\
&+\left.r(1-y)(2+\log(\frac{3 \lambda(t+r y)^{2}}{4 r^{2}(1-y)^{2}}))\right)\end{aligned})\\
&-\frac{1}{2}\int_{1}^{\infty} dy (g_{3}(t+r y)\cdot 2 r + g_{3}''(t+r y)\begin{aligned}[t]&\left(r(1+y)(-2+\log(\frac{4 r^{2}(1+y)^{2}}{3 \lambda(t+r y)^{2}}))\right.\\
&+\left. r(y-1)(2+\log(\frac{3 \lambda(t+r y)^{2}}{4 r (1-y)^{2}}))\right)\end{aligned})\end{split}\end{equation} 
Now, we can directly differentiate under the integral sign in $(t,r)$, given the symbol type estimates and smoothness of $\lambda$, and then integrate by parts, to show that $v_{ex,sub,0}$ solves
$$\left(-\partial_{t}^{2}+\partial_{r}^{2}+\frac{2}{r}\partial_{r}\right)\left(v_{ex,sub,0}\right) = \partial_{t}^{2}(v_{ex,ell,0})$$
which establishes \eqref{vexsubminussub0eqn}. Finally, a direct estimation of $\partial_{t}\left(\int_{t}^{\infty} ds (s-t) \partial_{s}^{2}(v_{ex,ell}(s,y)-v_{ex,ell,0}(s,y))\Bigr|_{y=s-t}\right)$ gives: for $0 \leq k \leq 2$ and $j=0$ or $0 \leq k \leq 1$ and $j=1$,
\begin{equation}\begin{split}&|\partial_{t}^{j}\partial_{r}^{k} \left(v_{ex,sub}(t,r) - v_{ex,sub,0}(t,r) + \int_{t}^{\infty} ds (s-t) \partial_{s}^{2}(v_{ex,ell}(s,y)-v_{ex,ell,0}(s,y))\Bigr|_{y=s-t}\right)|\\
&\leq \begin{cases} \frac{C r^{2}\lambda(t)^{5/2} \log(t)}{t^{5}} + \frac{C \lambda(t)^{7/2} \log^{2}(t)}{t^{4}}, \quad j=0, k=0\\
\frac{C \lambda(t)^{5/2} \log^{2}(t)}{t^{4}}, \quad j=1,k=0\\
\frac{C r \lambda(t)^{5/2} \log(t)}{t^{5}} + \frac{C \lambda(t)^{7/2} \log^{2}(t)}{r t^{4}}, \quad j=0,k=1\\
\frac{C \lambda(t)^{5/2} \log^{2}(t)}{r t^{4}}, \quad j+k=2\end{cases}\end{split}\end{equation}
which finishes the proof of the lemma.
\end{proof}
Finally, we consider the free wave $v_{2}$ solving
\begin{equation}\label{v20def}\begin{cases}-\partial_{t}^{2}v_{2}+\partial_{r}^{2}v_{2}+\frac{2}{r}\partial_{r}v_{2}=0\\
v_{2}(0,r)=0, \quad \partial_{t}v_{2}(0,r):=v_{2,0}(r)=\frac{\psi(r)}{r}\left(\frac{-15}{8} \pi \sqrt{\lambda(r)} \lambda''(r)\right)\end{cases}\end{equation}
(This specific choice of $v_{2,0}$ is made so that $w_{1}+v_{2}$ matches $u_{ell}$ to an appropriate order, as will be shown in Section \ref{firstordermatching}, see also the discussion following \eqref{uwintrofirsteqn}, and the comment after \eqref{matchingpart1comp}). We have
$$v_{2}(t,r) =  \frac{t}{2}\int_{0}^{\pi} \sin(\theta) v_{2,0}(\sqrt{r^{2}+t^{2}+2 r t \cos(\theta)})d\theta$$
We define $v_{2,main}$ to be the first three terms in a small $r$ expansion of $v_{2}$: (recall that $t \geq T_{0} > 2 T_{\lambda}$, so $\psi(t) =1$)
\begin{equation}\label{v2maindef}\begin{split} v_{2,main}(t,r) &=t v_{2,0}(t)+ \frac{1}{6} r^2 \left(2 v_{2,0}'(t)+t v_{2,0}''(t)\right)+\frac{1}{120} r^4 \left(4 v_{2,0}'''(t)+tv_{2,0}''''(t)\right)\end{split}\end{equation}
It will also be convenient to consider various pieces of $v_{2,main}$:
\begin{equation}\label{v2qmdef}v_{2,qm}(t,r) = -\frac{15}{8} \left(\pi  \sqrt{\lambda(t)} \lambda''(t)\right)+\frac{1}{6} r^2 \left(2v_{2,0}'(t)+t v_{2,0}''(t)\right):=v_{2,lm}(t,r) + \frac{1}{6} r^2 \left(2v_{2,0}'(t)+t v_{2,0}''(t)\right)\end{equation}
We also note that
\begin{equation}\label{v20ests}|v_{2,0}^{(k)}(r)| \leq \frac{C_{k} \mathbbm{1}_{\{r \geq T_{\lambda}\}} \lambda(r)^{3/2}}{r^{3+k}}\end{equation}
Using the identical procedures as in Lemmas \ref{w1minusmainest}, \ref{w1estlemma}, we get:
\begin{lemma}\label{v2minusmainest} For $0 \leq j, k \leq 2$, or $j=3,k=0$
$$|\partial_{r}^{k}\partial_{t}^{j}\left(v_{2}-v_{2,main}\right)(t,r)| \leq \frac{C r^{6-k} \lambda(t)^{3/2}}{t^{8+j}}, \quad r \leq \frac{t}{2}$$
For $j \geq 0$,
$$|\partial_{t}^{j}v_{2,lm}(t,r)| \leq \frac{C_{j}\lambda(t)^{3/2}}{t^{2+j}}, \quad |\partial_{t}^{j}\left(v_{2,qm}-v_{2,lm}\right)| \leq \frac{C_{j} r^{2} \lambda(t)^{3/2}}{t^{4+j}}$$
$$|\partial_{t}^{j}(v_{2}-v_{2,qm})| \leq \begin{cases}\frac{C_{j}r^{4}\lambda(t)^{3/2}}{t^{6+j}}, \quad r \leq \frac{t}{2}\\
\frac{C_{j}\sup_{x \in [T_{\lambda},t+r]} \lambda(x)^{3/2} \log(t+r)}{r \langle t-r \rangle^{j+1}} + \frac{C_{j} r^{2} \lambda(t)^{3/2}}{t^{4+j}}, \quad r \geq \frac{t}{2}\end{cases}$$
\end{lemma}
\begin{lemma}\label{v2estlemma}For all $j, k \geq 0$, there exists $C_{j,k} >0$, such that
$$|\partial_{r}^{k}\partial_{t}^{j}v_{2}(t,r)| \leq \frac{C_{j,k} \lambda(t)^{3/2}}{t^{2+j+k}}, \quad r \leq \frac{t}{2}$$
$$|\partial_{r}^{k}\partial_{t}^{j}v_{2}(t,r)| \leq \frac{C_{j,k} \left(\sup_{x \in [T_{\lambda},t+r]}\lambda(x)^{3/2}\right) \log(t+r)}{r \langle t -r \rangle^{k+j+1}}, \quad r \geq \frac{t}{2}$$
$$|(\partial_{t}+\partial_{r})v_{2}| \leq \frac{C\left(\sup_{x \in [T_{\lambda},t+r]}\lambda(x)^{3/2}\right)\log(r+t)}{\langle t-r \rangle t r}, \quad r \geq \frac{t}{2}$$
\end{lemma}
Letting $v_{1}=w_{1}+v_{ex}+v_{2}$, we thus have a particular solution to 
$$-\partial_{t}^{2}v_{1}+\partial_{r}^{2}v_{1}+\frac{2}{r}\partial_{r}v_{1}=\partial_{t}^{2}Q_{\lambda(t)}(r)$$
\subsection{Second large $r$ correction}
Now, we need to correct the linear error term from $v_{1}$, which is (recall \eqref{Vdef})
$$V(t,r) (w_{1}+v_{2}):=RHS_{2}(t,r)$$ 
From Lemmas \ref{w1estlemma} and \ref{v2estlemma}, we get
\begin{equation}\label{dtjrhs2est}\begin{split}&|\partial_{t}^{j} RHS_{2}(t,r)| \\
&\leq \begin{cases} \frac{C_{j}}{t^{j+2} \sqrt{\lambda(t)}(\frac{r^{2}}{\lambda(t)^{2}}+1)^{3/2}}, \quad r \leq \frac{t}{2}\\
\frac{C_{j} \lambda(t)^{2}\log(r)}{r^{5}}\left(\frac{1}{\langle t-r \rangle^{j}}+\frac{1}{t^{j}}\right) \left(\sup_{x \in [T_{\lambda},t+r]}\sqrt{\lambda(x)} + \left(\sup_{x \in [T_{\lambda},t+r]}\lambda(x)^{3/2}\right)\left(\frac{1}{\langle t-r \rangle} + \frac{1}{t}\right)\right), \quad r \geq \frac{t}{2}\end{cases}\end{split}\end{equation}
We consider the particular solution to 
\begin{equation}\label{uw2eqn}-\partial_{t}^{2}u_{w,2}+\partial_{r}^{2}u_{w,2}+\frac{2}{r}\partial_{r}u_{w,2} = RHS_{2}(t,r)\end{equation}
given by
$$u_{w,2}(t,r) = \int_{t}^{\infty} ds u_{w,2,s}(t,r)$$
where $u_{w,2,s}$ solves
$$\begin{cases}-\partial_{t}^{2}u_{w,2,s}+\partial_{r}^{2}u_{w,2,s}+\frac{2}{r}\partial_{r}u_{w,2,s}=0\\
u_{w,2,s}(s,r)=0\\
\partial_{t}u_{w,2,s}(s,r) = RHS_{2}(s,r)\end{cases}$$
Therefore,
\begin{equation}\label{uw2withell}\begin{split}u_{w,2}(t,r)& = \int_{t}^{\infty} ds\left(\frac{-1}{2r}\int_{|r-(s-t)|}^{r+s-t}  y RHS_{2}(s,y) dy\right)\end{split}\end{equation}

\begin{lemma}\label{uw2estlemma}For $r \geq \lambda(t)$,
\begin{equation}\label{uw2ptwse}|u_{w,2}(t,r)| \leq \frac{C \sup_{x \in [T_{\lambda},t+r]}\left(\lambda(x)^{5/2}\right) \log^{2}(t+r)}{r t^{2}}\end{equation}
\begin{equation}\label{druw2ptwse}|\partial_{r}u_{w,2}(t,r)| \leq \frac{C \sup_{x \in [T_{\lambda},t+r]}\left(\lambda(x)^{5/2}\right)\log^{2}(t+r)}{r t^{2}}\left(\frac{1}{r}+\frac{1}{t}\right) + \frac{C \sup_{x \in [T_{\lambda},t+r]}\left(\lambda(x)^{3/2}\right) \log^{2}(t+r)}{r t^{2}}\end{equation}
\begin{equation}\label{uw2enest}||\partial_{t}u_{w,2}(t,r)||_{L^{2}(r^{2} dr)} + ||\partial_{r}u_{w,2}(t,r)||_{L^{2}(r^{2} dr)} \leq \frac{C \lambda(t)}{t}\end{equation}

\end{lemma}
\begin{proof}
A direct estimation using \eqref{dtjrhs2est} gives \eqref{uw2ptwse} and \eqref{druw2ptwse}. The energy estimate, \eqref{uw2enest} is proven with the same procedure used to establish \eqref{vexenest}. The only detail to note is that the following rough estimate
\begin{equation}\label{uw2roughptwse}|u_{w,2}(t,r)| \leq \frac{C \sqrt{\lambda(t)}}{t}, \quad r >0\end{equation}
 which results from directly inserting \eqref{dtjrhs2est} into \eqref{uw2withell}, verifies that 
 $$r u_{w,2}(t,r)^{2} \rightarrow 0, \quad r \rightarrow 0$$ 
 \end{proof}
The leading part of $u_{w,2}$ in the matching region is:
\begin{equation}\label{uw2elldef}u_{w,2,ell}(t,r) = \frac{-1}{r}\int_{0}^{r} dy y^{2} RHS_{2}(t,y) - \int_{r}^{\infty} dy y RHS_{2}(t,y)\end{equation}
which solves
\begin{equation}\label{uw2elleqn}\partial_{r}^{2}u_{w,2,ell}+\frac{2}{r}\partial_{r}u_{w,2,ell}=RHS_{2}(t,r)\end{equation}
\begin{lemma} We have the following estimates on $u_{w,2,ell}$. For $k \geq 0$,
\begin{equation}\begin{split}&|\partial_{t}^{k} u_{w,2,ell}(t,r)| \leq \begin{cases} \frac{C_{k}  \lambda(t)^{3/2}}{t^{2+k}}, \quad r \leq \lambda(t)\\
\frac{C_{k} \log^{2}(t) (1+\log(\frac{r}{\lambda(t)})) \lambda(t)^{2} \sup_{x \in [T_{\lambda},t]}\sqrt{\lambda(x)}}{r t^{2+k}}, \quad \lambda(t) \leq r \leq \frac{t}{2}\end{cases}\end{split}\end{equation}
If $r \geq \frac{t}{2}$, 
\begin{equation}\begin{split}&|\partial_{t}^{k} u_{w,2,ell}(t,r)| \\
&\leq\frac{C_{k}\lambda(t)^{2}\log^{2}(r)}{r t^{2}}\\
&\cdot \begin{cases} \left(\frac{\sup_{x \in [T_{\lambda},t+r]}\sqrt{\lambda(x)}}{t^{k}} + \frac{\sup_{x \in [T_{\lambda},t+r]}\lambda(x)^{3/2}}{t^{1+k}}\right), k=0,1\\
\left(\frac{1}{t \langle t-r \rangle^{k-1}}+\frac{1}{t^{k}}\right)\left(\sup_{x \in [T_{\lambda},t+r]}\sqrt{\lambda(x)} + \sup_{x \in [T_{\lambda},t+r]}\left(\lambda(x)^{3/2}\right) \left(\frac{1}{\langle t-r \rangle}+\frac{1}{t}\right)\right), \quad k \geq 2\end{cases}\end{split}\end{equation}\end{lemma}
\begin{proof}
If $r \leq \frac{t}{2}$, 
\begin{equation}\begin{split} \partial_{t}^{k}u_{w,2,ell}(t,r)&= \frac{-1}{r} \int_{0}^{r} dy y^{2}\partial_{t}^{k}RHS_{2}(t,y) - \int_{r}^{\frac{t}{2}} dy y \partial_{t}^{k}RHS_{2}(t,y) - \int_{\frac{t}{2}}^{\infty} dy y \partial_{t}^{k}\left(V(t,y) w_{1,0}(t,y)\right)\\
& - \int_{\frac{t}{2}}^{\infty} dy y \partial_{t}^{k}(V(t,y) (v_{2}(t,y)+\widetilde{v_{2}}(t,y)))\end{split}\end{equation}
We directly estimate all integrals except for the last one, using Lemmas \ref{w1estlemma}, \ref{v2estlemma}. Since the estimates on $t$ derivatives of $v_{2}$ and $\widetilde{v_{2}}$ in Lemmas \ref{v2estlemma} and \ref{w1estlemma}, respectively, are not a power of $t$ better than the corresponding estimates on $v_{2}(t,r)$ and $\widetilde{v_{2}}(t,r)$, in the region $r \geq \frac{t}{2}$, we treat the last integral term as follows.
We have
$$\partial_{t}^{k}\left(V(t,y) (v_{2}+\widetilde{v_{2}})\right) = \sum_{j=0}^{k} {k \choose j} \partial_{t}^{k-j}(V(t,y)) \partial_{t}^{j}(v_{2}+\widetilde{v_{2}})$$
Let $H(f) = \partial_{y}^{2}f+\frac{2}{y}\partial_{y}f$. If $j \geq 1$, we write (for both $v_{2}$ and $\widetilde{v_{2}}$)
\begin{equation}\label{dtjfreewave}\begin{cases} \partial_{t}^{j} v_{2}(t,y) = H^{\frac{j}{2}}v_{2}(t,y), \quad j \text{ is even }\\
\partial_{t}^{j}v_{2}(t,y) = H^{\frac{j-1}{2}}(\partial_{t}v_{2})(t,y), \quad j \text{ is odd }\end{cases}\end{equation}
Then, we integrate by parts in $y$, integrating the terms in \eqref{dtjfreewave}, using the following observation. If $f$ and $g$ are smooth functions of $y \in (0,\infty)$, $0<a<b$, and $n \geq 1$, then,
\begin{equation} \begin{split}\int_{a}^{b} dy y^{2}g(y) H^{n}(f)(y) &= \sum_{j=1}^{n}\begin{aligned}[t]&\left(-a^{2} H^{j-1}(g)(a) \partial_{y}H^{n-j} f(a) + b^{2}H^{j-1}(g)(b) \partial_{y}H^{n-j}f(b)\right.\\
&\left. + a^{2}\partial_{y}(H^{j-1}g)(a)H^{n-j}f(a) - b^{2}\partial_{y}(H^{j-1}g)(b) H^{n-j}f(b)\right)\end{aligned} \\
&+ \int_{a}^{b} f H^{n} (g) y^{2} dy\end{split}\end{equation}
In the case when $j$ is odd in \eqref{dtjfreewave}, this process results in an integral involving $\partial_{t}v_{2}$, which we further re-write as
$\partial_{t}v_{2}(t,y) = \left(\partial_{t}+\partial_{y}\right)v_{2}(t,y) - \partial_{y}v_{2}(t,y)$, and integrate by parts in $y$ for the $\partial_{y}v_{2}$ term. (The point is that $(\partial_{t}+\partial_{y})v_{2}(t,y)$ has better estimates near $y =t$ than just $\partial_{t}v_{2}(t,y)$, as per Lemmas \ref{v2estlemma} and \ref{w1estlemma}). In the region $r \geq \frac{t}{2}$, we start with
\begin{equation}\begin{split}\partial_{t}^{k} u_{w,2,ell}(t,r) &= \frac{-1}{r}\int_{0}^{\frac{t}{2}} dy y^{2} \partial_{t}^{k}RHS_{2}(t,y) - \frac{1}{r}\int_{\frac{t}{2}}^{r} dy y^{2}\partial_{t}^{k}(V(t,y)w_{1,0}) - \frac{1}{r}\int_{\frac{t}{2}}^{r} dy y^{2} \partial_{t}^{k}(V(t,y) (v_{2}+\widetilde{v_{2}}))\\
&-\int_{r}^{\infty} dy y \partial_{t}^{k}(V(t,y)w_{1,0}) - \int_{r}^{\infty} dy y \partial_{t}^{k}(V(t,y) (v_{2}+\widetilde{v_{2}}))\end{split}\end{equation}
Then, we use a similar procedure as above, integrating by parts for the sum of the third and fifth integrals above. Here, the only extra detail to note is that the boundary term at $y=r$ from one integration by parts for the third integral cancels the boundary term at $y=r$ from one integration by parts for the fifth integral. This allows for slightly better estimates than otherwise, and finishes the proof of the lemma.
\end{proof}
We define $u_{w,2,sub}(t,r) = u_{w,2}(t,r)-u_{w,2,ell}(t,r)$. Given the equations that $u_{w,2}$ and $u_{w,2,ell}$ solve, namely \eqref{uw2eqn} and \eqref{uw2elleqn}, $u_{w,2,sub}$ solves
$$-\partial_{t}^{2}u+\partial_{r}^{2}u+\frac{2}{r}\partial_{r}u =\partial_{t}^{2}u_{w,2,ell}$$ As in Lemma \ref{vexsublemma}, we have
\begin{lemma} \label{uw2sublemma}
$$u_{w,2,sub}(t,r) = \int_{0}^{\infty} dx \int_{|r-x|}^{r+x} ds \left(-\frac{x}{2r}\right) \partial_{1}^{2}u_{w,2,ell}(t+s,x)$$\end{lemma}
Let 
$$RHS_{2,0,0}(t,r) = V(t,r)\left(w_{1,lm}+v_{2,lm}\right), \quad RHS_{2,0}=V(t,r)\left(w_{1,cm}+v_{2,qm}\right)$$
Then, the part of $u_{w,2,ell}$ which will be important for the matching procedure is $u_{w,2,ell,0}$, which is defined by
\begin{equation}\label{uw2ell0def}u_{w,2,ell,0}(t,r) = \frac{-1}{r}\int_{0}^{r} dy y^{2} RHS_{2,0}(t,y) - \int_{r}^{\infty} dy y RHS_{2,0,0}(t,y)+\int_{0}^{r} dy y \left(RHS_{2,0}(t,y)-RHS_{2,0,0}(t,y)\right) \end{equation}
The following lemma gives the precise sense in which $u_{w,2,ell,0}$ is the leading part of $u_{w,2,ell}$.
\begin{lemma}\label{uw2ellminusell0estlemma}For $0 \leq j \leq 3$, $0 \leq k \leq 2$,  and in the region $\lambda(t) \leq r \leq \frac{t}{2}$,
$$|\partial_{t}^{j}\partial_{r}^{k}\left(u_{w,2,ell}-u_{w,2,ell,0}+\int_{0}^{\infty}  x V(t,x) (w_{1}+v_{2}-w_{1,lm}-v_{2,lm})(t,x)dx\right)| \leq \frac{C r^{3-k}\lambda(t)^{5/2}}{t^{6+j}}$$
Also, for $0 \leq j \leq 3$ and $r \leq t$, 
$$|\partial_{t}^{j}\left(u_{w,2,ell}-u_{w,2,ell,0}\right)(t,r)| \leq \frac{C \lambda(t)^{2}\sup_{x\in [T_{\lambda},t]}\left(\sqrt{\lambda(x)}\right) \log^{2}(t)}{t^{3}}\left( \frac{1}{t^{j}}+ \frac{1}{\langle t-r\rangle^{j}}\right)$$
For $k=1,2$,
$$|\partial_{t}^{2}\partial_{r}^{k}\left(u_{w,2,ell}-u_{w,2,ell,0}\right)(t,r)| \leq \begin{cases}\frac{C r^{6-k}}{\sqrt{\lambda(t)}t^{8}}, \quad r \leq \lambda(t)\\
\frac{C \lambda(t)^{2} \sup_{x \in [T_{\lambda},t]}\sqrt{\lambda(x)} \log(t)}{t^{5} \langle t-r \rangle^{k}} + \frac{C \lambda(t)^{2} \sup_{x \in [T_{\lambda},t]}\lambda(x)^{3/2} \log^{2}(t)}{t^{5}\langle t-r \rangle^{1+k}}, \quad \frac{t}{2} \leq r \leq t\end{cases}$$
Finally, for $0 \leq j \leq 3$,
$$|\partial_{t}^{j}\left(-\int_{0}^{\infty} dx x V(t,x) (w_{1}+v_{2}-w_{1,lm}-v_{2,lm})\right)| \leq \frac{C \lambda(t)^{2}\sup_{x \in [T_{\lambda},t]} \sqrt{\lambda(x)} \log(t)}{t^{3+j}}$$
and for $0 \leq j \leq 5$, and $s \geq t$,
\begin{equation}\label{uw2ellminusell0minusint}\begin{split}&|\partial_{s}^{j}\left(u_{w,2,ell}(s,y)-u_{w,2,ell,0}(s,y)+\int_{0}^{\infty} dx x V(s,x)(w_{1}+v_{2}-w_{1,lm}-v_{2,lm})(s,x)\right)\Bigr|_{y=s-t}|\\
&\leq \begin{cases} \frac{C_{j}(s-t)^{6}}{\sqrt{\lambda(s)} s^{6+j}}, \quad s-t \leq \lambda(s)\\
\frac{C_{j} (s-t)^{3}\lambda(s)^{5/2}}{s^{6+j}}, \quad \lambda(s) \leq s-t \leq \frac{s}{2}\\
\frac{C_{j}\log^{2}(s)\lambda(s)^{2} \sup_{x \in [T_{\lambda},s]}\sqrt{\lambda(x)}}{s^{4} t^{j-1}} + \frac{C_{j}\log^{2}(s)\lambda(s)^{2}\sup_{x \in [T_{\lambda},s]} \lambda(x)^{3/2}}{s^{4}t ^{j}}, \quad \frac{s}{2} \leq s-t \leq s\end{cases}\end{split}\end{equation}
\end{lemma}
\begin{proof} From \eqref{uw2elldef}, \eqref{uw2ell0def}, we get
\begin{equation}\begin{split}u_{w,2,ell}(t,r)-u_{w,2,ell,0}(t,r) &= \frac{-1}{r} \int_{0}^{r} dy y^{2} V(t,y)\left(w_{1}+v_{2}-w_{1,cm}-v_{2,qm}\right)\\
&+\int_{0}^{r} dy y V(t,y) (w_{1}-w_{1,cm}+v_{2}-v_{2,qm})-\int_{0}^{\infty} dy y V(t,y) (w_{1}+v_{2}-w_{1,lm}-v_{2,lm})\end{split}\end{equation}
We start with the last term in the above expression. From integration by parts,
\begin{equation}\begin{split}-\int_{0}^{\infty} dy y V(t,y) w_{1,0}(t,y) &= \frac{15 \sqrt{3} \lambda'(t)}{4 \sqrt{\lambda(t)}} - \frac{45 \pi}{16} \left(-\lambda''(t)\sqrt{\lambda(t)}+\frac{\lambda'(t)^{2}}{2 \sqrt{\lambda(t)}}\right)\\
&+\frac{5 \sqrt{3}}{8 \lambda(t)} \int_{0}^{\infty} \left(\sqrt{3}\arctan(\frac{\sqrt{3}\lambda(t)}{x})(x^{2}+3\lambda(t)^{2}) - 3 x \lambda(t)\right) \partial_{x}^{2}\left(\frac{\lambda'(t+x)}{\sqrt{\lambda(t+x)}}\right) dx\end{split}\end{equation}
Directly evaluating the integrals gives the following equations:
\begin{equation}\int_{0}^{\infty} dy y V(t,y) w_{1,lm}(t,y) = \frac{45 \pi ((\lambda'(t))^{2}-2 \lambda(t)\lambda''(t))}{32 \sqrt{\lambda(t)}}, \quad \int_{0}^{\infty} dy y V(t,y)v_{2,lm}(t,y) = \frac{225 \pi \sqrt{\lambda(t)} \lambda''(t)}{16}\end{equation}
\begin{equation}\begin{split} -\int_{0}^{\infty} dy y V(t,y) (\widetilde{v_{2}}+v_{2})(t,y) &= \frac{-15 \sqrt{3}}{4}\frac{\lambda'(t)}{\sqrt{\lambda(t)}} -\frac{225 \pi \sqrt{\lambda(t)}\lambda''(t)}{16}\\
& - \frac{45 \sqrt{3}}{4}\lambda(t)^{2} \int_{0}^{\infty} w dw \int_{|w-t|}^{w+t} \frac{dy}{(y^{2}+3\lambda(t)^{2})^{2}} \left(h(w)-h(t)\right)\end{split}\end{equation}
where
$$h(x) = \frac{\psi(x)}{x}\left(\frac{\lambda'(x)}{\sqrt{\lambda(x)}}+\frac{5 \sqrt{3}\pi}{4} \lambda''(x) \sqrt{\lambda(x)}\right)$$
Therefore,
\begin{equation}\label{termcuw2ellminusell0}\begin{split} &-\int_{0}^{\infty} dy y V(t,y) (w_{1}+v_{2}-w_{1,lm}-v_{2,lm})\\
&=\frac{5 \sqrt{3}}{8 \lambda(t)} \int_{0}^{\infty} \left(\sqrt{3}\arctan(\frac{\sqrt{3}\lambda(t)}{x})(x^{2}+3\lambda(t)^{2}) - 3 x \lambda(t)\right) \partial_{x}^{2}\left(\frac{\lambda'(t+x)}{\sqrt{\lambda(t+x)}}\right) dx\\
&- \frac{45 \sqrt{3}}{4}\lambda(t)^{2} \int_{0}^{\infty} w dw \int_{|w-t|}^{w+t} \frac{\left(h(w)-h(t)\right) dy}{(y^{2}+3\lambda(t)^{2})^{2}} \end{split}\end{equation}
We directly estimate the first term on the right-hand side of \eqref{termcuw2ellminusell0} (and its derivatives). On the other hand, since $v_{2}, \widetilde{v_{2}}$ are not symbols globally, we decompose the second term as 
\begin{equation}\begin{split}&\int_{0}^{\infty} w dw \int_{|t-w|}^{t+w} \frac{(h(w)-h(t))dy}{(y^{2}+3 \lambda(t)^{2})^{2}}\\
&=\int_{-\frac{t}{2}}^{t}(t+z) dz \int_{|z|}^{\infty} \frac{dy(h(t+z)-h(t))}{(y^{2}+3\lambda(t)^{2})^{2}} + \int_{2t}^{\infty} w dw \int_{w-t}^{\infty} \frac{dy(h(w)-h(t))}{(y^{2}+3\lambda(t)^{2})^{2}}\\
&+\int_{0}^{\frac{t}{2}} w dw \int_{t-w}^{t+w} \frac{dy(h(w)-h(t))}{(y^{2}+3\lambda(t)^{2})^{2}} - \int_{\frac{t}{2}}^{\infty} w dw \int_{w+t}^{\infty} \frac{dy(h(w)-h(t))}{(y^{2}+3\lambda(t)^{2})^{2}}\end{split}\end{equation}
Then, we differentiate in $t$, and directly estimate. This procedure does lead to a large number of terms to estimate, but, each one can be estimated directly with an elementary argument. All other estimates in the lemma statement follow from a direct computation.
 \end{proof}
Finally, it will be convenient to consider the following leading part of $u_{w,2,ell,0}$:
\begin{equation}\label{uw2ell00def}\begin{split}u_{w,2,ell,0,0}(t,r) &=\frac{675 \pi  \lambda(t)^{5/2} \lambda''(t)}{16 r^2}\\
&+\frac{45 \sqrt{3} \sqrt{\lambda(t)} \left(\lambda'(t)^2 \left(4 \log \left(\frac{\sqrt{3} \lambda(t)}{r}\right)-2\right)-\lambda(t) \lambda''(t) \left(8 \log \left(\frac{\sqrt{3} \lambda(t)}{r}\right)+5 \pi ^2-4\right)\right)}{32 r}\end{split}\end{equation}

We will also need the leading part of $u_{w,2,ell,0}-u_{w,2,ell,0,0}$ in the matching region.
$$u_{w,2,ell,0,1}(t,r):= \frac{-45}{2} a_{1}(t) \lambda(t)^{2} r +\frac{45}{4} \lambda(t)^{2}(6 b_{1}(t)+3 \sqrt{3} \pi a_{1}(t)\lambda(t)+b_{1}(t) \log(\frac{9 \lambda(t)^{4}}{r^{4}}))$$ 
where
\begin{equation}\label{a1b1def}a_{1}(t) = \frac{1}{4!} \partial_{t}^{3}\left(\frac{\sqrt{3}\lambda'(t)}{2\sqrt{\lambda(t)}}\right), \quad b_{1}(t) = \frac{1}{6}\left(2 v_{2,0}'(t)+t v_{2,0}''(t)\right)\end{equation}
Now, we prove the sense in which $u_{w,2,ell,0,0}, u_{w,2,ell,0,1}$ are the leading and subleading parts of $u_{w,2,ell,0}$ in the matching region.
\begin{lemma}\label{uw2ell00estlemma1} For $0 \leq j,k \leq 2$, $j=3, k=0$, and $r \geq \lambda(t)$
$$|\partial_{t}^{j}\partial_{r}^{k}\left(u_{w,2,ell,0}-u_{w,2,ell,0,0}-u_{w,2,ell,0,1}\right)(t,r)| \leq \frac{C \lambda(t)^{9/2} (|\log(r)|+\log(t))}{r^{k+3} t^{2+j}} \left(1+\frac{r^{2}}{t^{2}}\right)$$
For $j=2,3$,
$$|\partial_{t}^{j}(u_{w,2,ell,0}(t,r)-u_{w,2,ell,00}(t,r))| \leq \begin{cases} \frac{C \lambda(t)^{7/2}}{r^{2}t^{2+j}}, \quad r \leq \lambda(t)\\
\frac{C \lambda(t)^{9/2}\log(t)}{r^{3}t^{2+j}} +\frac{C r \lambda(t)^{5/2}\log(t)}{t^{4+j}}, \quad  t \geq r \geq \lambda(t)\end{cases}$$
For $0 \leq j \leq 5$ and $k=0$, or $j=2, k=1,2$,
$$|\partial_{t}^{j}\partial_{r}^{k}\left(u_{w,2,ell,0}-\left(u_{w,2,ell,00}-\frac{675 \pi \lambda(t)^{5/2} \lambda''(t)}{16 r^{2}}\right)\right)| \leq \begin{cases} \frac{C \lambda(t)^{5/2}}{r^{1+k} t^{2+j}} (1+|\log(\frac{r}{\lambda(t)})|), \quad r \leq \lambda(t)\\
\frac{C \lambda(t)^{7/2}}{t^{2+j} r^{2+k}} + \frac{C \lambda(t)^{5/2} r^{1-k}\log(t)}{t^{4+j}}, \quad \lambda(t) \leq r \leq t\end{cases}$$
\end{lemma}
\begin{proof}
The lemma follows from the following direct computations
\begin{equation}\begin{split}&-\frac{1}{r}\int_{0}^{r} dy y^{2} RHS_{2,0,0}(t,y)-\int_{r}^{\infty} dy y RHS_{2,0,0}(t,y) - u_{w,2,ell,0,0}(t,r) \\
&= \frac{45}{32} \sqrt{\lambda(t)}\lambda''(t) \left(\frac{\lambda(t)\sqrt{3} (10 \pi  g_{1}(\frac{r}{\lambda(t)})+4 g_{2}(\frac{r}{\lambda(t)}))}{r}+4 g_{1}(\frac{r}{\lambda(t)})\right) -\frac{45}{32} \frac{\lambda'(t)^2}{\sqrt{\lambda(t)}} \left(2 g_{1}(\frac{r}{\lambda(t)})+\frac{2 \sqrt{3}\lambda(t) g_{2}(\frac{r}{\lambda(t)})}{r}\right) \end{split}\end{equation}
where
$$g_{1}(R)=\tan ^{-1}\left(\frac{\sqrt{3}}{R}\right)-\frac{\sqrt{3}}{R}, \quad g_{2}(R) = \log \left(\frac{3}{R^2}+1\right)$$
and
\begin{equation}\begin{split}&-\frac{1}{r}\int_{0}^{r} dy y^{2} \left(RHS_{2,0}(t,y)-RHS_{2,0,0}(t,y)\right)+\int_{0}^{r} dy y \left(RHS_{2,0}(t,y)-RHS_{2,0,0}(t,y)\right) - u_{w,2,ell,0,1}(t,r) \\
&=\frac{135 \sqrt{3} \lambda(t)^3 \left(-2 r a_{1}(t) \tan ^{-1}\left(\frac{\sqrt{3} \lambda(t)}{r}\right)-2 b_{1}(t) \tan ^{-1}\left(\frac{r}{\sqrt{3} \lambda(t)}\right)\right)}{4 r}\\
&+\frac{45 \lambda(t)^2 \left(12 a_{1}(t) \lambda(t)^2 \log \left(\frac{3 \lambda(t)^2}{3 \lambda(t)^2+r^2}\right)-2 r b_{1}(t) \log \left(\frac{3 \lambda(t)^2}{r^2}+1\right)\right)}{4 r}\end{split}\end{equation}
\end{proof}
Direct estimation also gives the following lemma.
\begin{lemma}\label{uw2ell00estlemma}For $r \geq \lambda(t)$ , $j \geq 0$, and $k=0,1$,
$$|\partial_{r}^{k}\partial_{t}^{j}u_{w,2,ell,0,0}(t,r)| \leq \frac{C_{j}\lambda(t)^{5/2}}{r^{1+k} t^{2+j}} (1+|\log(\frac{\lambda(t)}{r})|), \quad |\partial_{t}^{j}u_{w,2,ell,0,1}(t,r)| \leq \frac{C_{j}\lambda(t)^{5/2} r}{t^{4+j}}$$
$$|\partial_{t}^{n}\left(u_{w,2,ell,0}-\left(u_{w,2,ell,0,0}-\frac{675 \pi \lambda(t)^{5/2}\lambda''(t)}{16 r^{2}}\right)\right)| \leq \begin{cases} \frac{C \lambda(t)^{5/2}(1+|\log(\frac{r}{\lambda(t)})|)}{r t^{2+n}}, \quad r \leq \lambda(t)\\
\frac{C\lambda(t)^{7/2}}{r^{2}t^{2+n}} + \frac{C \lambda(t)^{5/2}}{t^{3+n}}, \quad t \geq r \geq \lambda(t)\end{cases}, \quad 0 \leq n \leq 3$$\end{lemma}
Recall that $u_{w,2,sub}$ is a solution to
$$-\partial_{t}^{2}u+\partial_{r}^{2}u+\frac{2}{r}\partial_{r}u=\partial_{t}^{2}u_{w,2,ell}(t,r)$$
The following part of $u_{w,2,sub}$ will also be important for the matching procedure:
\begin{equation}\label{uw3ell0def}u_{w,3,ell,0}(t,r):= -\frac{1}{r} \int_{\lambda(t)}^{r} dy y^{2}\partial_{t}^{2}u_{w,2,ell,0,0}(t,y) + \int_{\lambda(t)}^{r} dy y \partial_{t}^{2}u_{w,2,ell,0,0}(t,y) \end{equation}
Directly from \eqref{uw2ell00estlemma}, we get
\begin{lemma} $$|\partial_{t}^{k}u_{w,3,ell,0}(t,r)| \leq \frac{C r \lambda(t)^{5/2} (1+\log(\frac{r}{\lambda(t)}))}{t^{4+k}}, \quad r \geq \lambda(t), \quad 0 \leq k \leq 2$$\end{lemma}
$u_{w,3,ell,0}$ is the leading part of $u_{w,2,sub,0}$ in the matching region, in the following sense.
\begin{lemma}\label{uw2subminusell30estlemma} For $\lambda(t) \leq r \leq \frac{t}{4}$,
\begin{equation}\begin{split}&|\partial_{t}^{j}\partial_{r}^{k}\left(u_{w,2,sub}-u_{w,3,ell,0}\right.-\begin{aligned}[t]&\left(\int_{t}^{\infty} w_{2}(s) ds + \int_{t}^{\infty} ds q_{1}''(s) \log(\frac{s-t}{\sqrt{3}\lambda(s)})\right.\\
&\left.\left.+\int_{t}^{\infty}ds(s-t)\partial_{s}^{2}\left(u_{w,2,ell}(s,y)-(u_{w,2,ell,00}(s,y) - \frac{675 \pi \lambda(s)^{5/2}\lambda''(s)}{16 y^{2}})\right)\Bigr|_{y=s-t}\right)\right)|\end{aligned} \\
&\leq \begin{cases} \frac{C r^{2}\lambda(t)^{2} \sup_{x \in [T_{\lambda},t]}\sqrt{\lambda(x)} \log^{2}(t)}{t^{5}}+\frac{C \lambda(t)^{7/2} \log^{2}(t)}{t^{4}}, \quad j=k=0\\
\frac{C r \lambda(t)^{2}\sup_{x \in [T_{\lambda},t]}\sqrt{\lambda(x)} \log^{2}(t)}{t^{5}} + \frac{C \lambda(t)^{7/2} \log^{2}(t)}{t^{4} r}, \quad k=1,j=0\\
\frac{C\lambda(t)^{2}\sup_{x \in [T_{\lambda},t]}\sqrt{\lambda(x)} \log^{2}(t)}{t^{4}}, \quad j=1, k=0\\
\frac{C \lambda(t)^{2} \sup_{x \in [T_{\lambda},t]}\sqrt{\lambda(x)} \log^{2}(t)}{t^{4} r}, \quad j=1, k=1 \text{  or  } j=0,k=2\end{cases}\end{split}\end{equation}
where $$q_{1}(t) = \frac{-45 \sqrt{3} \lambda(t)^{2}}{2} \partial_{t}^{2}\sqrt{\lambda(t)}, \quad q_{2}(t) = \frac{-45 \sqrt{3}}{32}\left(-8 \partial_{t}^{2}\left(\sqrt{\lambda(t)}\right) \lambda(t)^{2}+5 \pi^{2}\lambda(t)^{3/2}\lambda''(t)\right)$$
and
$$w_{2}(t) = \frac{q_{1}(t) \lambda'(t)^{2}}{\lambda(t)^{2}} - q_{2}''(t) -\frac{(2q_{1}'(t)\lambda'(t)+q_{1}(t)\lambda''(t))}{\lambda(t)}$$\end{lemma}
\begin{proof} We recall the definition of $u_{w,2,ell,00}$ in \eqref{uw2ell00def}, and start by defining
$$u_{w,2,sub,0}(t,r) = \int_{t}^{\infty} ds \int_{|s-t-r|}^{s-t+r} dy \left(\frac{-y}{2r}\right) \partial_{s}^{2}\left(u_{w,2,ell,00}(s,y) -\frac{675 \pi \lambda(s)^{5/2}\lambda''(s)}{16 y^{2}}\right)$$ 
Let $q_{1},q_{2},w_{2}$ be as in the lemma statement. Then,
\begin{equation}\begin{split}u_{w,2,sub,0}(t,r) &= \int_{t}^{\infty} ds \int_{|s-t-r|}^{s-t+r} \frac{dy}{2r} \left(w_{2}(s) - \log(\frac{\sqrt{3} \lambda(s)}{y}) q_{1}''(s)\right)\\
&= \int_{t}^{\infty} \frac{ds}{2r}\begin{aligned}[t]&\left(w_{2}(s) (s-t+r-|s-t-r|)\right.\\
&\left.-q_{1}''(s)\right.\begin{aligned}[t]&\left(|s-t-r|\log(|s-t-r|)-(s-t+r)\log(s-t+r)\right.\\
&\left.\left.-(|s-t-r|-(s-t+r))(1+\log(\sqrt{3}\lambda(s)))\right)\right)\end{aligned}\end{aligned}\end{split}\end{equation}
We now write this expression in such a way as to make manifest those terms (which are roughly on the order of $\frac{r \lambda(t)^{5/2}}{t^{4}}$) which will cancel in the combination $u_{w,2,sub,0}-u_{w,3,ell,0}$. This leads to
\begin{equation}\label{uw2sub0fordiff}\begin{split}u_{w,2,sub,0}(t,r) &=\frac{r}{2}\left(-w_{2}(t) + \frac{3}{2}q_{1}''(t) + q_{1}''(t) \log(\frac{\sqrt{3}\lambda(t)}{r})\right)+\int_{t}^{\infty} w_{2}(s) ds + \int_{t}^{\infty} ds q_{1}''(s) \log(\frac{s-t}{\sqrt{3}\lambda(s)})\\
&+r \int_{0}^{1} dy \left(q_{1}''(t+r y) \log(\frac{\lambda(t+r y)}{r}) - q_{1}''(t) \log(\frac{\lambda(t)}{r})\right) +r \int_{0}^{1} dy (q_{1}''(t+r y)-q_{1}''(t))\log(\frac{\sqrt{3}}{y})\\
&-\int_{t}^{t+r} ds (w_{2}(s)-w_{2}(t)) - \frac{r}{2}(q_{1}''(t+r)-q_{1}''(t))(-1+\log(4)) \\
&-\frac{r^{2}}{2} \int_{1}^{\infty} q_{1}'''(t+r y) ((1+y^{2}) \coth^{-1}(y) + y(-1+\log(1-\frac{1}{y^{2}})))dy\\
&+\frac{r}{2}\int_{0}^{1} dy \begin{aligned}[t]&\left((w_{2}(t+r y)-w_{2}(t))2 y\right.\\
&\left.-(q_{1}''(t+r y)-q_{1}''(t))\right.\begin{aligned}[t]&((1-y)\log(r(1-y))-(1+y)\log(r(1+y))\\
&\left.+2 y(1+\log(\sqrt{3}\lambda(t+r y))))\right)\end{aligned}\end{aligned}\\
&+\frac{r}{2}\int_{0}^{1} dy (-q_{1}''(t) 2 y) (\log(\lambda(t+r y)-\log(\lambda(t))))\end{split}\end{equation}
We remark that the sixth integral term in the expression above comes from integration by parts:
\begin{equation}\begin{split}&\frac{1}{2} \int_{t+r}^{\infty} ds q_{1}''(s) \log(\frac{(s-t-r)(s-t+r)}{(s-t)^{2}}) - \frac{1}{2r} \int_{t+r}^{\infty} ds q_{1}''(s)(s-t)\left(\log(\frac{s-t-r}{s-t+r})+\frac{2r}{s-t}\right)\\
&=\frac{-r}{2} q_{1}''(t) (-1+\log(4))\\
&-\frac{r}{2}(q_{1}''(t+r)-q_{1}''(t))(-1+\log(4)) - \frac{r^{2}}{2} \int_{1}^{\infty} q_{1}'''(t+r y) ((1+y^{2})\coth^{-1}(y) + y(-1+\log(1-\frac{1}{y^{2}}))) dy\end{split}\end{equation}
We recall the definition of $u_{w,3,ell,0}$ (\eqref{uw3ell0def}), and note that
\begin{equation}\begin{split}&\frac{-1}{r} \int_{0}^{r} y^{2}\partial_{t}^{2}\left(u_{w,2,ell,00}(t,y)-\frac{675 \pi \lambda(t)^{5/2}\lambda''(t)}{16 y^{2}}\right) dy + \int_{0}^{r} y \partial_{t}^{2}\left(u_{w,2,ell,00}(t,y)-\frac{675 \pi \lambda(t)^{5/2}\lambda''(t)}{16 y^{2}}\right) dy\\
&=\frac{r}{2}(-w_{2}(t) + \frac{3}{2}q_{1}''(t) + q_{1}''(t) \log(\sqrt{3}\frac{\lambda(t)}{r}))\end{split}\end{equation}
which are precisely the first non-integral terms in \eqref{uw2sub0fordiff}. Therefore,
\begin{equation}\begin{split}&u_{w,2,sub,0}-u_{w,3,ell,0}\\
&= \frac{1}{r} \int_{\lambda(t)}^{r} dy y^{2} \frac{675 \pi}{16 y^{2}} \partial_{t}^{2}(\lambda(t)^{5/2}\lambda''(t)) - \int_{\lambda(t)}^{r} dy y \frac{675 \pi}{16 y^{2}} \partial_{t}^{2}(\lambda(t)^{5/2}\lambda''(t))\\
&-\frac{1}{r}\int_{0}^{\lambda(t)} y^{2} \partial_{t}^{2}\left(u_{w,2,ell,00}-\frac{675 \pi \lambda(t)^{5/2}\lambda''(t)}{16 y^{2}}\right) dy + \int_{0}^{\lambda(t)} y  \partial_{t}^{2}\left(u_{w,2,ell,00}-\frac{675 \pi \lambda(t)^{5/2}\lambda''(t)}{16 y^{2}}\right)dy\\
&+\int_{t}^{\infty}w_{2}(s) ds + \int_{t}^{\infty} ds q_{1}''(s) \log(\frac{s-t}{\sqrt{3}\lambda(s)}) + r \int_{0}^{1} dy \left(q_{1}''(t+r y) \log(\frac{\lambda(t+r y)}{r}) - q_{1}''(t) \log(\frac{\lambda(t)}{r})\right)\\
&+r \int_{0}^{1} dy (q_{1}''(t+r y)-q_{1}''(t)) \log(\frac{\sqrt{3}}{y}) - \int_{t}^{t+r} ds (w_{2}(s)-w_{2}(t)) - \frac{r}{2}(q_{1}''(t+r)-q_{1}''(t))(-1+\log(4))\\
&-\frac{r^{2}}{2} \int_{1}^{\infty} q_{1}'''(t+r y) ((1+y^{2})\coth^{-1}(y) + y(-1+\log(1-\frac{1}{y^{2}}))) dy\\
&+\frac{r}{2}\int_{0}^{1} dy \begin{aligned}[t]&\left((w_{2}(t+r y)-w_{2}(t))2 y\right.\\
&\left.-(q_{1}''(t+r y)-q_{1}''(t))\right.\begin{aligned}[t]&((1-y)\log(r(1-y))-(1+y)\log(r(1+y))\\
&\left.+2 y(1+\log(\sqrt{3}\lambda(t+r y))))\right)\end{aligned}\end{aligned}\\
&+\frac{r}{2}\int_{0}^{1}dy(-q_{1}''(t)) 2 y(\log(\lambda(t+r y))-\log(\lambda(t)))\end{split}\end{equation}
Each term on the right-hand side of this expression (and its derivatives) can be straightforwardly estimated now. We thus get: for $0 \leq k \leq 2$ and $j=0$ or $0 \leq k \leq 1$ and $j=1$,
\begin{equation}\begin{split}&|\partial_{t}^{j}\partial_{r}^{k} \left(u_{w,2,sub,0}-u_{w,3,ell,0}-\left(\int_{t}^{\infty} w_{2}(s) ds + \int_{t}^{\infty} ds q_{1}''(s) \log(\frac{s-t}{\sqrt{3}\lambda(s)})\right)\right)|\\
&\leq \frac{C}{t^{4+j} r^{k}} \log(t) \lambda(t)^{5/2} \left(\frac{r^{2}}{t}+\lambda(t)\right), \quad \lambda(t) \leq r \leq t\end{split}\end{equation}
The next step is to estimate $u_{w,2,sub}-u_{w,2,sub,0}$. Let
$$F(s,y) = \partial_{s}^{2}\left(u_{w,2,ell}(s,y) -(u_{w,2,ell,00}(s,y) - \frac{675 \pi \lambda(s)^{5/2}\lambda''(s)}{16 y^{2}})\right)$$ 
Using the definitions of $u_{w,2,sub}$ and $u_{w,2,sub,0}$, we get
\begin{equation}\begin{split}u_{w,2,sub}(t,r)-u_{w,2,sub,0}(t,r) &= \int_{t}^{\infty} ds \int_{|s-t-r|}^{s-t+r} dy \left(\frac{-y}{2r}\right) F(s,y)\\
&= \int_{t}^{\infty} ds \int_{|s-t-r|}^{s-t+r} dy \left(\frac{-y}{2r}\right) \partial_{s}^{2}\left(\frac{675 \pi \lambda(s)^{5/2}\lambda''(s)}{16 y^{2}}\right)\\
&+\int_{t}^{\infty} ds \int_{|s-t-r|}^{s-t+r} dy \left(\frac{-y}{2r}\right) \partial_{s}^{2}\left(u_{w,2,ell}(s,y) - u_{w,2,ell,00}(s,y)\right)\end{split}\end{equation}
Direct estimation gives, for $0 \leq k \leq 2$ and $j=0$ or $0 \leq k \leq 1$ and $j=1$, and in the region $\lambda(t) \leq r \leq \frac{t}{4}$,
$$|\partial_{t}^{j}\partial_{r}^{k}\left(\int_{t}^{\infty} ds \int_{|s-t-r|}^{s-t+r} dy \left(\frac{-y}{2r}\right) \partial_{s}^{2}\left(\frac{675 \pi \lambda(s)^{5/2}\lambda''(s)}{16 y^{2}}\right)\right)| \leq \frac{C \lambda(t)^{7/2} \log(t)}{r^{k}t^{4+j}}$$
Therefore, it suffices to estimate the following integral (with $G(s,y) =\partial_{s}^{2}\left(u_{w,2,ell}(s,y) - u_{w,2,ell,00}(s,y)\right)$) 
\begin{equation}\begin{split} &\int_{t}^{\infty} ds \int_{|s-t-r|}^{s-t+r} dy \left(\frac{-y}{2r}\right) G(s,y)\\
&=\int_{0}^{r} \frac{dy}{r} \left(r-(r-y)\right)\left(-\frac{y}{2}\right) G(t,y) + \int_{r}^{2r} \frac{dy}{r} \left(r-(y-r)\right) \left(\frac{-y}{2}\right) G(t,y)\\
&+\int_{0}^{1} r dz \int_{r-rz}^{r+rz} dy \left(\frac{-y}{2r}\right) (G(rz+t,y)-G(t,y))-\int_{t+r}^{\infty}(s-t) G(s,s-t) ds\\
&-\frac{1}{2r}\int_{t+r}^{\infty} ds \left(\int_{s-t-r}^{s-t+r} dy y G(s,y) - 2 r (s-t) G(s,s-t)\right)\end{split}\end{equation}
Each integral is directly estimated, using Lemmas \ref{uw2ellminusell0estlemma} and \ref{uw2ell00estlemma1}, except for the following. We write
\begin{equation}\label{desplit} \begin{split} \int_{t+r}^{\infty} (s-t) G(s,s-t) ds&= \int_{t+r}^{\infty} (s-t) \partial_{s}^{2}\left(\frac{-675 \pi \lambda(s)^{5/2} \lambda''(s)}{16 y^{2}}\right)\Bigr|_{y=s-t} ds\\
&+\int_{t}^{\infty} (s-t) \partial_{s}^{2}\left(u_{w,2,ell}(s,y)- (u_{w,2,ell,00}(s,y)-\frac{675 \pi \lambda(s)^{5/2} \lambda''(s)}{16 y^{2}})\right)\Bigr|_{y=s-t} ds\\
&-\int_{t}^{t+r}(s-t)\partial_{s}^{2}\left(u_{w,2,ell}(s,y)- (u_{w,2,ell,00}(s,y)-\frac{675 \pi \lambda(s)^{5/2} \lambda''(s)}{16 y^{2}})\right)\Bigr|_{y=s-t} ds\end{split}\end{equation}
The point is that the second term in \eqref{desplit} is not quite perturbative, while the others can be directly estimated. This gives rise to
\begin{equation}\begin{split}&|u_{w,2,sub}-u_{w,2,sub,0}-\int_{t}^{\infty}ds(s-t)\partial_{s}^{2}\left(u_{w,2,ell}(s,y)-(u_{w,2,ell,00}(s,y) - \frac{675 \pi \lambda(s)^{5/2}\lambda''(s)}{16 y^{2}})\right)\Bigr|_{y=s-t}|\\
&\leq \frac{C r^{2}\lambda(t)^{2}\sup_{x \in [T_{\lambda},t]}\sqrt{\lambda(x)} \log^{2}(t)}{t^{5}} + \frac{C \lambda(t)^{7/2} \log^{2}(t)}{t^{4}}, \quad \lambda(t) \leq r \leq \frac{t}{4}\end{split}\end{equation}
We estimate the derivatives of $$u_{w,2,sub}-u_{w,2,sub,0}-\int_{t}^{\infty}ds(s-t)\partial_{s}^{2}\left(u_{w,2,ell}(s,y)-(u_{w,2,ell,00}(s,y) - \frac{675 \pi \lambda(s)^{5/2}\lambda''(s)}{16 y^{2}})\right)\Bigr|_{y=s-t}$$ exactly as in the proof of Lemma \ref{vexsubminussub0estlemma}. This gives rise to the following. For $j=0$, $0 \leq k \leq 2$ or $j=1, 0 \leq k \leq 1$, and in the region $\lambda(t) \leq r \leq \frac{t}{4}$,
\begin{equation}\begin{split}&|\partial_{r}^{k}\partial_{t}^{j}\left(u_{w,2,sub}-u_{w,2,sub,0}-\int_{t}^{\infty}ds(s-t)\partial_{s}^{2}\left(u_{w,2,ell}(s,y)-(u_{w,2,ell,00}(s,y) - \frac{675 \pi \lambda(s)^{5/2}\lambda''(s)}{16 y^{2}})\right)\Bigr|_{y=s-t}\right)|\\
&\leq \begin{cases} \frac{C r \lambda(t)^{2} \sup_{x \in [T_{\lambda},t]}\sqrt{\lambda(x)} \log^{2}(t)}{t^{5}} + \frac{C \lambda(t)^{7/2} \log^{2}(t)}{t^{4} r}, \quad k=1, j=0\\
\frac{C \lambda(t)^{2} \sup_{x \in [T_{\lambda},t]}\sqrt{\lambda(x)} \log^{2}(t)}{t^{4}}, \quad j=1, k=0\\
\frac{C \lambda(t)^{2} \sup_{x \in [T_{\lambda},t]}\sqrt{\lambda(x)} \log^{2}(t)}{r t^{4}}, \quad j+k=2\end{cases}
\end{split}\end{equation}
\end{proof}
Finally, let $v_{3}$ be the solution to
\begin{equation}\label{v3def}\begin{cases}-\partial_{t}^{2}v_{3}+\partial_{r}^{2}v_{3}+\frac{2}{r}\partial_{r}v_{3}=0\\
v_{3}(0,r)=0, \quad \partial_{t}v_{3}(0,r)=v_{3,0}(r)\end{cases}\end{equation}
where $v_{3,0}$ is to be specified later. In particular, we have
$$v_{3}(t,r) =  \frac{t}{2}\int_{0}^{\pi} \sin(\theta) v_{3,0}(\sqrt{r^{2}+t^{2}+2 r t \cos(\theta)})d\theta$$ Then, $u=u_{w,2}+v_{3}$ is a particular solution to the equation
$$-\partial_{t}^{2}u+\partial_{r}^{2}u+\frac{2}{r}\partial_{r}u=RHS_{2}(t,r)$$
and $v_{3,0}$ will be chosen soon so that $u_{w,2}+v_{3}$ has a desired behavior in the matching region. The leading part of $v_{3}$ in the matching region will turn out to be
\begin{equation}\label{v3maindef}v_{3,main}(t,r) = t v_{3,0}(t)\end{equation}
\subsection{Matching, Part 1}\label{firstordermatching}
We recall the leading parts of $u_{ell}$, $w_{1}$, $v_{2}$, and $v_{ex}$ (from \eqref{uellmaindef}, \eqref{w1linmaindef}, \eqref{v2qmdef}, \eqref{vexell0def}, respectively) in the matching region.
\begin{equation}\begin{split}&u_{ell,main}(t,r) \\
&= -\frac{\sqrt{3}c_{1}(t) \lambda(t)}{2 r}+\sqrt{\lambda(t)} \lambda''(t) \left(\frac{675 \pi  \lambda(t)^2}{16 r^2}-\frac{3 \sqrt{3} \lambda(t) \left(37-20 \log \left(\frac{r^2}{3 \lambda(t)^2}\right)\right)}{8 r}+\frac{\sqrt{3} r}{4 \lambda(t)}-\frac{15 \pi }{8}\right)\\
&+\frac{\left(\frac{15 \sqrt{3} \lambda(t)}{16 r}-\frac{\sqrt{3} r}{8 \lambda(t)}\right) \lambda'(t)^2}{\sqrt{\lambda(t)}}\end{split}\end{equation}
$$w_{1,lm}(t,r) = \frac{1}{2} r \left(\frac{\sqrt{3} \lambda''(t)}{2 \sqrt{\lambda(t)}}-\frac{\sqrt{3} \lambda'(t)^2}{4 \lambda(t)^{3/2}}\right)$$
$$v_{2,lm}(t,r) =-\frac{15}{8} \pi  \sqrt{\lambda(t)} \lambda''(t)$$
\begin{equation}\begin{split} &v_{ex,ell,0}(t,r) = \frac{3 \sqrt{3} \sqrt{\lambda(t)} \left(\lambda'(t)^2 \left(15 \log \left(\frac{4 r^{2}}{3 \lambda(t)^2}\right)-17\right)+2 \lambda(t) \lambda''(t) \left(5 \log \left(\frac{4r^{2}}{3 \lambda(t)^2}\right)+1\right)\right)}{16 r}\end{split}\end{equation}
Recalling the expression for $u_{w,2,ell,0,0}$ from \eqref{uw2ell00def}, direct computation gives
\begin{equation}\label{matchingpart1comp}\begin{split} &u_{ell,main}-(u_{w,2,ell,00}+v_{2,lm}+v_{ex,ell,0}+w_{1,lm})\\
&= -\frac{\sqrt{3} \sqrt{\lambda(t)} \left(16 c_{1}(t) \sqrt{\lambda(t)}+6 (15 \log (4)-37) \lambda'(t)^2+3 \left(-75 \pi ^2+212+20 \log (4)\right) \lambda(t) \lambda''(t)\right)}{32 r}\end{split}\end{equation}
Note that $v_{2,lm}$ and $w_{1,lm}$ each contain terms of the form $f_{k}(t) r^{k}$ for $k=0,1$ which exactly cancel with terms of the same form from $u_{ell,main}$ when these functions are combined as in \eqref{matchingpart1comp}. This is (partly) due to the specific choice of $v_{2,0}$ from \eqref{v20def}. Also, $u_{ell,main},u_{w,2,ell,00}$, and $v_{ex,ell,0}$ each separately contain terms of the form $f(t) \frac{\log(r)}{r}$, but, all such terms exactly cancel in the combination in \eqref{matchingpart1comp}. This is reminiscent of a similar ``automatic'' matching of logarithmically higher order terms observed in \cite{wm2}. We choose 
$$c_{1}(t) = \frac{3 \left((74-30 \log (4)) \lambda'(t)^2+\left(75 \pi ^2-4 (53+5 \log (4))\right) \lambda(t) \lambda''(t)\right)}{16 \sqrt{\lambda(t)}}$$
so that 
\begin{equation}\label{matching1}u_{ell,main}-(u_{w,2,ell,00}+v_{2,lm}+v_{ex,ell,0}+w_{1,lm})=0\end{equation}
Now that we have fixed the function $c_{1}$, we can directly estimate $u_{ell}-u_{ell,main}$, and $u_{ell}$. Note that the last estimate in the lemma statement is not sharp in the sense that $\partial_{r}^{k}u_{ell}(t,r)$ has no singularity as $r \rightarrow 0$, but, the estimate suffices for its use later on.
\begin{lemma}\label{uellminusuellmainestlemma} For $0 \leq j, k \leq 2$, 
$$|\partial_{t}^{j}\partial_{R}^{k}\left(u_{ell}(t,R\lambda(t))-u_{ell,main}(t,R\lambda(t))\right)| \leq C\frac{\lambda(t)^{3/2}(1+\log(R))}{t^{2+j} R^{3+k}}, \quad R \geq 1$$
For $0 \leq j,k \leq 4$,
$$|\partial_{t}^{j}\partial_{R}^{k}\left(u_{ell}(t,R\lambda(t))\right)| \leq \frac{C \lambda(t)^{3/2}}{t^{2+j}} \begin{cases} 1, \quad k=0\\
R^{2-k}, \quad 1 \leq k \leq 2\\
R^{4-k}, \quad 3 \leq k \leq 4\end{cases}, \quad R \leq 1$$
\begin{equation}\label{uellestforuell3} |\partial_{t}^{j}\partial_{r}^{k}u_{ell}(t,r)| \leq \frac{C \lambda(t)^{3/2}}{t^{2+j} r^{k}} \begin{cases} 1, \quad r \leq \lambda(t)\\
\frac{r}{\lambda(t)}, \quad r \geq \lambda(t)\end{cases}, \quad 0 \leq j+k \leq 6\end{equation}

\end{lemma}
\subsection{Matching, Part 2}\label{secondordermatching}
Next, we consider matching of higher order terms. We recall the decomposition of $u_{ell,2}$ given in \eqref{uell2formatching}, and define $u_{ell,2,main}$ to be the sum of the terms arising in a large $r$ expansion of $u_{ell,2}(t,r)$ which do not decay as $r \rightarrow \infty$. In particular, we have the following. Let 
$$a(t) = \frac{1}{2}\partial_{t}\left(\frac{\sqrt{3}\lambda'(t)}{2\sqrt{\lambda(t)}}\right), \quad b(t) = \frac{-15}{8} \pi \sqrt{\lambda(t)}\lambda''(t)$$
We use \eqref{matching1}, recall the definitions of $u_{w,3,ell,0}$ and $v_{ex,ell,0}$ in \eqref{uw3ell0def} and \eqref{vexell0def}, respectively, inspect \eqref{uell2formatching}, and claim that  the sum of the terms arising in a large $r$ expansion of $u_{ell,2}(t,r)$ which do not decay as $r \rightarrow \infty$ is given by.
\begin{equation}\label{uell2maindef}\begin{split} &u_{ell,2,main}(t,r)\\
&=v_{ex,sub,ell}(t,r)+u_{w,3,ell,0}(t,r) - \frac{2}{\lambda(t) \sqrt{3}} \int_{\lambda(t)}^{\infty} x^{2}\partial_{1}^{2}(u_{w,2,ell,0,0}+v_{ex,ell,0})(t,x)\left(\phi_{0}(\frac{x}{\lambda(t)})+\frac{\sqrt{3}\lambda(t)}{2 x}\right) dx \\
&- \frac{15}{2}\lambda(t)^{2}\log(\frac{r}{\lambda(t)})b''(t)+ \frac{1}{12} a''(t) r^{3}+\frac{1}{6}b''(t) r^{2}-\frac{15}{8} \lambda(t)^{2}a''(t) r + \frac{43}{6} \lambda(t)^{3} a''(t) + \frac{53}{4}\lambda(t)^{2}b''(t)\\
&-\frac{\lambda(t)^{2}}{12} \left(-206 \lambda(t) a''(t) -3(13 + 30 \log(\frac{3}{2}))b''(t)\right)\\
&-\frac{2}{\sqrt{3}}\int_{1}^{\infty} s^{2}\lambda(t)^{2} \partial_{1}^{2}e_{ell,1}(t,s\lambda(t)) \phi_{0}(s) ds - \frac{2}{\sqrt{3}}\int_{0}^{1} s^{2}\lambda(t)^{2} \partial_{1}^{2}u_{ell}(t,s\lambda(t)) \phi_{0}(s) ds\end{split}\end{equation}
Then, $u_{ell,2,main}$ is the leading part of $u_{ell,2}$ in the matching region in the following sense.
\begin{lemma}\label{uell2minusuell2mainestlemma}We have the following estimates.\\
For $j=0$ and $0 \leq k \leq 2$ or $j=1$ and $0 \leq k \leq 1$
$$|\partial_{t}^{j}\partial_{r}^{k}\left(u_{ell,2}-u_{ell,2,main}\right)(t,r)| \leq\frac{C \lambda(t)^{9/2} (1+\log^{2}(\frac{r}{\lambda(t)}))}{r^{1+k}t^{4+j}}, \quad r \geq \lambda(t)$$
$$|\partial_{t}^{j}\partial_{r}^{k}u_{ell,2}(t,r)| \leq \frac{C r^{2-k}(r + \lambda(t))\sqrt{\lambda(t)}}{t^{4+j}}, \quad j,k=0,1$$
$$|\partial_{t}^{2+j}\partial_{r}^{k}u_{ell,2}(t,r)| \leq \frac{C r^{2-k} \lambda(t)^{3/2}}{t^{6+j}}, \quad r \leq \lambda(t), \quad 0 \leq j,k \leq 2$$
$$|\partial_{t}^{2+j}u_{ell,2}(t,r)| \leq \frac{C r^{3} \sqrt{\lambda(t)}}{t^{6+j}}, \quad r \geq \lambda(t), \quad 0 \leq j \leq 2$$\end{lemma}
(The estimates on $u_{ell,2}$ in the above lemma follow directly from Lemma \ref{uellminusuellmainestlemma}, and the definition of $u_{ell,2}$). Next, recalling the definitions of $a_{1}$ and $b_{1}$ (\eqref{a1b1def}), a direct computation gives
\begin{equation}\label{em2def}\begin{split} &e_{m,2}:=u_{ell,2,main} -\left(v_{ex,sub,ell}+u_{w,3,ell,0}+w_{1,cm}-w_{1,lm}+v_{2,qm}-v_{2,lm}+u_{w,2,ell,0,1}\right)\\
&=- \frac{2}{\lambda(t) \sqrt{3}} \int_{\lambda(t)}^{\infty} x^{2}\partial_{1}^{2}(u_{w,2,ell,0,0}+v_{ex,ell,0})(t,x)\left(\phi_{0}(\frac{x}{\lambda(t)})+\frac{\sqrt{3}\lambda(t)}{2 x}\right) dx +\frac{43}{6} \lambda(t)^{3}a''(t)+\frac{53}{4}\lambda(t)^{2}b''(t) \\
&+ \frac{-2}{\sqrt{3}}\int_{1}^{\infty}s^{2}\lambda(t)^{2}\partial_{1}^{2}e_{ell,1}(t,s\lambda(t))\phi_{0}(s) ds- \frac{2}{\sqrt{3}}\int_{0}^{1} s^{2}\lambda(t)^{2} \partial_{1}^{2}u_{ell}(t,s\lambda(t)) \phi_{0}(s) ds\\
& - \frac{\lambda(t)^{2}}{12}\left(-206 \lambda(t) a''(t)-3(13+30 \log(\frac{3}{2})) b''(t)\right)-\frac{45}{4} \lambda(t)^{2}((6+\log(9))b_{1}(t)+3 \sqrt{3}\pi a_{1}(t)\lambda(t))\end{split}\end{equation}
In particular, the $r^{3}$, $r^{2}$, $r$, and $\log(r)$ terms from the functions on the left-hand side all exactly cancel in the combination $e_{m,2}$, which is independent of $r$, and turns out to be perturbative, as per the following lemma.
\begin{lemma}\label{em2estlemma} For $0 \leq j \leq 2$,
$$|\partial_{t}^{j}e_{m,2}(t)| \leq \frac{C \lambda(t)^{7/2}}{t^{4+j}}$$
\end{lemma}
We have not yet chosen the initial velocity, $v_{3,0}$, of the free wave $v_{3}$ from \eqref{v3def}, and don't need to choose it to allow for cancellations between $v_{3,main}$ and $e_{m,2}$, in light of the above lemma. On the other hand, as per Lemmas \ref{vexsubminussub0estlemma}, \ref{uw2ellminusell0estlemma}, and \ref{uw2subminusell30estlemma}, $u_{w,2,sub}-u_{w,3,ell,0}$, $u_{w,2,ell}-u_{w,2,ell,0}$, and $v_{ex,sub}-v_{ex,sub,ell}$ contain terms that are not quite perturbative. Therefore, we choose $v_{3,0}$, the initial velocity of the free wave $v_{3}$ from \eqref{v3def}, so that
\begin{equation}\label{thirdordercancellation}\begin{aligned}[t]&\int_{t}^{\infty} w_{2}(s) ds + \int_{t}^{\infty} ds q_{1}''(s) \log(\frac{s-t}{\sqrt{3}\lambda(s)})-\int_{t}^{\infty} ds g_{3}(s)- \int_{t}^{\infty} ds g_{2}''(s) \log\left(\frac{4 (s-t)^{2}}{3 \lambda(s)^{2}}\right)\\
&+\int_{t}^{\infty}ds(s-t)\partial_{s}^{2}\left(u_{w,2,ell}(s,y)-(u_{w,2,ell,00}(s,y) - \frac{675 \pi \lambda(s)^{5/2}\lambda''(s)}{16 y^{2}})\right)\Bigr|_{y=s-t}\\
&-\int_{t}^{\infty} ds (s-t)\partial_{s}^{2}\left(v_{ex,ell}(s,y)-v_{ex,ell,0}(s,y)\right)\Bigr|_{y=s-t}\\
&-\int_{0}^{\infty}  x V(t,x) (w_{1}+v_{2}-w_{1,lm}-v_{2,lm})(t,x) dx+v_{3,main}(t,r)=0, \quad t \geq T_{0}\end{aligned}\end{equation}
All of the estimates in the entire proof thus far are valid for all $t \geq T_{0}$ for any $T_{0} \geq T_{0,1}$ (recall \eqref{t0init}, $T_{0,1}$ is some fixed, absolute constant). In particular, they are valid for all $t \geq T_{0,1}$. Then, we choose
\begin{equation}\label{v30def}\begin{split}v_{3,0}(r) = \frac{\psi(r\frac{T_{\lambda}}{T_{0,1}})}{r} &\left(-\int_{t}^{\infty} w_{2}(s) ds - \int_{t}^{\infty} ds q_{1}''(s) \log(\frac{s-t}{\sqrt{3}\lambda(s)})+\int_{t}^{\infty} ds g_{3}(s)+ \int_{t}^{\infty} ds g_{2}''(s) \log\left(\frac{4 (s-t)^{2}}{3 \lambda(s)^{2}}\right)\right.\\
&-\int_{t}^{\infty}ds(s-t)\partial_{s}^{2}\left(u_{w,2,ell}(s,y)-(u_{w,2,ell,00}(s,y) - \frac{675 \pi \lambda(s)^{5/2}\lambda''(s)}{16 y^{2}})\right)\Bigr|_{y=s-t}\\
&\left.+\int_{t}^{\infty} ds (s-t)\partial_{s}^{2}\left(v_{ex,ell}(s,y)-v_{ex,ell,0}(s,y)\right)\Bigr|_{y=s-t}\right.\\
&\left.+\int_{0}^{\infty}  x V(t,x) (w_{1}+v_{2}-w_{1,lm}-v_{2,lm})(t,x) dx\right)\Bigr|_{t=r}\end{split}\end{equation}
where $\psi$ is defined in \eqref{psidef}, and further restrict $T_{0}$ to satisfy $T_{0} \geq 2 T_{0,1}$, but is otherwise arbitrary.  Lemmas \ref{vexellminusell0estlemma}, \ref{uw2ellminusell0estlemma}, \ref{uw2ell00estlemma1} thus imply
\begin{lemma}\label{v30estlemma} For $0 \leq j \leq 3$,
$$|v_{3,0}^{(j)}(r)| \leq \frac{C\mathbbm{1}_{\{r \geq T_{\lambda}\}}}{r^{4+j}} \lambda(r)^{2}\sup_{x \in [T_{\lambda},r]}\sqrt{\lambda(x)} \log^{2}(r)$$\end{lemma}
Now that we have defined and estimated $v_{3,0}$, we can estimate $v_{3}$, defined in \eqref{v3def}.
\begin{lemma}\label{v3minusv3mainestlemma} For $j=0$, $0 \leq k \leq 2$ or  $j=1$, and $0 \leq k \leq 1$,
$$|\partial_{r}^{k}\partial_{t}^{j}(v_{3}(t,r)-v_{3,main}(t,r))| \leq \frac{C r^{2-k}\lambda(t)^{2}\sup_{x \in [T_{\lambda},t]}\sqrt{\lambda(x)} \log^{2}(t)}{t^{5+j}}, \quad r \leq \frac{t}{2}$$
$$|\partial_{t}^{2}(v_{3}(t,r)-v_{3,main}(t,r))| \leq \frac{C r \lambda(t)^{2}\sup_{x \in [T_{\lambda},t]} \sqrt{\lambda(x)} \log^{2}(t)}{t^{6}}, \quad r \leq \frac{t}{2}$$
For $j+k \leq 2$,
$$|\partial_{r}^{k}\partial_{t}^{j}v_{3}(t,r)| \leq \frac{C (\sup_{x \in [T_{\lambda},r]}\sqrt{\lambda(x)})^{5} \log^{2}(r)}{r \langle t-r \rangle^{2}}\left(\frac{1}{\langle t-r\rangle^{j+k}}+\frac{1}{t^{j+k}}\right), \quad r \geq \frac{t}{2}$$
$$|\partial_{ttr} v_{3}(t,r)| + |\partial_{t}^{3}v_{3}(t,r)| \leq \begin{cases} \frac{C \lambda(t)^{2}\left(\sup_{x \in [T_{\lambda},t]} \sqrt{\lambda(x)}\right) \log^{2}(t)}{t^{6}}, \quad r \leq \frac{t}{2}\\
\frac{C \left(\sup_{x \in [T_{\lambda},t+r]}\sqrt{\lambda(x)}\right)^{5} \log^{2}(t+r)}{r} \left(\frac{1}{\langle t-r \rangle^{5}} + \frac{1}{t \langle t-r \rangle^{4}}\right), \quad r \geq \frac{t}{2}\end{cases}$$

\end{lemma}
\begin{proof}
We note that 
$$(v_{3}-v_{3,main})(t,r) = \frac{t}{2}\int_{0}^{\pi} \sin(\theta) \int_{0}^{r} dy (r-y) \partial_{y}^{2}(v_{3,0}(\sqrt{y^{2}+t^{2}+2 y t \cos(\theta)}))$$
This immediately leads to the estimates in the lemma statement for  $0 \leq j, k \leq 1$. Next, we use
$$(v_{3}-v_{3,main})(t,r) = \frac{t}{2}\int_{0}^{\pi} d\theta \sin(\theta) \int_{0}^{r} dy \frac{(y+t \cos(\theta))}{\sqrt{y^{2}+t^{2}+2 y t \cos(\theta)}} v_{3,0}'(\sqrt{t^{2}+y^{2}+2 y t \cos(\theta)}) $$
for the remaining estimates in the lemma statement,  since we only estimated three derivatives of $v_{3,0}$ in \eqref{v30estlemma}. We then finish the proof by using the same procedure as in Lemma \ref{v2estlemma}. \end{proof}
\subsection{Matching, Part 3}\label{thirdordermatching}
Here, we need to match the leading part of $u_{ell,3}$ (defined in \eqref{uell3def}) with appropriate parts of $w_{1,main}$ and $v_{2,main}$. The leading part of $u_{ell,3}$ in the matching region is given by the following function, as per the following lemma.
$$u_{ell,3,main}(t,r):= \int_{\lambda(t)}^{r} \left(\frac{-x^{2}}{r}+x\right)\partial_{1}^{2}(w_{1,cm}-w_{1,lm}+v_{2,qm}-v_{2,lm})(t,x) dx$$
\begin{lemma}\label{uell3minusuell3mainestlemma} For $j=0, 0 \leq k \leq 2$ or $j=1, 0 \leq k \leq 1$,
$$|\partial_{t}^{j}\partial_{r}^{k}\left(u_{ell,3}-u_{ell,3,main}\right)(t,r)| \leq \frac{C r^{3-k}\lambda(t)^{5/2}}{t^{6+j}}(1+\log(\frac{r}{\lambda(t)})), \quad r \geq \lambda(t)$$\end{lemma}
\begin{proof}
We recall that $u_{ell,2}$ is defined in \eqref{uell2def}, use the first order matching, \eqref{matching1}, and get
\begin{equation}\label{uell2foruell3minusmainest}\begin{split}u_{ell,2}(t,r)&= \int_{0}^{1}s^{2}\lambda(t)^{2}\left(\phi_{0}(\frac{r}{\lambda(t)})e_{2}(s)+\frac{\lambda(t)}{r}-(e_{2}(\frac{r}{\lambda(t)})\phi_{0}(s)+\frac{1}{s})\right)\partial_{1}^{2}(w_{1,lm}+v_{2,lm})(t,s\lambda(t)) ds\\
&+w_{1,cm}(t,r)-w_{1,lm}(t,r)+v_{2,qm}(t,r)-v_{2,lm}(t,r)\\
&+\int_{1}^{\frac{r}{\lambda(t)}} s^{2}\lambda(t)^{2}\left(\phi_{0}(\frac{r}{\lambda(t)})e_{2}(s)+\frac{\lambda(t)}{r} -(e_{2}(\frac{r}{\lambda(t)})\phi_{0}(s)+\frac{1}{s})\right)\partial_{1}^{2}(w_{1,lm}+v_{2,lm})(t,s\lambda(t)) ds\\
&+\int_{1}^{\frac{r}{\lambda(t)}}  (\phi_{0}(\frac{r}{\lambda(t)})e_{2}(s)-e_{2}(\frac{r}{\lambda(t)})\phi_{0}(s))s^{2}\lambda(t)^{2}\partial_{1}^{2}(u_{w,2,ell,0,0}+v_{ex,ell,0})(t,s\lambda(t)) ds\\
&+\int_{1}^{\frac{r}{\lambda(t)}}  (\phi_{0}(\frac{r}{\lambda(t)})e_{2}(s)-e_{2}(\frac{r}{\lambda(t)})\phi_{0}(s))s^{2}\lambda(t)^{2}\partial_{1}^{2}e_{ell,1}(t,s\lambda(t)) ds\\
&+\int_{0}^{1}  (\phi_{0}(\frac{r}{\lambda(t)})e_{2}(s)-e_{2}(\frac{r}{\lambda(t)})\phi_{0}(s))s^{2}\lambda(t)^{2}\partial_{1}^{2}(u_{ell}-v_{2,lm}-w_{1,lm})(t,s\lambda(t)) ds\end{split}\end{equation}
where we used
$$I:=\int_{0}^{\frac{r}{\lambda(t)}} s^{2}\lambda(t)^{2}\left(\frac{1}{s}-\frac{\lambda(t)}{r}\right) \partial_{1}^{2}(w_{1,lm}+v_{2,lm})(t,s\lambda(t)) ds=w_{1,cm}(t,r)-w_{1,lm}(t,r)+v_{2,qm}(t,r)-v_{2,lm}(t,r)$$
Therefore,
\begin{equation}\begin{split}&u_{ell,3}(t,r)-u_{ell,3,main}(t,r)\\
&=\int_{0}^{1}  (\phi_{0}(\frac{r}{\lambda(t)})e_{2}(s)-e_{2}(\frac{r}{\lambda(t)})\phi_{0}(s))s^{2}\lambda(t)^{2}\partial_{1}^{2}u_{ell,2}(t,s\lambda(t)) ds\\
&+\int_{\lambda(t)}^{r} \left(\frac{\phi_{0}(\frac{r}{\lambda(t)})x^{2}e_{2}(\frac{x}{\lambda(t)})}{\lambda(t)}-\frac{e_{2}(\frac{r}{\lambda(t)})x^{2}\phi_{0}(\frac{x}{\lambda(t)})}{\lambda(t)}+\frac{x^{2}}{r}-x\right)\partial_{1}^{2}(w_{1,cm}-w_{1,lm}+v_{2,qm}-v_{2,lm})(t,x)dx\\
&+\int_{\lambda(t)}^{r} \left(\frac{\phi_{0}(\frac{r}{\lambda(t)})x^{2}e_{2}(\frac{x}{\lambda(t)})}{\lambda(t)} - \frac{e_{2}(\frac{r}{\lambda(t)})x^{2}\phi_{0}(\frac{x}{\lambda(t)})}{\lambda(t)}\right)\partial_{1}^{2}(u_{ell,2}-I)(t,x)dx \end{split}\end{equation}
We then insert each term of \eqref{uell2foruell3minusmainest} into the last term in the above expression, and estimate the resulting integrals, and their derivatives directly.
 \end{proof}
 We also use Lemma \ref{uell2minusuell2mainestlemma} to directly estimate $u_{ell,3}$:
 \begin{lemma}\label{uell3estlemma} $$|\partial_{t}^{j}\partial_{r}^{k}u_{ell,3}(t,r)| \leq \frac{C \sqrt{\lambda(t)} r^{4-k}(r+\lambda(t))}{t^{6+j}}, \quad j,k=0,1$$
 For $0 \leq k \leq 1$,
 $$|\partial_{r}^{k}\partial_{t}^{2}u_{ell,3}(t,r)| \leq \begin{cases} \frac{C r^{4-k}\lambda(t)^{3/2}}{t^{8}}, \quad r \leq \lambda(t)\\
 \frac{C r^{5-k} \sqrt{\lambda(t)}}{t^{8}}, \quad r \geq \lambda(t)\end{cases}$$\end{lemma}
By direct computation, we get
$$u_{ell,3,main}-(w_{1,main}-w_{1,cm}+v_{2,main}-v_{2,qm}) = -\int_{0}^{\lambda(t)} \left(\frac{-x^{2}}{r}+x\right)\partial_{t}^{2}\left(w_{1,cm}-w_{1,lm}+v_{2,qm}-v_{2,lm}\right)(t,x) dx$$
Thus, the terms in $u_{ell,3,main}$ which grow fastest as $r$ approaches infinity match $w_{1,main}-w_{1,cm}+v_{2,main}-v_{2,qm}$. A direct estimation gives
\begin{lemma}\label{uell3mainminusleadingest} For $j=0$, $0 \leq k \leq 2$ or $j=1$, $0 \leq k \leq 1$, and for $r \geq \lambda(t)$,
$$|\partial_{t}^{j}\partial_{r}^{k}\left(u_{ell,3,main}-(w_{1,main}-w_{1,cm}+v_{2,main}-v_{2,qm})\right)| \leq \frac{C \lambda(t)^{11/2}}{t^{6+j}} \begin{cases} 1, \quad k=0\\
\frac{\lambda(t)}{r^{1+k}}, \quad k \geq 1\end{cases}$$\end{lemma}
\subsection{Preliminary Ansatz}
We need to add another correction into our ansatz, in order to eliminate the linear error term associated to $v_{3}$, namely
$$e_{3}(t,r):= \frac{-\chi_{\geq 1}(\frac{r}{h(t)}) \cdot 45 \lambda(t)^{2} v_{3}(t,r)}{(3\lambda(t)^{2}+r^{2})^{2}}$$
We let $u_{3}$ be the solution to
$$-\partial_{t}^{2}u_{3}+\partial_{r}^{2}u_{3}+\frac{2}{r}\partial_{r}u_{3}=e_{3}(t,r)$$
with zero Cauchy data at infinity. In other words,
\begin{equation}\label{u3formula}u_{3}(t,r) = \int_{t}^{\infty} ds \left(\frac{-1}{2r} \int_{|r-(s-t)|}^{r+s-t} y e_{3}(s,y) dy\right)\end{equation}
Then, we have
\begin{lemma}\label{u3estlemma} For $r \geq \frac{h(t)}{4}$, the following estimates are true.
$$|u_{3}(t,r)| \leq \frac{C \lambda(t)^{2} \left(\sup_{x \in [T_{\lambda},t]}\sqrt{\lambda(x)}\right)^{5} \left(\log^{2}(t) + \log^{2}(r)\right)}{t^{4}}$$
\begin{equation}\label{u3lgrest}|u_{3}(t,r)| \leq \frac{C \lambda(t)^{2}\left(\sup_{x \in [T_{\lambda},t]}\sqrt{\lambda(x)}\right)^{5} \log^{2}(t)}{r t^{3}}\end{equation}
$$|\partial_{r}u_{3}(t,r)| \leq \frac{C \lambda(t)^{2}\left(\sup_{x \in [T_{\lambda},t+r]}\sqrt{\lambda(x)}\right)^{5} (\log^{2}(t)+\log^{2}(r))}{r t^{3}} \left(\frac{1}{t}+\frac{(1+\frac{r}{t})^{\frac{5 C_{u}}{2}}}{\langle t-r \rangle^{2}}+\frac{r}{t^{3}}\right)$$
$$|\partial_{t}u_{3}(t,r)| \leq  C \frac{\lambda(t)^{2} \left(\sup_{x \in [T_{\lambda},t]}\sqrt{\lambda(x)}\right)^{5} \log^{2}(t)}{t^{5}}, \quad \frac{h(t)}{4} \leq r \leq \frac{t}{2}$$
\begin{equation}\begin{split}|\partial_{tr}u_{3}(t,r)| + |\partial_{t}^{2}u_{3}(t,r)| &\leq \frac{C \lambda(t)^{2} \left(\sup_{x \in [T_{\lambda},t]}\sqrt{\lambda(x)}\right)^{5} \log^{2}(t)}{r t^{4}h(t)^{2}} \left(1+ \frac{r h(t)^{2}}{t}\right), \quad \frac{h(t)}{4} \leq r \leq \frac{t}{2}\end{split}\end{equation}
\begin{equation}\begin{split}|\partial_{r}^{2}u_{3}| &\leq \frac{C \lambda(t)^{2} \left(\sup_{x \in [T_{\lambda},t]}\sqrt{\lambda(x)}\right)^{5} \log^{2}(t)}{r t^{3}h(t)^{2}} \cdot \left(\frac{h(t)^{2}}{r^{3}}+\frac{h(t)^{2}}{t r} + \frac{1}{t}\right), \quad \frac{h(t)}{4} \leq r \leq \frac{t}{2}\end{split}\end{equation}
Finally, 
\begin{equation}\label{u3enest}||\partial_{t}u_{3}||_{L^{2}(r^{2} dr)}+||\partial_{r}u_{3}||_{L^{2}(r^{2} dr)} \leq \frac{C \lambda(t)^{2}  \left(\sup_{x \in [T_{\lambda},t]}\sqrt{\lambda(x)}\right)^{5} \log^{2}(t)}{t^{3}}\end{equation}
\end{lemma}
\begin{proof}
We start with the following estimates on $e_{3}$, which are a direct consequence of Lemma \ref{v3minusv3mainestlemma}. 
\begin{equation}\label{e3ests} |e_{3}(t,r)| \leq \frac{C \mathbbm{1}_{\{r \geq \frac{h(t)}{2}\}} \lambda(t)^{2}}{(h(t)^{2}+r^{2})^{2}} \begin{cases} \frac{\lambda(t)^{2} \sup_{x \in [T_{\lambda},t]}\sqrt{\lambda(x)} \cdot \log^{2}(t)}{t^{3}}, \quad r \leq \frac{t}{2}\\
\frac{\left(\sup_{x \in [T_{\lambda},r]}\sqrt{\lambda(x)}\right)^{5} \log^{2}(r)}{r \langle t-r \rangle^{2}}, \quad r \geq \frac{t}{2}\end{cases}\end{equation}
Then, we insert these into \eqref{u3formula}, to get 
\begin{equation}\begin{split} u_{3}(t,r)&\leq C \int_{t}^{t+2r} \frac{ds}{r} \int_{|r-(s-t)|}^{r+s-t} \frac{y \lambda(s)^{4} \left(\sup_{x \in [T_{\lambda},s]}\sqrt{\lambda(x)}\right) \log^{2}(s)}{s^{3}(h(s)^{2}+y^{2})^{2}}\\
&+C \int_{t}^{t+2r} \frac{ds}{r} \int_{|r-(s-t)|}^{r+s-t} \frac{\lambda(s)^{2} \left(\sup_{x \in [T_{\lambda},s]}\sqrt{\lambda(x)}\right)^{5} \log^{2}(y) dy}{\langle s-y \rangle^{2} s^{4}}\\
&+C \int_{t+2r}^{\infty}\frac{ds}{r} \int_{\frac{s-t}{2}}^{\frac{3(s-t)}{2}} \frac{y \lambda(s)^{2}}{(h(s)^{2}+y^{2})^{2}} \frac{\lambda(s)^{2} \left(\sup_{x \in [T_{\lambda},s]}\sqrt{\lambda(x)}\right) \log^{2}(s)}{s^{3}} dy\\
&+C \int_{t+2r}^{\infty} \frac{ds}{r} \int_{|r-(s-t)|}^{r+s-t} \frac{\lambda(s)^{2}}{s^{4}} \left(\frac{\left(\sup_{x \in [T_{\lambda},s]}\sqrt{\lambda(x)}\right)^{5} \log^{2}(y)}{\langle s-y \rangle^{2}}\right)dy\end{split}\end{equation}
which gives the first estimate in the lemma statement. For \eqref{u3lgrest} (which is useful when $r \geq t$) we again insert \eqref{e3ests} into \eqref{u3formula}, and estimate the integrals by:
\begin{equation}\begin{split}|u_{3}(t,r)| &\leq C \int_{t}^{\infty} \frac{ds}{r} \int_{h(s)}^{\infty} dy \frac{\lambda(s)^{4} \left(\sup_{x \in [T_{\lambda},s]}\sqrt{\lambda(x)}\right) \log^{2}(s)}{y^{3}s^{5}} \\
&+ C \int_{t}^{\infty} \frac{ds}{r}\int_{\frac{s}{2}}^{\infty} dy \frac{\lambda(s)^{2} \left(\sup_{x \in [T_{\lambda},t]}\sqrt{\lambda(x)}\right)^{5} \log^{2}(s)}{s^{2}y^{2}\langle s-y \rangle^{2}}\end{split}\end{equation}
We estimate $\partial_{r}u_{3}$ similarly. For $\partial_{t}u_{3}$, we note that 
$$u_{3}(t,r) = \int_{0}^{\infty} dw \left(\frac{-1}{2r}\int_{|r-w|}^{r+w} y e_{3}(t+w,y) dy\right)$$
so
$$\partial_{t}u_{3}(t,r) =\int_{0}^{\infty} dw \left(\frac{-1}{2r} \int_{|r-w|}^{r+w} y \partial_{1}e_{3}(t+w,y) dy\right)$$
Then, we estimate $\partial_{t}u_{3}$ similarly to $u_{3}$. $\partial_{tr}u_{3}$ and $\partial_{t}^{2}u_{3}$ are estimated similarly to $\partial_{r}u_{3}$. Next, we use the equation
$$\partial_{r}^{2}u_{3} = e_{3}(t,r) -\frac{2}{r}\partial_{r}u_{3}+\partial_{t}^{2}u_{3}$$
to estimate $\partial_{r}^{2}u_{3}$ from our previous estimates. Finally, the energy estimate, \eqref{u3enest}, is proven with the same procedure used to establish \eqref{uw2enest}, with the analog of \eqref{uw2roughptwse} being 
$$|u_{3}(t,r)| \leq \frac{C \lambda(t)^{4} \left(\sup_{x \in [T_{\lambda},t]}\sqrt{\lambda(x)}\right) \log^{2}(t)}{h(t)^{3}t^{4}} + \frac{C \lambda(t)^{2} \left(\sup_{x \in [T_{\lambda},t]}\sqrt{\lambda(x)}\right)^{5} \log^{2}(t)}{t^{3}}$$
 (which results from directly inserting \eqref{e3ests} into \eqref{u3formula}). \end{proof}
Let 
\begin{equation}\label{psi2def}\psi_{2}(x) = \begin{cases} 1, \quad x \geq \frac{1}{2}\\
0, \quad x \leq \frac{1}{4}\end{cases}\end{equation}
We will insert the function $\psi_{2}(\frac{r}{h(t)}) u_{3}(t,r)$ into our ansatz, and define its error term as
$$e_{u_{3}}(t,r):= -\left(\left(-\partial_{t}^{2}+\partial_{r}^{2}+\frac{2}{r}\partial_{r}-V(t,r)\right)(\psi_{2}(\frac{r}{h(t)})u_{3})-e_{3}(t,r)\right)$$
where we recall the definition of $V$ in \eqref{Vdef}. Using the support properties of $e_{3}$ and $\psi_{2}$,
$$(1-\psi_{2}(\frac{r}{h(t)})) e_{3}(t,r) =0$$
and this gives
\begin{equation}\begin{split}e_{u_{3}}(t,r)&:= -2 \partial_{r} u_{3}(t,r) \partial_{r}(\psi_{2}(\frac{r}{h(t)}))+2 \partial_{t} u_{3}(t,r) \partial_{t}\left(\psi_{2}(\frac{r}{h(t)})\right)+ V(t,r) u_{3}(t,r)\psi_{2}(\frac{r}{h(t)})\\
&+u_{3}(t,r) \left(-\partial_{r}^{2} \left(\psi_{2}(\frac{r}{h(t)})\right)+\partial_{t}^{2} \left(\psi_{2}(\frac{r}{h(t)})\right)-\frac{2}{r} \partial_{r} \left(\psi_{2}(\frac{r}{h(t)})\right)\right)\end{split}\end{equation}
A direct application of Lemma \ref{u3estlemma} gives
\begin{lemma}\label{eu3lemma} $$||e_{u_{3}}(t,|\cdot|)||_{H^{1}(\mathbb{R}^{3})} \leq \frac{C \lambda(t)^{2} \left(\sup_{x \in [T_{\lambda},t]}\sqrt{\lambda(x)}\right)^{5} \log^{2}(t)}{\sqrt{h(t)} t^{4}}$$\end{lemma}
We now assemble our previous corrections into a function, $u_{a}$, as follows. Let
$$u_{e}(t,r) = u_{ell}(t,r)+u_{ell,2}(t,r)+u_{ell,3}(t,r), \quad u_{w}(t,r)=w_{1}(t,r)+v_{ex}(t,r)+v_{2}(t,r)+u_{w,2}(t,r)+v_{3}(t,r)$$
\begin{equation}\label{chiprops}\chi_{\leq 1} \in C^{\infty}([0,\infty)), \quad \chi_{\leq 1}(x) = \begin{cases} 1, \quad x \leq 1\\
0, \quad x \geq 2\end{cases}, 0 \leq \chi_{\leq 1}(x) \leq 1, \quad \chi_{\geq 1}(x) = 1-\chi_{\leq 1}(x)\end{equation}
$$u_{c}(t,r) = \chi_{\leq 1}(\frac{r}{h(t)})u_{e}(t,r) +(1-\chi_{\leq 1}(\frac{r}{h(t)}))u_{w}(t,r)$$
Then, we define 
$$u_{a}(t,r)=u_{c}(t,r)+u_{3}(t,r) \psi_{2}(\frac{r}{h(t)})$$
If we substitute $u(t,r) = Q_{\lambda(t)}(r) + u_{a}(t,r)+v(t,r)$ into \eqref{slw}, we obtain
\begin{equation}\label{eqnafterua}\begin{split}&-\partial_{t}^{2}v+\partial_{r}^{2}v+\frac{2}{r}\partial_{r}v+\frac{45 \lambda(t)^2}{\left(3 \lambda(t)^2+r^2\right)^2}v(t,r)\\
&=e_{match}(t,r)+e_{ell,3}(t,r)+e_{ex}(t,r)+e_{w,2}(t,r)+e_{u_{3}}(t,r) -((Q_{\lambda(t)}+u_{a}+v)^{5}-Q_{\lambda(t)}^{5}-5Q_{\lambda(t)}^{4}(u_{a}+v))\end{split}\end{equation}
where
\begin{equation}\label{ematchdef}\begin{split}e_{match}(t,r) &= \left(\partial_{t}^{2}(\chi_{\leq 1}(\frac{r}{h(t)}))-\partial_{r}^{2}(\chi_{\leq 1}(\frac{r}{h(t)}))-\frac{2}{r}\partial_{r}(\chi_{\leq 1}(\frac{r}{h(t)}))\right)\left(u_{e}-u_{w}\right)\\
&+2\partial_{t}(\chi_{\leq 1}(\frac{r}{h(t)}))\partial_{t}(u_{e}-u_{w})-2\partial_{r}(\chi_{\leq 1}(\frac{r}{h(t)}))\partial_{r}(u_{e}-u_{w})\end{split}\end{equation}
$$e_{ell,3}(t,r) = \chi_{\leq 1}(\frac{r}{h(t)}) \partial_{t}^{2}u_{ell,3}$$
$$e_{ex}(t,r):= - \chi_{\geq 1}(\frac{r}{h(t)}) \frac{45 \lambda(t)^{2} v_{ex}(t,r)}{(3\lambda(t)^{2}+r^{2})^{2}}$$
$$e_{w,2}(t,r) = V(t,r) u_{w,2}(t,r) \chi_{\geq 1}(\frac{r}{h(t)})$$
We now estimate each of the linear error terms, starting with $e_{match}$. Using \eqref{matching1} and \eqref{em2def}, we get
\begin{equation}\begin{split}u_{e}(t,r)-u_{w}(t,r)&= u_{ell}-u_{ell,main}+u_{ell,2}-u_{ell,2,main}+u_{ell,3}-u_{ell,3,main}-(v_{3}-v_{3,main})+e_{m,2}\\
&-(w_{1}-w_{1,main})-(v_{2}-v_{2,main})-(u_{w,2}-u_{w,2,ell,00}-u_{w,2,ell,0,1}-u_{w,3,ell,0})\\
&-(v_{ex}-v_{ex,ell,0}-v_{ex,sub,ell})+u_{ell,3,main}-(w_{1,main}-w_{1,cm}+v_{2,main}-v_{2,qm})-v_{3,main}\end{split}\end{equation}
By directly combining Lemmas \ref{uellminusuellmainestlemma}, \ref{uell2minusuell2mainestlemma}, \ref{uell3minusuell3mainestlemma}, \ref{v3minusv3mainestlemma}, \ref{em2estlemma}, \ref{uw2subminusell30estlemma}, \ref{v2minusmainest}, \ref{uw2ell00estlemma1}, \ref{uw2ellminusell0estlemma}, \ref{vexsubminussub0estlemma}, \ref{vexellminusell0estlemma}, \ref{uell3mainminusleadingest}, we get
\begin{lemma} \label{ueminusuw0estlemma}For $j=0, 0 \leq k \leq 2$ or $j=1, 0 \leq k \leq 1$, and in the region $\lambda(t) \leq r \leq \frac{t}{4}$, we have
\begin{equation}\begin{split}|\partial_{t}^{j}\partial_{r}^{k}((u_{e}-u_{w}+w_{1}-w_{1,main})(t,r))| &\leq \frac{C \sqrt{\lambda(t)}}{r^{3+k}t^{2+j}} \left(\lambda(t)^{4} \log^{2}(t) + \frac{r^{9}\lambda(t)}{t^{6}}\right) \\
&+ C \begin{cases} \frac{r^{2}\lambda(t)^{2}\sup_{x \in [T_{\lambda},t]}\sqrt{\lambda(x)} \log^{2}(t)}{t^{5}} + \frac{\lambda(t)^{7/2} \log^{2}(t)}{t^{4}}, \quad j=k=0\\
 \frac{ r \lambda(t)^{2}\sup_{x \in [T_{\lambda},t]}\sqrt{\lambda(x)} \log^{2}(t)}{t^{5}} + \frac{\lambda(t)^{7/2} \log^{2}(t)}{r t^{4}}, \quad k=1,j=0\\
\frac{\lambda(t)^{2}\sup_{x \in [T_{\lambda},t]}\sqrt{\lambda(x)} \log^{2}(t)}{t^{4}}, \quad j=1,k=0\\
\frac{\lambda(t)^{2} \sup_{x \in [T_{\lambda},t]}\sqrt{\lambda(x)} \log^{2}(t)}{t^{4}r}, \quad j+k =2\end{cases}\end{split}\end{equation}\end{lemma}
Lemmas \ref{w1minusmainest} and \ref{ueminusuw0estlemma} directly give
\begin{lemma} \label{ematch0estlemma}
$$||e_{match}(t,|\cdot|)||_{H^{1}(\mathbb{R}^{3})}\leq \frac{C \log^{2}(t)}{t^{4+\delta}}$$
where
$$\delta = \frac{1}{2}\text{min}\{4-7C_{u}-\frac{13}{2}a(1-C_{u}),-2-C_{u}+\frac{7}{2}a(1-C_{u}),1-4C_{u}+\frac{3}{2}a(-1+C_{u}),\frac{1}{2}a(1-C_{u})-3C_{u}\}>0$$
\end{lemma}
\noindent Note that the positivity of $\delta$ is due to \eqref{hdef}. Next, we consider $e_{w,2}$.  By Lemma \ref{uw2estlemma}, we get
\begin{lemma}\label{ew2lemma} $$||e_{w,2}(t,|\cdot|)||_{H^{1}(\mathbb{R}^{3})} \leq \frac{C \lambda(t)^{2} \sup_{x \in [T_{\lambda},t]}(\lambda(x)^{5/2}) \log^{2}(t)}{h(t)^{7/2}t^{2}} + \frac{C \lambda(t)^{2} \log^{2}(t) \sup_{x \in [T_{\lambda},t]}(\lambda(x)^{3/2})}{t^{2}h(t)^{7/2}}$$\end{lemma}
Similarly, Lemma \ref{vexestlemma} gives
\begin{lemma}\label{eexlemma}
$$||e_{ex}(t,|\cdot|)||_{H^{1}(\mathbb{R}^{3})} \leq \frac{C \lambda(t)^{9/2} \log(t)}{t^{2}h(t)^{7/2}}\left(1+\frac{1}{h(t)}\right)$$\end{lemma}
The next linear error term to consider is $e_{ell,3}$. From Lemma \ref{uell3estlemma}, we get
\begin{lemma}\label{eell3lemma}
$$||e_{ell,3}(t,|\cdot|)||_{H^{1}(\mathbb{R}^{3})} \leq \frac{C h(t)^{13/2} \sqrt{\lambda(t)}}{t^{8}}$$\end{lemma}
\subsection{Nonlinear error terms, part 1}\label{nonlinear1section}
We study the nonlinear error terms in \eqref{eqnafterua}. In particular, we define
$$u_{a,0}(t,r):= \chi_{\leq 1}(\frac{r}{h(t)}) u_{ell}(t,r)+(1-\chi_{\leq 1}(\frac{r}{h(t)}))(w_{1}+v_{2}+v_{3})$$
$$u_{a,1}(t,r):= \chi_{\leq 1}(\frac{r}{h(t)})(u_{ell,2}(t,r)+u_{ell,3}(t,r)) + (1-\chi_{\leq 1}(\frac{r}{h(t)}))(v_{ex}+u_{w,2}) + u_{3}(t,r)\psi_{2}(\frac{r}{h(t)})$$
so that $u_{a}(t,r)=u_{a,0}(t,r)+u_{a,1}(t,r)$. Next, we split the nonlinear terms into $N_{0}$, which is not quite perturbative, and $N_{1}$ which is.
\begin{equation}\label{n0def}N_{0}(t,r):= -\left(10 Q_{\lambda(t)}^{3}u_{a,0}^{2} + 10Q_{\lambda}^{2} u_{a,0}^{3}+5 Q_{\lambda}u_{a,0}^{4}+u_{a,0}^{5}\right)\end{equation}
$$N_{1}(t,r):= -\left(10 Q_{\lambda(t)}^{3}(u_{a}^{2}-u_{a,0}^{2})+10 Q_{\lambda}^{2}(u_{a}^{3}-u_{a,0}^{3}) + 5 Q_{\lambda}(u_{a}^{4}-u_{a,0}^{4})+u_{a}^{5}-u_{a,0}^{5}\right)$$
We begin by showing that $N_{1}$ is perturbative:
\begin{lemma} $$||N_{1}(t,|\cdot|)||_{H^{1}(\mathbb{R}^{3})} \leq \frac{C \log^{10}(t)}{t^{4+\delta_{1}}}$$
where $\delta_{1}>0$ is given by $$\delta_{1}=\frac{1}{2}\text{min}\{2+\frac{5}{2}a(-1+C_{u})-5C_{u},\frac{3}{2}a(1-C_{u})-3C_{u},\frac{3}{2}a -\frac{3}{2}(1-a)C_{l}-\frac{7}{2}C_{u},\frac{3}{2}-C_{l}-\frac{9}{2}C_{u},2-C_{l}-\frac{25}{2}C_{u}\}$$\end{lemma}
\begin{proof}
From the definition of $N_{1}$, we get
\begin{equation}\label{n1estintstep}|N_{1}(t,r)| \leq C |u_{a,1}| (|Q_{\lambda}|^{3}(|u_{a}|+|u_{a,0}|)+u_{a,0}^{4}+u_{a}^{4})\end{equation}
Let 
\begin{equation}\begin{split} u_{a,0,est}(t,r) &= \left( \mathbbm{1}_{\{r \leq 2 h(t)\}} \frac{\lambda(t)^{3/2}}{t^{2}} \begin{cases} 1, \quad r \leq \lambda(t)\\
\frac{r}{\lambda(t)}, \quad r \geq \lambda(t)\end{cases} \right.\\
&\left.+ \mathbbm{1}_{\{r \geq h(t)\}} \begin{cases} \frac{r \sup_{x \in [T_{\lambda},t]}\sqrt{\lambda(x)}}{t^{2}}, \quad r \leq \frac{t}{2}\\
\frac{\left(\sup_{x \in [T_{\lambda},t+r]}\sqrt{\lambda(x)}\right)\log^{2}(t+r)}{r} \left(1+\frac{\left(\sup_{x \in [T_{\lambda},t+r]}\sqrt{\lambda(x)}\right)^{4}}{\langle t-r \rangle^{2}}\right), \quad r \geq \frac{t}{2}\end{cases}\right)\end{split}\end{equation}
Then, using Lemmas \ref{uellminusuellmainestlemma}, \ref{w1estlemma}, \ref{v2estlemma}, \ref{v3minusv3mainestlemma} and \ref{v30estlemma}, we get,
\begin{equation}\begin{split}|\partial_{r}^{k}u_{a,0}(t,r)| &\leq \begin{cases} \frac{C r}{\sqrt{\lambda(t)} t^{2}}, \quad k=1 \text{ and } r \leq \lambda(t)\\
C \left(\frac{1}{r^{k}}+\frac{1}{\langle t-r \rangle^{k}}+\frac{1}{t^{k}}\right) \cdot u_{a,0,est}(t,r), \quad k=0 \text{ and } r>0 \text{ or } k=1 \text{ and } r \geq \lambda(t)\end{cases}\end{split}\end{equation}
For later use, we also note that, for $0 \leq j \leq 3$,
\begin{equation}\label{dtjua0est}\begin{split}|\partial_{t}^{j} u_{a,0}(t,r)| &\leq C \left(\frac{1}{t^{j}}+\frac{1}{\langle t-r \rangle^{j}}\right) \cdot u_{a,0,est}(t,r)\end{split}\end{equation}
and the following two estimates are true:
\begin{equation}\label{dtrua0est} |\partial_{tr}u_{a,0}(t,r)| \leq C u_{a,0,est}(t,r) \left(\frac{1}{tr} + \frac{1}{\langle t-r \rangle^{2}} + \frac{1}{r \langle t-r \rangle} + \frac{1}{t^{2}}\right)\end{equation}
\begin{equation}\label{dttrua0est} |\partial_{ttr}u_{a,0}(t,r)| \leq C u_{a,0,est}(t,r) \left(\frac{1}{t^{2}r} + \frac{1}{\langle t-r \rangle^{3}} +\frac{1}{r \langle t-r \rangle^{2}} + \frac{1}{t^{3}}\right)\end{equation}
Next, we use Lemmas \ref{uell2minusuell2mainestlemma}, \ref{uell3estlemma}, \ref{vexestlemma}, \ref{uw2estlemma}, \ref{u3estlemma}, to get, for $k=0,1$,
\begin{equation}\label{drkua1est}\begin{split}|\partial_{r}^{k}u_{a,1}(t,r)| &\leq C \left(\frac{1}{r^{k}}+\frac{1}{t^{k}}\right) \cdot \left( \mathbbm{1}_{\{r \leq 2 h(t)\}} \frac{r^{2}(r+\lambda(t))\sqrt{\lambda(t)}}{t^{4}}\right. \\
&\left.+ C \mathbbm{1}_{\{r \geq \frac{h(t)}{4}\}} \left(\frac{\left(\sup_{x \in [T_{\lambda},t+r]}\lambda(x)^{5/2}\right)\log^{2}(t+r)}{r t^{2}} + \frac{\lambda(t)^{2}\left(\sup_{x \in [T_{\lambda},t]}\sqrt{\lambda(x)}\right)^{5} \log^{2}(t)}{t^{3} \text{max}\{r,t\}}\right)\right)\\
&+C\begin{cases}0, \quad k=0\\
|\partial_{r}u_{3}(t,r)| |\psi_{2}(\frac{r}{h(t)})| +  \mathbbm{1}_{\{r \geq h(t)\}} \frac{\left(\sup_{x \in [T_{\lambda},t+r]}\lambda(x)^{3/2}\right) \log^{2}(t+r)}{r t^{2}}, \quad k=1\end{cases}\end{split}\end{equation}
We obtain the lemma statement by inserting the above into \eqref{n1estintstep}, and estimating directly, using Lemma \ref{u3estlemma} to estimate $||\partial_{r}u_{3}(t,r)||_{L^{2}(r^{2} dr)}$. 
\end{proof}
Next, we consider $N_{0}$, which we recall is defined in \eqref{n0def}. We let $u_{N_{0}}$ be the solution to the following equation with $0$ Cauchy data at infinity.
\begin{equation}\label{un0eqn}-\partial_{t}^{2}u_{N_{0}}+\partial_{r}^{2}u_{N_{0}}+\frac{2}{r}\partial_{r}u_{N_{0}}=N_{0}\end{equation}
We have
\begin{equation}\label{un0def}u_{N_{0}}(t,r) = \int_{t}^{\infty} ds \left(\frac{-1}{2r}\int_{|r-(s-t)|}^{r+s-t} y N_{0}(s,y) dy\right)\end{equation}
Note that all of the estimates in the entire proof thus far are valid for all $t \geq T_{0}$ for any $T_{0} \geq 2 T_{0,1}$ (recall that we restricted $T_{0}$ after \eqref{v30def}). In particular, they are valid for all $t \geq 2T_{0,1}$.
We then define
\begin{equation}\label{v40def}v_{4,0}(r) = \frac{1}{r} \int_{r}^{\infty} ds (s-r) N_{0}(s,s-r) \psi(\frac{r T_{\lambda}}{2 T_{0,1}}), \quad r >0\end{equation}
for $\psi$ as in \eqref{psidef}, restrict $T_{0}$ to satisfy $T_{0} \geq 4 T_{0,1}$, but is otherwise arbitrary, and we let $v_{4}$ solve
$$\begin{cases}-\partial_{t}^{2}v_{4}+\partial_{r}^{2}v_{4}+\frac{2}{r}\partial_{r}v_{4}=0\\
v_{4}(0,r)=0, \quad \partial_{t}v_{4}(0,r) = v_{4,0}(r)\end{cases}$$
Then, we have the following estimates
\begin{lemma}\label{un0estlemma}
For $j=0,1,2$,
\begin{equation}\label{un0smrest}\begin{split}&|\partial_{t}^{j}\left(u_{N_{0}}(t,r)+t v_{4,0}(t)\right)|\\
&\leq \frac{C}{t^{j}}\left( \frac{ \text{min}\{1,r\} \left(\sup_{x \in [T_{\lambda},t]} \sqrt{\lambda(x)}\right)^{25} \log^{10}(t)}{t^{4}}+\frac{ r \left(\sup_{x \in [T_{\lambda},t]} \sqrt{\lambda(x)}\right)^{5} (\log^{10}(t)+\log^{10}(r))}{t^{4}}\right.\\
&+\left.\frac{ r^{2}\left(\sup_{x \in [T_{\lambda},t]}\sqrt{\lambda(x)}\right)^{3}}{t^{4}}\right), \quad r \leq \frac{t}{2}\end{split}\end{equation}
\begin{equation}\label{un0lgrest}|u_{N_{0}}(t,r)| \leq \frac{C \left(\sup_{x \in [T_{\lambda},t]}\sqrt{\lambda(x)}\right)^{5} \log^{10}(t)}{r t^{2}} , \quad r \geq \frac{t}{2}\end{equation}
\begin{equation}\label{drun0smrest}\begin{split}|\partial_{r}u_{N_{0}}(t,r)| &\leq \frac{C \left(\sup_{x \in [T_{\lambda},t]}\sqrt{\lambda(x)}\right)^{25} \log^{10}(t) \text{min}\{1,r\}}{r t^{4}}\\
&+\frac{C \left(\sup_{x \in [T_{\lambda},t]}\sqrt{\lambda(x)}\right)^{5} (\log^{10}(t)+\log^{10}(r))}{t^{4}} + \frac{C r \left(\sup_{x \in [T_{\lambda},t]}\sqrt{\lambda(x)}\right)^{3}}{t^{4}}, \quad r \leq \frac{t}{2}\end{split}\end{equation}
\begin{equation}\begin{split} |\partial_{r}u_{N_{0}}(t,r)|&\leq \frac{C \left(\sup_{x \in [T_{\lambda},t+r]}\sqrt{\lambda(x)}\right)^{5} \log^{10}(t+r)}{r t^{3}} \left(1+\left(\sup_{x \in [T_{\lambda},t+r]}\sqrt{\lambda(x)}\right)^{20}\left(\frac{1}{t}+\frac{1}{\langle t-r \rangle^{10}}\right)\right), \quad r \geq \frac{t}{2}\end{split}\end{equation}
For $0 \leq k \leq 2$,
\begin{equation}\label{drkv4smrest}|\partial_{r}^{k}(v_{4}(t,r)-t v_{4,0}(t))| \leq\frac{C r^{2-k} \left(\sup_{x \in [T_{\lambda},t]}\sqrt{\lambda(x)}\right)^{5} \log^{10}(t)}{t^{5}}, \quad r \leq \frac{t}{2}\end{equation}
For $j=1,2$,
\begin{equation}\label{dtjv4smrest}|\partial_{t}^{j}(v_{4}-t v_{4,0}(t))| \leq \frac{C r^{3-j} \left(\sup_{x \in [T_{\lambda},t]}\sqrt{\lambda(x)}\right)^{5} \log^{10}(t)}{t^{6}}, \quad r \leq \frac{t}{2}\end{equation}
For $0 \leq j+k \leq 2$,
\begin{equation}\label{v4lgrest}|\partial_{t}^{j}\partial_{r}^{k}v_{4}(t,r)| \leq \frac{C}{r \langle t-r \rangle^{2}}\left(\sup_{x \in [T_{\lambda},t+r]}\sqrt{\lambda(x)}\right)^{5} \log^{10}(t+r) \left(\frac{1}{\langle t-r \rangle^{j+k}}+\frac{1}{t^{j+k}}\right), \quad r \geq \frac{t}{2}\end{equation}
Finally, \begin{equation}\label{un0enest} ||\partial_{t}u_{N_{0}}(t,r)||_{L^{2}(r^{2}dr)} + ||\partial_{r}u_{N_{0}}||_{L^{2}(r^{2} dr)} \leq \frac{C \left(\sup_{x \in [T_{\lambda},t]}\sqrt{\lambda(x)}\right)^{5} \log^{10}(t)}{t^{5/2}} \end{equation}
\end{lemma}
\begin{proof}
We start with estimates on $N_{0}$. Define
$$N_{0,est}(t,r) = \begin{cases}\frac{\lambda(t)^{3/2}}{t^{4}}, \quad r \leq \lambda(t)\\
\frac{\lambda(t)^{5/2} }{r t^{4}}, \quad \lambda(t) \leq r \leq h(t)\\
\frac{\left(\sup_{x \in [T_{\lambda},t]}\sqrt{\lambda(x)}\right)^{5}}{r t^{4}}, \quad h(t) \leq r \leq \frac{t}{2}\\
\frac{\left(\sup_{x \in [T_{\lambda},t+r]}\sqrt{\lambda(x)}\right)^{5} \log^{10}(t+r)}{r^{5}} \left(1+\frac{\left(\sup_{x \in [T_{\lambda},t+r]}\sqrt{\lambda(x)}\right)^{20}}{\langle t-r \rangle^{10}}\right), \quad r \geq \frac{t}{2}\end{cases}$$ 
Directly estimating the definition of $N_{0}$, we get, for $k=0,1$,
\begin{equation}\label{drkn0est}\begin{split}|\partial_{r}^{k}N_{0}(t,r)| \leq C \left(\frac{1}{r}+\frac{1}{\langle t-r \rangle} + \frac{1}{t}\right)^{k} \cdot N_{0,est}(t,r) \end{split}\end{equation}
We also note that 
$$|\partial_{t}N_{0}(t,r)| \leq C(Q_{\lambda}(r)^{2}u_{a,0}^{2}+u_{a,0}^{4}) |\partial_{t}Q_{\lambda}| + C(|Q_{\lambda}|^{3} |u_{a,0}| + u_{a,0}^{4}) |\partial_{t}u_{a,0}|$$
Similarly directly estimating $\partial_{t}^{2}N_{0}(t,r)$, and using \eqref{dtjua0est}, we get, for $j=1,2,3$,
$$|\partial_{t}^{j}N_{0}(t,r)| \leq C N_{0,est}(t,r) \left(\frac{1}{t^{j}}+\frac{1}{\langle t-r \rangle^{j}}\right)$$
Also with the same procedure, for $j=1,2$,
$$|\partial_{t}^{j}\partial_{r}N_{0}(t,r)| \leq C \left(\frac{1}{t^{j}}+\frac{1}{\langle t-r \rangle^{j}}\right)\left(\frac{1}{r}+\frac{1}{\langle t-r \rangle} + \frac{1}{t}\right) \cdot N_{0,est}(t,r)$$
Inserting \eqref{drkn0est} into \eqref{un0def} gives \eqref{un0lgrest}. To obtain \eqref{un0smrest} for $j=0$, we first decompose \eqref{un0def} as
\begin{equation}\label{un0estintstep1} \begin{split} u_{N_{0}}(t,r) &= \int_{t}^{t+r} ds \left(\frac{-1}{2r}\right) \int_{|r-(s-t)|}^{r+s-t} y N_{0}(s,y) dy - \int_{t}^{\infty} ds (s-t) N_{0}(s,s-t) + \int_{t}^{t+r} ds (s-t) N_{0}(s,s-t)\\
&+\int_{t+r}^{\infty} ds \left(\frac{-1}{2r}\right) \int_{s-t-r}^{s-t+r} (y N_{0}(s,y) - y N_{0}(s,s-t)) dy\end{split}\end{equation}
To estimate the last term in \eqref{un0estintstep1}, we use the mean value theorem for the function $x \mapsto N_{0}(s,x)$, and directly estimate the remaining terms, to get \eqref{un0smrest} for $j=0$. Next, we directly estimate the definition of $v_{4,0}$, \eqref{v40def}, to get, for $0 \leq j \leq 3$,
\begin{equation}\label{v40ests}|v_{4,0}^{(j)}(r)| \leq \frac{C \left(\sup_{x \in [T_{\lambda},r]}\sqrt{\lambda(x)}\right)^{5} \log^{10}(r) \mathbbm{1}_{\{r \geq T_{\lambda}\}}}{r^{4+j}}\end{equation}
For \eqref{drkv4smrest}, we write
\begin{equation}\begin{split}v_{4}(t,r) -t v_{4,0}(t) &= \frac{t}{2}\int_{0}^{\pi} \sin(\theta) \left(v_{4,0}(\sqrt{r^{2}+t^{2}+2 r t \cos(\theta)})-v_{4,0}(t)\right) d\theta\\
&= \frac{t}{2} \int_{0}^{\pi} \sin(\theta) \int_{0}^{r} dy (r-y)\partial_{y}^{2} \left(v_{4,0}(\sqrt{y^{2}+t^{2}+2 y t \cos(\theta)})\right) d\theta\end{split}\end{equation}
Directly estimating this gives \eqref{drkv4smrest}, and \eqref{dtjv4smrest} for $j=1$. Since we only have three derivatives of $v_{4,0}$ estimated in \eqref{v40ests}, we obtain \eqref{dtjv4smrest} for $j=2$ by directly differentiating and estimating
$$v_{4}(t,r)-t v_{4,0}(t) = \frac{t}{2}\int_{0}^{\pi}\sin(\theta) \int_{0}^{r} \partial_{y}(v_{4,0}(\sqrt{y^{2}+t^{2}+2 y t \cos(\theta)}))dy d\theta$$
Next, we obtain \eqref{v4lgrest} with the same procedure as for the analogous estimates in Lemma \ref{v3minusv3mainestlemma}. Next, for \eqref{drun0smrest}, we note that
\begin{equation}\begin{split} \partial_{r}u_{N_{0}}(t,r) &= \partial_{r}\left(u_{N_{0}}(t,r) + \int_{t}^{\infty} ds (s-t) N_{0}(s,s-t)\right)\end{split}\end{equation}
and we write $u_{N_{0}}(t,r) + \int_{t}^{\infty} ds (s-t) N_{0}(s,s-t)$ using the expression \eqref{un0estintstep1}. \eqref{drun0smrest} follows from direct differentiation. If $r \geq \frac{t}{2}$, we simply directly differentiate \eqref{un0def}, to get
$$\partial_{r}u_{N_{0}}(t,r) = \frac{-u_{N_{0}}(t,r)}{r} + \int_{t}^{\infty} ds \left(\frac{-1}{2r}\right) \left((r+s-t) N_{0}(s,r+s-t) + (s-t-r) N_{0}(s,|s-t-r|)\right)$$
To estimate $\partial_{t}^{j}(u_{N_{0}}(t,r) + \int_{t}^{\infty} ds (s-t) N_{0}(s,s-t))$, we use \eqref{un0estintstep1} to get
\begin{equation}\begin{split}&\partial_{t}^{j}\left(u_{N_{0}}(t,r) + \int_{t}^{\infty} ds (s-t) N_{0}(s,s-t)\right)\\
&=\partial_{t}^{j}\left(\int_{0}^{r} dw \left(\frac{-1}{2r}\right) \int_{|r-w|}^{r+w} y N_{0}(w+t,y) dy\right) + \partial_{t}^{j}\left(\int_{0}^{r} dw w N_{0}(t+w,w)\right)\\
&+\partial_{t}^{j}\left(\int_{r}^{\infty} dw \left(\frac{-1}{2r}\right) \int_{w-r}^{w+r} y (N_{0}(t+w,y) - N_{0}(t+w,w)) dy\right)\end{split}\end{equation}
We then directly estimate this expression. 

Finally, a direct computation using \eqref{drkn0est} gives
$$||N_{0}(t,r)||_{L^{2}(r^{2} dr)} \leq \frac{C \left(\sup_{x \in [T_{\lambda},t]}\sqrt{\lambda(x)}\right)^{5} \log^{10}(t)}{t^{7/2}} \left(1+\frac{\left(\sup_{x \in [T_{\lambda},t]}\sqrt{\lambda(x)}\right)^{20}}{\sqrt{t}}\right)$$
The same procedure used to establish \eqref{u3enest} then gives \eqref{un0enest}. 
\end{proof}
We define the linear error terms associated to $u_{N_{0}}$ and $v_{4}$ by
\begin{equation}\label{eun0ev4defs} e_{u_{N_{0}}}(t,r) := \frac{-45 \lambda(t)^{2} u_{N_{0}}(t,r)}{(3\lambda(t)^{2}+r^{2})^{2}}, \quad e_{v_{4}}(t,r):=\frac{-45 \lambda(t)^{2} v_{4}(t,r)}{(3\lambda(t)^{2}+r^{2})^{2}}\end{equation}
Then, let $\psi_{1} \in C^{\infty}([0,\infty))$ satisfy $0 \leq \psi_{1}(x) \leq 1$, 
\begin{equation}\label{psi1def}\psi_{1}(x) = \begin{cases}1, \quad x \leq \frac{1}{2}\\
0, \quad x \geq \frac{3}{4}\end{cases}\end{equation}
First, a direct application of Lemma \ref{un0estlemma} shows that the following piece of the linear error terms of $u_{N_{0}}$ and $v_{4}$ is perturbative. 
\begin{lemma} \begin{equation}\begin{split}||(e_{u_{N_{0}}}+e_{v_{4}}\psi_{1}(\frac{|\cdot|}{t}))\chi_{\geq 1}(\frac{4|\cdot|}{t})||_{H^{1}(\mathbb{R}^{3})} \leq \frac{C \lambda(t)^{2}}{t^{9/2}} & \left(\sup_{x \in [T_{\lambda},t]} \sqrt{\lambda(x)}\right)^{3}\end{split}\end{equation}\end{lemma}
Next, we let $u_{4}$ solve the following equation with zero Cauchy data at infinity:
\begin{equation}\label{u4eqn}-\partial_{t}^{2}u_{4}+\partial_{r}^{2}u_{4}+\frac{2}{r}\partial_{r}u_{4} = e_{v_{4}}(t,r) (1-\psi_{1}(\frac{r}{t})) \chi_{\geq 1}(\frac{4 r}{t})\end{equation}
We have
\begin{equation} \label{u4def}u_{4}(t,r) = \int_{t}^{\infty} ds \left(\frac{-1}{2r} \int_{|r-(s-t)|}^{r+s-t} y \left(e_{v_{4}}(s,y) (1-\psi_{1}(\frac{y}{s})) \chi_{\geq 1}(\frac{4 y}{s})\right) dy\right)\end{equation}
Then, we get
\begin{lemma} \label{u4estlemma}
\begin{equation}|u_{4}(t,r)| \leq \begin{cases} \frac{C \lambda(t)^{2} \left(\sup_{x \in [T_{\lambda},t]} \sqrt{\lambda(x)}\right)^{5} \log^{10}(t)}{t^{5}}, \quad r \leq \frac{t}{2}\\
\frac{C \lambda(t)^{2}\left(\sup_{x \in [T_{\lambda},t]} \sqrt{\lambda(x)}\right)^{5} \log^{10}(t)}{r t^{3}}, \quad r \geq \frac{t}{2}\end{cases}\end{equation}
$$|\partial_{r}u_{4}(t,r)| \leq \begin{cases} \frac{C \lambda(t)^{2} \left(\sup_{x \in [T_{\lambda},t]} \sqrt{\lambda(x)}\right)^{5}\log^{10}(t)}{t^{6}}, \quad r \leq \frac{t}{2}\\
\frac{C \lambda(t)^{2} \left(\sup_{x \in [T_{\lambda},r]} \sqrt{\lambda(x)}\right)^{5} \log^{10}(r)}{t^{3} r}, \quad r \geq \frac{t}{2}\end{cases}$$
$$||\partial_{t}u_{4}(t,r)||_{L^{2}(r^{2} dr)} + ||\partial_{r}u_{4}(t,r)||_{L^{2}(r^{2} dr)} \leq \frac{C \lambda(t)^{2} \left(\sup_{x \in [T_{\lambda},t]} \sqrt{\lambda(x)}\right)^{5} \log^{10}(t)}{t^{3}}$$\end{lemma}
\begin{proof} We directly estimate \eqref{u4def}, using Lemma \ref{un0estlemma}. We then directly estimate the $r$ derivative of the $u_{4}$-analog of \eqref{un0estintstep1}. The energy estimate is obtained with the same procedure as in Lemma \ref{u3estlemma}.\end{proof}
We define the linear error term of $u_{4}$ by 
$$e_{u_{4}}(t,r):= \frac{-45 \lambda(t)^{2} u_{4}(t,r)}{(3 \lambda(t)^{2}+r^{2})^{2}}$$
A direct application of Lemma \ref{u4estlemma} gives
\begin{lemma} $$||e_{u_{4}}(t,|\cdot|)||_{H^{1}(\mathbb{R}^{3})} \leq \frac{C \sqrt{\lambda(t)}(1+\lambda(t)) \left(\sup_{x \in [T_{\lambda},t]}\sqrt{\lambda(x)}\right)^{5} \log^{10}(t)}{t^{5}}$$\end{lemma}
Finally, we eliminate the remaining piece of the $v_{4}$ and $u_{N_{0}}$ linear error terms. In particular, we let $u_{N_{0},ell}$ be the following particular solution to 
\begin{equation}\label{un0elleqn}\partial_{r}^{2}u + \frac{2}{r}\partial_{r}u - V u = F(t,r):= (e_{u_{N_{0}}} + e_{v_{4}}) \chi_{\leq 1}(\frac{4r}{t})\end{equation}
$$u_{N_{0},ell}(t,r) = \phi_{0}(\frac{r}{\lambda(t)}) \int_{0}^{\frac{r}{\lambda(t)}} s^{2}\lambda(t)^{2} F(t,s\lambda(t)) e_{2}(s) ds - e_{2}(\frac{r}{\lambda(t)}) \int_{0}^{\frac{r}{\lambda(t)}} s^{2}\lambda(t)^{2} F(t,s\lambda(t)) \phi_{0}(s) ds$$
Another direct computation using Lemma \ref{un0estlemma} gives
\begin{lemma}\label{un0ellestlemma} For $j=0$ and $0 \leq k \leq 2$ or $j=1,2$ and $0 \leq k \leq 1$,
$$|\partial_{t}^{j}\partial_{r}^{k}u_{N_{0},ell}(t,r)| \leq \frac{C}{t^{j}r^{k}} \begin{cases} \frac{\text{min}\{1,r\} r^{2} \left(\sup_{x \in [T_{\lambda},t]}\sqrt{\lambda(x)}\right)^{25} \log^{10}(t)}{\lambda(t)^{2}t^{4}}\\
+\frac{C r^{4} \left(\sup_{x \in [T_{\lambda},t]}\sqrt{\lambda(x)}\right)^{3}}{\lambda(t)^{2}t^{4}} + \frac{C r^{3} \left(\sup_{x \in [T_{\lambda},t]}\sqrt{\lambda(x)}\right)^{5} (\log^{10}(t)+\log^{10}(r))}{\lambda(t)^{2}t^{4}}, \quad r \leq \lambda(t)\\
\frac{\lambda(t) \left(\sup_{x \in [T_{\lambda},t]}\sqrt{\lambda(x)}\right)^{5} \log^{11}(t)}{t^{4}} + \frac{C \text{min}\{1,r\} \left(\sup_{x \in [T_{\lambda},t]}\sqrt{\lambda(x)}\right)^{25} \log^{10}(t)}{t^{4}}\\
+\frac{C(1+\log(\frac{r}{\lambda(t)})) \lambda(t)^{2} \left(\sup_{x \in [T_{\lambda},t]}\sqrt{\lambda(x)}\right)^{3}}{t^{4}}, \quad \lambda(t) \leq r \leq \frac{t}{2}\\
\frac{C \left(\sup_{x \in [T_{\lambda},t]}\sqrt{\lambda(x)}\right)^{25} \log^{10}(t)}{t^{4}} + \frac{C \lambda(t)^{2} \left(\sup_{x \in [T_{\lambda},t]}\sqrt{\lambda(x)}\right)^{3} \log(t)}{t^{4}} \\
+\frac{C \lambda(t) \left(\sup_{x \in [T_{\lambda},t]}\sqrt{\lambda(x)}\right)^{5} \log^{11}(t)}{t^{4}}, \quad r \geq \frac{t}{2}\end{cases}$$\end{lemma}
We will insert $\psi_{1}(\frac{r}{t}) u_{N_{0},ell}(t,r)$ into our ansatz, and define its error term as
\begin{equation}\begin{split} e_{u_{N_{0},ell}}(t,r)&:= -\left(\left(-\partial_{t}^{2}+\partial_{r}^{2}+\frac{2}{r}\partial_{r} - V(t,r)\right)\left(\psi_{1}(\frac{r}{t}) u_{N_{0},ell}(t,r)\right) - F(t,r)\right)\end{split}\end{equation}
By the support properties of $\chi_{\leq 1}$, $F(t,r) =0$ for $r \geq \frac{t}{2}$, which means that $\psi_{1}(\frac{r}{t}) F(t,r) = F(t,r)$. Using this observation, we get
\begin{lemma}$$||e_{u_{N_{0},ell}}(t,|\cdot|)||_{H^{1}(\mathbb{R}^{3})} \leq \frac{C \left(\sup_{x \in [T_{\lambda},t]}\sqrt{\lambda(x)}\right)^{25} \log^{11}(t)}{t^{9/2}} \left(1+\frac{1}{\lambda(t)}\right)$$\end{lemma}
\subsection{Nonlinear error terms, part 2}
The new function added to our ansatz last section is
$$u_{n}(t,r):= u_{N_{0}}(t,r) + v_{4}(t,r) + u_{4}(t,r) + u_{N_{0},ell}(t,r) \psi_{1}(\frac{r}{t}) := u_{new}(t,r) + v_{4}(t,r)(1-\psi_{1}(\frac{r}{t}))$$
The corresponding new nonlinear interactions which need to be studied are contained in $N_{2}$, defined by
$$N_{2}(t,r):= -\left(5 \left(\left(Q_{\lambda}+u_{a}\right)^{4}-Q_{\lambda}^{4}\right) u_{n} + 10(Q_{\lambda} + u_{a})^{3} u_{n}^{2} + 10(Q_{\lambda}+u_{a})^{2} u_{n}^{3} + 5(Q_{\lambda}+u_{a}) u_{n}^{4} + u_{n}^{5}\right)$$
This is split into a perturbative piece, $N_{2,1}$, and the remainder $N_{2,0}$.
\begin{equation}\label{n20def}N_{2,0} = N_{2} \Bigr|_{u_{n} \rightarrow v_{4}(1-\psi_{1}(\frac{r}{t}))}, \quad N_{2,1} = N_{2}-N_{2,0}\end{equation}
For ease of notation, we let $\widetilde{v_{4}}(t,r):= v_{4}(t,r) (1-\psi_{1}(\frac{r}{t}))$. Then, we have
\begin{lemma} $$||N_{2,1}(t,|\cdot|)||_{H^{1}(\mathbb{R}^{3})} \leq \frac{C \log^{21}(t)}{t^{4+\delta_{2}}}, \quad \delta_{2}= \frac{1}{2}\text{min}\{\frac{1}{2}-\frac{7}{2}C_{u}-C_{l},\frac{3}{2}-\frac{29}{2} C_{u}-C_{l}\}>0$$
\end{lemma}
\begin{proof}
We note that
$$u_{new}(t,r) = u_{N_{0}}(t,r) + \psi_{1}(\frac{r}{t}) v_{4}(t,r) + u_{4}(t,r) + u_{N_{0},ell}(t,r) \psi_{1}(\frac{r}{t})$$
and recall that $\psi_{1}(\frac{r}{t}) =1$ if $\frac{r}{t} \leq \frac{1}{2}$. Then, using Lemmas \ref{un0estlemma}, \ref{u4estlemma}, and \ref{un0ellestlemma}, we get
\begin{equation}\label{unewest}|u_{new}(t,r)| \leq C \begin{cases} \frac{\text{min}\{1,r\} \left(\sup_{x \in [T_{\lambda},t]}\sqrt{\lambda(x)}\right)^{25} \log^{10}(t)}{t^{4}} + \frac{r \left(\sup_{x \in [T_{\lambda},t]}\sqrt{\lambda(x)}\right)^{5} (\log^{10}(t) + \log^{10}(r))}{t^{4}}\\
+\frac{r^{2} \left(\sup_{x \in [T_{\lambda},t]}\sqrt{\lambda(x)}\right)^{3} \log(t)}{t^{4}} + \frac{\lambda(t) \left(\sup_{x \in [T_{\lambda},t]}\sqrt{\lambda(x)}\right)^{5} \log^{11}(t)}{t^{4}}, \quad r \leq \frac{t}{2}\\
\frac{\left(\sup_{x \in [T_{\lambda},t]}\sqrt{\lambda(x)}\right)^{25} \log^{10}(t)}{r t^{3}} + \frac{\left(\sup_{x \in [T_{\lambda},t]}\sqrt{\lambda(x)}\right)^{5} \log^{11}(t)}{r t^{2}}, \quad r \geq \frac{t}{2}\end{cases}\end{equation}
Using the above estimate on $u_{new}$ and Lemma \ref{u4estlemma}, along with 
$$|N_{2,1}| \leq C |u_{new}|\left(|Q_{\lambda}|^{3} (|u_{a}|+|u_{new}| + |\widetilde{v_{4}}|) + u_{new}^{4}+\widetilde{v_{4}}^{4} + u_{a}^{4}\right)$$
a straightforward, but slightly long calculation gives
$$||N_{2,1}||_{L^{2}(r^{2} dr)} \leq \frac{C \log^{21}(t)}{t^{4+\epsilon_{2}}}, \quad \epsilon_{2} = \frac{1}{2} \text{min}\{\frac{1}{2}-\frac{7}{2}C_{u}, \frac{3}{2}-\frac{29}{2} C_{u}\} >0 $$
$\partial_{r}N_{2,1}$ is estimated with the same procedure (which is a direct, but slightly long computation).
\end{proof}
Next, we need to eliminate $N_{2,0}$. We let $u_{N_{2}}$ be the solution to the following equation, with 0 Cauchy data at infinity.
\begin{equation}\label{un2eqn}-\partial_{t}^{2}u+\partial_{r}^{2}u + \frac{2}{r}\partial_{r}u = N_{2,0}\end{equation}
Then, we have
\begin{lemma}\label{un2estlemma}
$$|u_{N_{2}}(t,r)| \leq \begin{cases} \frac{C \left(\sup_{x \in [T_{\lambda},t]}\sqrt{\lambda(x)}\right)^{9} \log^{18}(t)}{t^{5}}, \quad r \leq \frac{t}{2}\\
\frac{C}{r} \frac{\left(\sup_{x \in [T_{\lambda},t]}\sqrt{\lambda(x)}\right)^{25} \log^{50}(t)}{t^{3}}, \quad r \geq \frac{t}{2}\end{cases}$$
$$|\partial_{r}u_{N_{2}}(t,r)| \leq C \begin{cases} \frac{\left(\sup_{x \in [T_{\lambda},t]}\sqrt{\lambda(x)}\right)^{9} \log^{50}(t)}{t^{6}}, \quad r \leq \frac{t}{2}\\
\frac{\left(\sup_{x \in [T_{\lambda},r]}\sqrt{\lambda(x)}\right)^{9} \log^{50}(r)}{r} \left(\frac{\left(\sup_{x \in [T_{\lambda},r]}\sqrt{\lambda(x)}\right)^{4}}{t^{3} \langle t-r \rangle^{4}} + \frac{1}{\langle t-r \rangle^{2} t^{3}}\right.\\
\left. + \frac{\left(\sup_{x \in [T_{\lambda},r]}\sqrt{\lambda(x)}\right)^{16}}{t^{3} \langle t-r \rangle^{10}} + \frac{\left(\sup_{x \in [T_{\lambda},r]}\sqrt{\lambda(x)}\right)^{16}}{t^{4}}\right), \quad r \geq \frac{t}{2}\end{cases}$$
$$||\partial_{t}u_{N_{2}}||_{L^{2}(r^{2}dr)} + ||\partial_{r}u_{N_{2}}||_{L^{2}(r^{2} dr)} \leq \frac{C \left(\sup_{x \in [T_{\lambda},t]}\sqrt{\lambda(x)}\right)^{25} \log^{50}(t)}{t^{3}}$$
\end{lemma}
\begin{proof}
We recall \eqref{n20def} and the fact that $\psi_{1}(\frac{r}{t}) =1$ if $\frac{r}{t} \leq \frac{1}{2}$, and use Lemma \ref{un0estlemma} to get
$$|N_{2,0}(t,r)| \leq C N_{2,0,est}(t,r)$$
for
\begin{equation}\begin{split}&N_{2,0,est}(t,r)\\
&=\frac{ \mathbbm{1}_{\{r \geq \frac{t}{2}\}}\log^{10}(r)}{r \langle t-r \rangle^{2}} \left(\sup_{x \in [T_{\lambda},r]} \sqrt{\lambda(x)}\right)^{5} \\
&\cdot \begin{aligned}[t]&\left(\frac{\left(\sup_{x \in [T_{\lambda},r]} \sqrt{\lambda(x)}\right)^{8} \log^{2}(r)}{r^{4} t^{2}} + \frac{\left(\sup_{x \in [T_{\lambda},r]} \sqrt{\lambda(x)}\right)^{8} \log^{10}(r)}{r^{4} \langle t-r \rangle^{2}}+\frac{\left(\sup_{x \in [T_{\lambda},r]} \sqrt{\lambda(x)}\right)^{4} \log^{8}(r)}{r^{4}}\right.\\
&\left.+\frac{\left(\sup_{x \in [T_{\lambda},r]} \sqrt{\lambda(x)}\right)^{20} \log^{8}(r)}{r^{4}t^{8}}\right.\left. + \frac{\left(\sup_{x \in [T_{\lambda},r]} \sqrt{\lambda(x)}\right)^{20} \log^{40}(r)}{r^{4} \langle t-r \rangle^{8}}\right)\end{aligned}\end{split}\end{equation}
We also get
\begin{equation}\begin{split}|\partial_{r}N_{2,0}(t,r)| &\leq C N_{2,0,est}(t,r) \left(\frac{1}{r}+\frac{1}{\langle t-r \rangle}+\frac{1}{t}\right)\\
&+C \mathbbm{1}_{\{r \geq \frac{t}{2}\}} \left(\frac{ \left(\sup_{x \in [T_{\lambda},r]}\sqrt{\lambda(x)}\right)^{23} \log^{42}(r)}{r^{2}\langle t-r \rangle^{2}t^{5}} + \frac{\left(\sup_{x \in [T_{\lambda},t]}\sqrt{\lambda(x)}\right)^{26}\log^{40}(t) }{t^{4}} |\partial_{r}u_{3}(t,r)|\right)\end{split}\end{equation}
We note that 
\begin{equation}\label{un2def} u_{N_{2}}(t,r) = \int_{t}^{\infty} ds \left(\frac{-1}{2r} \int_{|r-(s-t)|}^{r+s-t} y N_{2,0}(s,y) dy\right)\end{equation}
We then directly estimate \eqref{un2def} for all $r>0$, and the $r$ derivative of the $u_{N_{2}}$-analog of \eqref{un0estintstep1} for $r \leq \frac{t}{2}$. When $r \geq \frac{t}{2}$, we directly estimate the $r$ derivative of \eqref{un2def}. The energy estimate is obtained with the same procedure as in Lemma \ref{u3estlemma}.\end{proof}
Next, we estimate $e_{u_{N_{2}}}$, the linear error term associated to $u_{N_{2}}$, defined by
$$e_{u_{N_{2}}}(t,r):= \frac{-45 \lambda(t)^{2} u_{N_{2}}(t,r)}{(3\lambda(t)^{2}+r^{2})^{2}}$$
A direct application of Lemma \ref{un2estlemma} gives
\begin{lemma}
$$||e_{u_{N_{2}}}(t,|\cdot|)||_{H^{1}(\mathbb{R}^{3})} \leq \frac{C \left(\sup_{x \in [T_{\lambda},t]} \sqrt{\lambda(x)}\right)^{9} \log^{50}(t)}{\sqrt{\lambda(t)} t^{5}} \left(1+\frac{1}{\lambda(t)}\right)$$\end{lemma}
At this stage, the only error terms that remain are the nonlinear interactions between $u_{N_{2}}$ and the previous corrections, which we collect together in $N_{3}$:
$$-N_{3}= 5\left(\left(Q_{\lambda}+u_{a}+u_{n}\right)^{4}-Q_{\lambda}^{4}\right) u_{N_{2}} + 10(Q_{\lambda}+u_{a}+u_{n})^{3}u_{N_{2}}^{2}+10(Q_{\lambda}+u_{a}+u_{n})^{2}u_{N_{2}}^{3}+5(Q_{\lambda}+u_{a}+u_{n})u_{N_{2}}^{4}+u_{N_{2}}^{5}$$
Then, we have
\begin{lemma}\label{n3estlemma}$$||N_{3}(t,|\cdot|)||_{H^{1}(\mathbb{R}^{3})} \leq \frac{C \log^{90}(t)}{t^{4+\frac{1}{2}(\frac{1}{2}-6C_{u})}}$$\end{lemma}
\begin{proof}
We start with
$$|N_{3}(t,r)| \leq C |u_{N_{2}}|(|Q_{\lambda}|^{3}(|u_{a}|+|u_{n}|+|u_{N_{2}}|)+u_{a}^{4}+u_{n}^{4}+u_{N_{2}}^{4})$$
Using \eqref{dtjua0est}, \eqref{drkua1est}, \eqref{unewest} and Lemmas \ref{un0estlemma}, \ref{un2estlemma}, we get
\begin{equation}\label{uansatzest}|u_{a}|+|u_{n}|+|u_{N_{2}}| \leq C \begin{cases} \frac{\log(t) \left(\sup_{x \in [T_{\lambda},t]}\sqrt{\lambda(x)}\right)^{3}}{t^{2}}+\frac{\text{min}\{1,r\} \left(\sup_{x \in [T_{\lambda},t]}\sqrt{\lambda(x)}\right)^{25} \log^{10}(t)}{t^{4}}, \quad r \leq \lambda(t)\\
\frac{\left(\sup_{x \in [T_{\lambda},t]}\sqrt{\lambda(x)}\right) r}{t^{2}}, \quad \lambda(t) \leq r \leq \frac{t}{2}\\
\frac{\left(\sup_{x \in [T_{\lambda},r]}\sqrt{\lambda(x)}\right) \log^{2}(r)}{r} + \frac{ \left(\sup_{x \in [T_{\lambda},r]}\sqrt{\lambda(x)}\right)^{5} \log^{11}(r)}{r t^{2}}+\frac{\left(\sup_{x \in [T_{\lambda},r]}\sqrt{\lambda(x)}\right)^{5} \log^{10}(r)}{r \langle t-r \rangle^{2}}, \quad r \geq \frac{t}{2}\end{cases}\end{equation}
Then, a direct computation gives
$$||N_{3}(t,r)||_{L^{2}(r^{2}dr)} \leq \frac{C \log^{90}(t)}{t^{4+\delta_{3,0}}}, \quad \delta_{3,0}=\frac{1}{2}\text{min}\{\frac{5}{2}-\frac{29}{2}C_{u},3-\frac{45}{2}C_{u}\}>0$$
$\partial_{r}N_{3}$ is estimated similarly. The only detail worth noting is that we use the energy estimates from Lemmas \ref{un2estlemma} and \ref{u3estlemma} to estimate $||\partial_{r}u_{N_{2}}||_{L^{2}((\frac{t}{2},\infty),r^{2} dr)}$ and $||\partial_{r}u_{3}||_{L^{2}((\frac{h(t)}{4},\infty),r^{2} dr)}$, and pointwise estimates on derivatives otherwise.
\end{proof}
The estimates from Lemma \ref{n3estlemma} show that $N_{3}$ can be treated perturbatively in the next section of the argument. Therefore, our final ansatz, which we denote by $u_{ansatz}$ is the sum of all terms added thus far:
$$u_{ansatz}(t,r)= u_{a}(t,r)+u_{n}(t,r)+u_{N_{2}}(t,r)$$
Now, we estimate the error term (which we denote by $F_{5}$) of $u_{ansatz}$:
\begin{equation}\label{f5def}\begin{split}F_{5}(t,r)&= e_{match}+e_{ell,2}+e_{ex}+e_{w,2}+e_{u_{3}}+(e_{u_{N_{0}}}+e_{v_{4}}\psi_{1}(\frac{r}{t}))\chi_{\geq 1}(\frac{4r}{t})\\
&+e_{u_{4}}+e_{u_{N_{0},ell}}+e_{u_{N_{2}}}+N_{1}+N_{2,1}+N_{3}\end{split}\end{equation}  
\begin{lemma}\label{f5estlemma} $$||F_{5}(t,|\cdot|)||_{H^{1}(\mathbb{R}^{3})} \leq \frac{C \log^{90}(t)}{t^{4+2 \epsilon}}$$
where $\epsilon >3 C_{l}$ and is given by
$$\epsilon = \frac{1}{2}\text{min}\{\delta,4-\frac{13}{2}a(1-C_{u}) -7C_{u},-2+\frac{7}{2}a (1-C_{u}) -\frac{5}{2}C_{u}-\frac{3}{2}C_{l},\frac{a}{2}(1-C_{u})-4C_{u},\frac{1}{2}-C_{l}-\frac{25}{2}C_{u},\delta_{1},\delta_{2},\frac{1}{2}(\frac{1}{2}-6C_{u})\}$$\end{lemma}
For later use, we decompose $\partial_{r}u_{ansatz}$ as follows.
\begin{equation}\partial_{r}u_{ansatz}(t,r) =u_{an,r,0}(t,r)+u_{an,r,1}(t,r)\end{equation}
where
\begin{equation}\label{uanr0def} u_{an,r,1}(t,r) = \partial_{r}u_{N_{0}}(t,r) + \partial_{r}u_{N_{2}}(t,r)\end{equation}
\begin{lemma} 
\begin{equation}\label{uaninftyest}||u_{ansatz}(t,r)||_{L^{\infty}_{r}} \leq \frac{C \left(\sup_{x \in [T_{\lambda},t]}\sqrt{\lambda(x)}\right)^{5} \log^{10}(t)}{t}\end{equation}
\begin{equation}\label{drqq2uaninftyest} ||\partial_{r}(Q_{\lambda}(r)) Q_{\lambda}(r)^{2} u_{ansatz}||_{L^{\infty}_{r}} + ||Q_{\lambda}(r)^{3}u_{an,r,0}(t,r)||_{L^{\infty}_{r}} \leq \frac{C}{t^{2}\lambda(t)}\end{equation}
\begin{equation}\label{bracket3uanr0inftyest} ||\frac{1}{\langle \frac{r}{\lambda(t)}\rangle^{3}} u_{an,r,0}||_{L^{\infty}_{r}} \leq \frac{C \left(\sup_{x \in [T_{\lambda},t]}\sqrt{\lambda(x)}\right)}{t^{2}}\end{equation}
\begin{equation}\label{uan3uanr0inftyest} ||u_{ansatz}^{3} u_{an,r,0}||_{L^{\infty}_{r}} \leq \frac{C \left(\sup_{x \in [T_{\lambda},t]}\sqrt{\lambda(x)}\right)^{25} \log^{40}(t)}{t^{4}}\end{equation}
\begin{equation}\label{q2uansatz2linftyest} ||Q_{\lambda}(r)^{2}u_{ansatz}(t,r)^{2}||_{L^{\infty}_{r}} \leq C \left(\frac{\log^{2}(t) \left(\sup_{x \in [T_{\lambda},t]}\sqrt{\lambda(x)}\right)^{6}}{\lambda(t)t^{4}} + \frac{\lambda(t) \left(\sup_{x \in [T_{\lambda},t]}\sqrt{\lambda(x)}\right)^{10} \log^{22}(t)}{t^{4}}\right)\end{equation}
\begin{equation}\label{uanr12est}||u_{an,r,1}(t,r)||_{L^{2}(r^{2} dr)} \leq \frac{C \left(\sup_{x \in [T_{\lambda},t]}\sqrt{\lambda(x)}\right)^{5} \log^{10}(t)}{t^{5/2}}\end{equation}
\begin{equation}\label{q3uaninftyest}\begin{split}||Q_{\lambda}(r)^{3} u_{ansatz}(t,r)||_{L^{\infty}_{r}}&\leq \frac{\left(\lambda'(t)\right)^{2}}{\lambda(t)^{2}} \frac{3 (111-45 \log(4))}{16}+\frac{|\lambda''(t)|}{\lambda(t)} \frac{15}{16} \left(-12+12\sqrt{3}+15\pi^{2}-4 \log(4)\right)\\
&+\frac{C \log(t) \left(\sup_{x \in [T_{\lambda},t]}\sqrt{\lambda(x)}\right)^{4}}{h(t)^{2} t^{2}} + \frac{C \left(\sup_{x \in [T_{\lambda},t]}\sqrt{\lambda(x)}\right)^{5} \log^{10}(t)}{t^{3}\lambda(t)^{3/2}}\end{split}\end{equation}
\end{lemma}

\begin{proof}
The proof is a direct calculation. In particular, we use \eqref{uansatzest} for \eqref{uaninftyest} and \eqref{q2uansatz2linftyest}. We use \eqref{uellexp} (which gives the precise constants in \eqref{q3uaninftyest}) along with Lemmas \ref{uell2minusuell2mainestlemma}, \ref{uell3estlemma}, \ref{w1estlemma}, \ref{vexestlemma}, \ref{v2estlemma}, \ref{uw2estlemma}, \ref{v30estlemma}, \ref{v3minusv3mainestlemma}, \ref{un0estlemma}, \ref{u4estlemma}, \ref{un0ellestlemma}, \ref{un2estlemma} for \eqref{q3uaninftyest}, \eqref{drqq2uaninftyest}, \eqref{uan3uanr0inftyest}. Lemmas \ref{un0estlemma} and \ref{un2estlemma} give \eqref{uanr12est}.
\end{proof}
\section{Constructing the Exact Solution}
If we substitute $u(t,r) = Q_{\lambda(t)}(r)+u_{ansatz}(t,r)+v(t,r)$ into \eqref{slw}, we get
\begin{equation}\label{veqn}-\partial_{t}^{2}v+\partial_{r}^{2}v+\frac{2}{r}\partial_{r}v+\frac{45 \lambda(t)^2}{\left(3 \lambda(t)^2+r^2\right)^2}v = F_{5}(t,r)+F_{3}(t,r)\end{equation}
where
\begin{equation}\label{f3def}F_{3}= L_{1}(v) + N(v), \quad L_{1}(v) = (-5(Q_{\lambda}+u_{ansatz})^{4}+5 Q_{\lambda}^{4})v\end{equation}
$$N(v) = -10(Q_{\lambda}+u_{ansatz})^{3}v^{2}-10(Q_{\lambda}+u_{ansatz})^{2}v^{3}-5(Q_{\lambda}+u_{ansatz})v^{4}-v^{5}$$
We proceed as follows. First, we formally derive the equation solved by $y$ defined by
\begin{equation}\label{yintermsofv}y(t,\xi) = (y_{0}(t),y_{1}(t,\xi \lambda(t)^{-2})) = \mathcal{F}\left(\left(\cdot\right) v(t,\cdot\lambda(t))\right)(\xi)\end{equation}
where $\mathcal{F}$ is the distorted Fourier transform from \cite{kstslw}, which we regard as a two-component vector, as in Proposition 4.3 of \cite{kstslw}. Then, we prove that this formally derived equation has a solution, say $(y_{6,0},y_{6,1})$, which is regular enough to allow us to justify the statement that the function $v$ defined by the following expression, with $y_{0}=y_{6,0}, y_{1}=y_{6,1}$
\begin{equation}\label{vintermsofy}v(t,r) = \frac{\lambda(t)}{r}\left(y_{0}(t) \phi_{d}(\frac{r}{\lambda(t)})+\int_{0}^{\infty} \phi(\frac{r}{\lambda(t)},\xi) y_{1}(t,\frac{\xi}{\lambda(t)^{2}}) \rho(\xi) d\xi\right)\end{equation}
is a solution to \eqref{veqn}. Letting $v$ be as in \eqref{vintermsofy}, we formally get that \eqref{veqn} is equivalent to the system consisting of the following two equations. We remind the reader that the first component of the distorted Fourier transform of a function is a real number, while the second component is a function of frequency.
\begin{equation}\label{yeqn}\begin{split}&-y_{0}''(t)-\frac{\lambda'(t)}{\lambda(t)}y_{0}'(t)-\frac{\xi_{d}}{\lambda(t)^{2}}y_{0} = F_{2,0}(t)-\frac{\lambda'(t)}{\lambda(t)} y_{0}'(t) +\mathcal{F}\left(\left(\cdot\right)(F_{5}+F_{3})(t,\cdot\lambda(t))\right)_{0}\\
&-\partial_{t}^{2}y_{1}-\omega y_{1} = F_{2,1}(t,\omega \lambda(t)^{2}) + \mathcal{F}\left(\left(\cdot\right)(F_{5}+F_{3})(t,\cdot\lambda(t))\right)_{1}(\omega\lambda(t)^{2}) \end{split}\end{equation}
where
\begin{equation}\label{F2def}\begin{split}F_{2}(t,\eta) &= \begin{bmatrix} y_{0}(t)\\
y_{1}(t,\frac{\eta}{\lambda(t)^{2}})\end{bmatrix} \frac{\lambda''(t)}{\lambda(t)} +\frac{2\lambda'(t)}{\lambda(t)} \begin{bmatrix}y_{0}'(t)\\
\partial_{1}y_{1}(t,\frac{\eta}{\lambda(t)^{2}})\end{bmatrix}- 2 \left(\frac{\lambda'(t)}{\lambda(t)}\right) \mathcal{K}\left(\begin{bmatrix} y_{0}'(t)\\
\partial_{1}y_{1}(t,\frac{\cdot}{\lambda(t)^{2}})\end{bmatrix}\right)\\
&-\left(\frac{\lambda''(t)}{\lambda(t)}+\left(\frac{\lambda'(t)}{\lambda(t)}\right)^{2}\right) \mathcal{K}\left(\begin{bmatrix} y_{0}(t)\\
y_{1}(t,\frac{\cdot}{\lambda(t)^{2}})\end{bmatrix}\right)+2 \left(\frac{\lambda'(t)}{\lambda(t)}\right)^{2} [\mathcal{K},\xi \partial_{\xi}]\left(\begin{bmatrix} y_{0}(t)\\
y_{1}(t,\frac{\cdot}{\lambda(t)^{2}})\end{bmatrix}\right)\\
&+ \left(\frac{\lambda'(t)}{\lambda(t)}\right)^{2} \mathcal{K}\left(\mathcal{K}\left(\begin{bmatrix} y_{0}(t)\\
y_{1}(t,\frac{\cdot}{\lambda(t)^{2}})\end{bmatrix}\right)\right)\\
&:= \begin{bmatrix} F_{2,0}(t)\\
F_{2,1}(t,\eta)\end{bmatrix}\end{split}\end{equation}
and we use the notation $x_{i}$ to denote the $i+1$ entry of the vector $x=\begin{bmatrix}x_{0}\\
x_{1}\end{bmatrix}$ .
\subsection{The iteration space}\label{iterationspacesection}
The space in which we will solve \eqref{yeqn} is defined as follows. We define $Z$ to be the set of elements $\begin{bmatrix} y_{0}(t)\\
y_{1}(t,\omega)\end{bmatrix}$ such that $y_{0}:[T_{0},\infty) \rightarrow \mathbb{R}$, and $y_{1}$ is an (equivalence class) of measurable functions, $y_{1}:[T_{0},\infty) \times (0,\infty) \rightarrow \mathbb{R}$ satisfying
$$y_{0}(t) \in C^{1}_{t}([T_{0},\infty))$$
$$y_{1}(t,\omega) \frac{t^{2+\epsilon-\frac{3}{2}C_{l}}}{\log^{90}(t)}\lambda(t) \langle \omega \lambda(t)^{2}\rangle \sqrt{\rho(\omega \lambda(t)^{2})} \in C^{0}_{t}([T_{0},\infty), L^{2}(d\omega))$$
$$\partial_{t}y_{1}(t,\omega) \frac{t^{3+\epsilon-\frac{3}{2}C_{l}}}{\log^{90}(t)} \lambda(t) \langle \sqrt{\omega \lambda(t)^{2}}\rangle \sqrt{\rho(\omega \lambda(t)^{2})} \in C^{0}_{t}([T_{0},\infty), L^{2}(d\omega))$$
and $||\begin{bmatrix} y_{0}\\
y_{1}\end{bmatrix}||_{Z} < \infty$
where
\begin{equation}\label{znormdef}\begin{split}||\begin{bmatrix} y_{0}\\
y_{1}\end{bmatrix}||_{Z} &= \sup_{t \geq T_{0}}\left(\frac{t^{2+\epsilon-\frac{3}{2}C_{l}}}{\log^{90}(t)} \left(|y_{0}(t)| +\lambda(t) || \langle \omega \lambda(t)^{2}\rangle y_{1}(t,\omega)||_{L^{2}(\rho(\omega \lambda(t)^{2}) d\omega)} \right) \right.\\
&\left. + \frac{t^{3+\epsilon-\frac{3}{2}C_{l}}}{\log^{90}(t)}\left(|y_{0}'(t)|+ \lambda(t) ||\partial_{t}y_{1}(t,\omega) ||_{L^{2}(\rho(\omega \lambda(t)^{2}) d\omega)}+C_{Z}^{-1}\lambda(t) ||\sqrt{\omega \lambda(t)^{2}} \partial_{t}y_{1}(t,\omega)||_{L^{2}(\rho(\omega\lambda(t)^{2})d\omega)}\right)\right)\end{split}\end{equation} 
and $C_{Z}>0$ is otherwise arbitrary, and will be further constrained later on. Then, $(Z,||\cdot||_{Z})$ is a normed vector space. We remark that this space is similar to the iteration space used in a Yang-Mills paper of the author \cite{ym}. By directly estimating \eqref{F2def}, and using the boundedness of $\mathcal{K}$ and $[\mathcal{K},\xi\partial_{\xi}]$ from Proposition 5.2 of \cite{kstslw}, we obtain the following lemma.

\begin{lemma} \label{f2estlemma}
There exists $C>0$, independent of $T_{0}$ and $C_{Z}$, such that, for all $y \in Z$,
\begin{equation}\label{f21est}\begin{split}&|F_{2,0}(t)|+ \lambda(t) ||F_{2,1}(t,\omega\lambda(t)^{2})||_{L^{2}(\rho(\omega \lambda(t)^{2}) d\omega)}\\
&\leq2 \cdot \frac{\log^{90}(t) ||y||_{Z}}{t^{4+\epsilon-\frac{3}{2}C_{l}}} \begin{aligned}[t]&\left(C_{2}(1+||\mathcal{K}||_{\mathcal{L}(L^{2,0}_{\rho})})+M^{2}\left(||\mathcal{K}||_{\mathcal{L}(L^{2,0}_{\rho})}+2||[\mathcal{K},\xi\partial_{\xi}]||_{\mathcal{L}(L^{2,0}_{\rho})} + ||\mathcal{K}||_{\mathcal{L}(L^{2,0}_{\rho})}^{2}\right)\right.\\
&\left.+2 M \left(1+||\mathcal{K}||_{\mathcal{L}(L^{2,0}_{\rho})}\right)\right)\end{aligned}\end{split}\end{equation}
\begin{equation}\begin{split}\label{f21weightest}\lambda(t)||\sqrt{\omega \lambda(t)^{2}} F_{2,1}(t,\omega\lambda(t)^{2})||_{L^{2}(\rho(\omega\lambda(t)^{2})d\omega)} &\leq \frac{C \log^{90}(t) ||y||_{Z}}{t^{4+\epsilon-\frac{3}{2}C_{l}}}\\
&+\frac{2 M ||y||_{Z} \log^{90}(t)}{t^{4+\epsilon-\frac{3}{2}C_{l}}} \left(C_{Z}(1+||\mathcal{K}||_{\mathcal{L}(L^{2,\frac{1}{2}}_{\rho})})+||\mathcal{K}||_{\mathcal{L}(L^{2,\frac{1}{2}}_{\rho})}\right) \end{split}\end{equation}

\end{lemma}

Next, we understand the relation between $v$ and $y$ in \eqref{vintermsofy} by proving the following.
\begin{lemma}\label{vytranslemma} For all $y_{d} \in \mathbb{R}$, $y: (0,\infty)\rightarrow \mathbb{R}$ satisfying $y(\xi) \langle \xi \rangle \in L^{2}(\rho(\xi) d\xi)$, if
\begin{equation}\label{vfromy} w(R) =\frac{1}{R}\left( y_{d} \phi_{d}(R) + \int_{0}^{\infty} \phi(R,\xi)y(\xi) \rho(\xi) d\xi\right), \quad R>0\end{equation}
then, for $R>0$,
\begin{equation}\label{voverrest}|w(R)| \leq \frac{C}{\langle R \rangle}\left(|y_{d}| +||\langle \xi \rangle y(\xi)||_{L^{2}(\rho(\xi) d\xi)}\right)\end{equation}
\begin{equation}\label{l2isom} || w(R)||_{L^{2}(R^{2} dR)} = \left(|y_{d}|^{2}+|| y(\xi)||_{L^{2}( \rho(\xi) d\xi)}^{2}\right)^{1/2}\end{equation}
\begin{equation}\label{sobolevtrans}C^{-1} ||w(|\cdot|)||_{H^{1}(\mathbb{R}^{3})}  \leq  \left(|y_{d}|^{2} + ||y(\xi) \sqrt{\langle \xi \rangle}||_{L^{2}(\rho(\xi) d\xi)}^{2}\right)^{1/2} \leq C ||w(|\cdot|)||_{H^{1}(\mathbb{R}^{3})} \end{equation}
\end{lemma}
\begin{proof}
From Proposition 4.5 of \cite{kstslw}, there exists some constant $C_{1}>1$ such that, for all $R^{2}\xi \geq C_{1}$,
$$\phi(R,\xi) = 2 \text{Re}\left(a(\xi) f_{+}(R,\xi)\right)$$
where
$$f_{+}(R,\xi) = e^{i R \sqrt{\xi}} \sigma(R \sqrt{\xi},R)$$
Using Proposition 4.5 of \cite{kstslw}, we have
$$|\phi(R,\xi)| \leq C|a(\xi)|, \quad R^{2}\xi \geq C_{1}$$
When $R^{2}\xi \leq C_{1}$, we use Proposition 4.4 of \cite{kstslw} to write
$$\phi(R,\xi) = \widetilde{\phi_{0}}(R) + \frac{1}{R}\sum_{j=1}^{\infty} (R^{2}z)^{j} \phi_{j}(R^{2})$$
where 
$$|\phi_{j}(R)| \leq \frac{C^{j} R}{(j-1)! \langle R \rangle^{1/2}}$$
This gives, for $R^{2}\xi \leq C_{1}$,
$$|\phi(R,\xi)| \leq \frac{C R}{\langle R \rangle}$$
Next, we use the discussion preceding Lemma 4.2 of \cite{kstslw}, which notes that $\phi_{d} \in C^{\infty}([0,\infty))$, $\phi_{d}(0)=0$, and $\phi_{d}$ is exponentially decaying.
Directly inserting these estimates into \eqref{vfromy}, gives \eqref{voverrest}.  Finally, \eqref{l2isom} is the $L^{2}$ isometry property of $\mathcal{F}$, and \eqref{sobolevtrans} follows from Lemma 2.7 of \cite{kstslw}.
\end{proof}
Now, we can estimate $F_{3}$ (defined in \eqref{f3def}).
\begin{lemma} There exists $C>0$ so that for all $y_{1},y_{2} \in \overline{B_{1}(0)} \subset Z$ of the form 
$$y_{i} := \begin{bmatrix} y_{0,i}(t)\\
y_{1,i}(t,\omega)\end{bmatrix} , \quad i=1,2$$
if 
$$v_{i}(t,r) = \frac{\lambda(t)}{r}\left(y_{0,i}(t) \phi_{d}(\frac{r}{\lambda(t)})+\int_{0}^{\infty} \phi(\frac{r}{\lambda(t)},\xi) y_{1,i}(t,\frac{\xi}{\lambda(t)^{2}}) \rho(\xi) d\xi\right), \quad r>0$$
then,
\begin{equation}\label{f3estlemma}\begin{split}||F_{3}(v_{2})-F_{3}(v_{1})||_{L^{2}(r^{2} dr)} & \leq \frac{2 \cdot \log^{90}(t)}{t^{2+\epsilon-\frac{3}{2}C_{l}}} \lambda(t)^{3/2} ||y_{1}-y_{2}||_{Z}\\
&\cdot\begin{aligned}[t]&\left(20||Q_{\lambda}(r)^{3}u_{ansatz}||_{L^{\infty}}\right.\\
&\left.+C\left(||Q_{\lambda}(r)^{2}u_{ansatz}^{2}||_{L^{\infty}}+||u_{ansatz}||_{L^{\infty}}^{4} + \frac{\log^{360}(t)}{t^{8+4\epsilon-6C_{l}}} + \frac{\log^{90}(t)}{\lambda(t)^{3/2} t^{2+\epsilon-\frac{3}{2}C_{l}}}\right)\right)\end{aligned}\end{split}\end{equation}
\begin{equation}\begin{split}&||\partial_{r}(F_{3}(v_{2})-F_{3}(v_{1}))||_{L^{2}(r^{2}dr)}\\
&\leq \frac{C \log^{90}(t) ||y_{1}-y_{2}||_{Z}}{t^{2+\epsilon-\frac{3}{2}C_{l}}} \begin{aligned}[t]&(\lambda(t)^{3/2}\begin{aligned}[t]&\left(||\partial_{r}Q_{\lambda}Q_{\lambda}^{2}u_{ansatz}||_{L^{\infty}_{r}} + \frac{||Q_{\lambda}^{3}u_{ansatz}||_{L^{\infty}_{r}}}{\lambda(t)} + ||\partial_{r}Q_{\lambda}u_{ansatz}^{3}||_{L^{\infty}_{r}}\right.\\
&\left.+||Q_{\lambda}^{3} u_{an,r,0}||_{L^{\infty}_{r}} + ||u_{ansatz}^{3}u_{an,r,0}||_{L^{\infty}_{r}} + \sqrt{\lambda(t)} ||u_{ansatz}||_{L^{\infty}_{r}}^{4}\right.\\
&\left.+\frac{\left(||Q_{\lambda}^{3} u_{an,r,1}||_{L^{2}(r^{2} dr)} + ||u_{ansatz}||_{L^{\infty}_{r}}^{3} ||u_{an,r,1}||_{L^{2}(r^{2} dr)}\right)}{\lambda(t)^{3/2}}\right)\end{aligned}\\
&\left.+\frac{\log^{270}(t)}{t^{6+3\epsilon-\frac{9}{2}C_{l}}} \left(\lambda(t)^{3/2} ||\frac{u_{an,r,0}}{\langle \frac{r}{\lambda(t)}\rangle^{3}}||_{L^{\infty}_{r}} + ||u_{an,r,1}||_{L^{2}(r^{2} dr)}\right)\right.\\
&+\frac{\log^{360}(t) \sqrt{\lambda(t)}}{t^{8+4\epsilon-6C_{l}}} + \frac{\log^{90}(t)}{t^{2+\epsilon-\frac{3}{2}C_{l}}}\left(\frac{1}{\lambda(t)} + \sqrt{\lambda(t)} ||u_{ansatz}||_{L^{\infty}_{r}}^{3}\right))\end{aligned}\end{split}\end{equation}

\end{lemma}
We will also need a simple estimate regarding the density of the spectral measure, $\rho(\xi)$, which is proven by a direct application of Lemma 4.6 of \cite{kstslw}.
\begin{lemma}\label{rhoscalinglemma} There exists $C_{\rho}>0$ (independent of $T_{0}$) so that
$$\frac{\rho(\omega\lambda(x)^{2})}{\rho(\omega\lambda(t)^{2})} \leq C_{\rho} \left(\frac{\lambda(x)}{\lambda(t)}+\frac{\lambda(t)}{\lambda(x)}\right), \quad x,t \geq T_{0}$$
\end{lemma}

Define 
\begin{equation}\label{tdef}T\left(\begin{bmatrix}y_{0}\\
y_{1}\end{bmatrix}\right)(t,\omega) = \begin{bmatrix} e_{+}(t) \int_{t}^{\infty} \frac{F_{0}(w) \lambda(w)}{2 \sqrt{-\xi_{d}}} e_{-}(w) dw + e_{-}(t) \int_{T_{0}}^{t} \frac{F_{0}(w)\lambda(w)}{2 \sqrt{-\xi_{d}}} e_{+}(w) dw\\
\int_t^{\infty } \frac{\sin \left(\sqrt{\omega} (t-s)\right)}{\sqrt{\omega}} F_{1}(s,\omega)  ds \end{bmatrix}, \quad \begin{bmatrix}y_{0}\\
y_{1}\end{bmatrix} \in \overline{B_{1}(0)} \subset Z\end{equation}
where
$$F_{i}(t,\omega) =\begin{cases}\mathcal{F}\left(\left(\cdot\right)(F_{5}+F_{3})(t,\cdot\lambda(t))\right)_{0}+  F_{2,0}(t)-\frac{\lambda'(t)}{\lambda(t)} y_{0}'(t),\text{  } i=0\\
\mathcal{F}\left(\left(\cdot\right)(F_{5}+F_{3})(t,\cdot\lambda(t))\right)_{1}(\omega\lambda(t)^{2})+ F_{2,1}(t,\omega\lambda(t)^{2}),\text{  } i=1\end{cases}$$
$$e_{\pm}(t) = \exp \left(\pm \sqrt{-\xi_{d}} \int_{T_{\lambda}}^t \frac{1}{\lambda(s)} \, ds\right)$$
and $F_{3}$ is given by \eqref{f3def}, with $v$ given by \eqref{vintermsofy}. Lemmas \ref{f2estlemma}, \ref{f3estlemma}, \ref{f5estlemma}, and \eqref{cofconstr1} directly give
\begin{lemma}\label{rhslemma} There exists $C>0$ (\emph{independent} of $T_{0}$) so that, for all $y \in \overline{B}_{1}(0)$, and $s , t \geq T_{0}$,
$$|F_{0}(t)| \leq \frac{C \log^{90}(t)}{t^{4+\epsilon-\frac{3}{2}C_{l}}}$$
\begin{equation}\begin{split}&\lambda(s) ||F_{1}(s,\omega)||_{L^{2}(\rho(\omega \lambda(s)^{2}) d\omega)}< \frac{\log^{90}(s)}{s^{4+\epsilon-\frac{3}{2}C_{l}}}\begin{aligned}[t]&\left(\frac{C \log^{10}(s)}{s^{1-\frac{5}{2}C_{u}-\frac{3}{2}C_{l}}}+\frac{C \log^{90}(s)}{s^{\epsilon-3C_{l}}}\right.\left.+\frac{1}{12 \sqrt{C_{\rho}}}\right)\end{aligned}\end{split}\end{equation}
There exists $C_{1}>0$ (\emph{independent} of $T_{0}$ and $C_{Z}$) so that, for all $y \in \overline{B}_{1}(0)$, and $s  \geq T_{0}$,
\begin{equation} \lambda(s)||\sqrt{\omega\lambda(s)^{2}} F_{1}(s,\omega)||_{L^{2}(\rho(\omega\lambda(s)^{2})d\omega)} \leq \frac{C_{1} \log^{90}(s)}{s^{4+\epsilon-\frac{3}{2}C_{l}}}+ \frac{C_{Z} \log^{90}(s)}{24 \sqrt{C_{\rho}}s^{4+\epsilon-\frac{3}{2}C_{l}}}\end{equation}
\end{lemma}
We proceed to show that $T$ maps $\overline{B_{1}(0)} \subset Z$ into itself. If $\begin{bmatrix}y_{0}\\
y_{1}\end{bmatrix} \in \overline{B_{1}(0)} \subset Z$, we have
\begin{equation}|T(\begin{bmatrix}y_{0}\\
y_{1}\end{bmatrix})_{0}| \leq e_{+}(t)\int_{t}^{\infty}\frac{|F_{0}(w)| \lambda(w)}{2 \sqrt{-\xi_{d}}} e_{-}(w) dw + e_{-}(t) \int_{T_{0}}^{t} \frac{|F_{0}(w)| \lambda(w)}{2 \sqrt{-\xi_{d}}} e_{+}(w) dw\end{equation}
From Lemma \ref{rhslemma}, we thus get
\begin{equation}\begin{split}&|T(\begin{bmatrix}y_{0}\\
y_{1}\end{bmatrix})_{0}| \leq C\left(e_{+}(t)\int_{t}^{\infty}\frac{\log^{90}(w) \lambda(w)}{w^{4+\epsilon-\frac{3}{2}C_{l}}} e_{-}(w) dw + e_{-}(t) \int_{T_{0}}^{t} \frac{\log^{90}(w) \lambda(w)}{w^{4+\epsilon-\frac{3}{2}C_{l}} } e_{+}(w) dw\right):=C I(t)\end{split}\end{equation}
Integrating by parts (where we integrate $\frac{\pm \sqrt{-\xi_{d}}}{\lambda(w)} e_{\pm}(w)$ and differentiate the rest of the integrand), we get
\begin{equation}\begin{split} I(t) &\leq e_{+}(t)\left(\frac{\log^{90}(t) \lambda(t)^{2} e_{-}(t)}{\sqrt{-\xi_{d}} t^{4+\epsilon-\frac{3}{2}C_{l}}} + \int_{t}^{\infty} dw e_{-}(w)\frac{C}{w^{1-C_{u}}} \frac{\log^{90}(w) \lambda(w)}{w^{4+\epsilon-\frac{3}{2}C_{l}}}\right)\\
&+e_{-}(t) \left(\frac{\log^{90}(t) \lambda(t)^{2} e_{+}(t)}{t^{4+\epsilon-\frac{3}{2}C_{l}}\sqrt{-\xi_{d}}}+ \int_{T_{0}}^{t} e_{+}(w) \frac{C}{w^{1-C_{u}}} \frac{\log^{90}(w) \lambda(w)}{w^{4+\epsilon-\frac{3}{2}C_{l}}}dw\right)\end{split}\end{equation}
where $C$ is \emph{independent} of $T_{0}$. To get this, we also used
\begin{equation} \frac{\log^{90}(w)\lambda(w)^{2}}{w^{4+\epsilon-\frac{3}{2}C_{l}}\sqrt{-\xi_{d}}} e_{+}(w)\Bigr|_{w=T_{0}}^{w=t} \leq \frac{\log^{90}(t)\lambda(t)^{2}}{t^{4+\epsilon-\frac{3}{2}C_{l}}\sqrt{-\xi_{d}}} e_{+}(t)\end{equation}
In other words, the boundary term at $T_{0}$ is negative. We therefore get, for $C$ independent of $T_{0}$,
\begin{equation}I(t) \leq \frac{C \log^{90}(t) \lambda(t)^{2}}{t^{4+\epsilon-\frac{3}{2}C_{l}}} + \frac{C I(t)}{T_{0}^{1-C_{u}}}, \quad t \geq T_{0}\end{equation}
Since $1-C_{u}>0$, there exists $T_{3}>0$ so that, for all $T_{0} \geq T_{3}$, 
$$I(t) \leq \frac{C_{2} \log^{90}(t) \lambda(t)^{2}}{t^{4+\epsilon-\frac{3}{2}C_{l}}}, \quad t \geq T_{0}$$
(where $C_{2}$ is independent of $T_{0}$). From here on, we further restrict $T_{0}$ to satisfy $T_{0} \geq T_{3}$. Since 
$$T(\begin{bmatrix}y_{0}\\
y_{1}\end{bmatrix})_{0}'(t) = \frac{\sqrt{-\xi_{d}}}{\lambda(t)}\left(e_{+}(t) \int_{t}^{\infty} \frac{F_{0}(w) \lambda(w) e_{-}(w) dw}{2 \sqrt{-\xi_{d}}}-e_{-}(t) \int_{T_{0}}^{t} \frac{F_{0}(w) \lambda(w) e_{+}(w) dw}{2 \sqrt{-\xi_{d}}}\right)$$
we get
$$|T(\begin{bmatrix}y_{0}\\
y_{1}\end{bmatrix})_{0}| + \lambda(t) |T(\begin{bmatrix}y_{0}\\
y_{1}\end{bmatrix})_{0}'(t)| \leq \frac{C \log^{90}(t) \lambda(t)^{2}}{t^{4+\epsilon-\frac{3}{2}C_{l}}}$$
Using Lemma \ref{rhoscalinglemma} and \eqref{lambdacomparg}, we get
\begin{equation}\begin{split}\lambda(t) ||T(\begin{bmatrix}y_{0}\\
y_{1}\end{bmatrix})_{1}(t,\omega)||_{L^{2}(\rho(\omega\lambda(t)^{2}) d\omega)} &\leq \sqrt{C_{\rho}} \int_{t}^{\infty} ds (s-t) ||\lambda(s) F_{1}(s,\omega)||_{L^{2}(\rho(\omega\lambda(s)^{2})d\omega)} \left(\frac{\sqrt{\lambda(t)}}{\sqrt{\lambda(s)}} + \frac{\lambda(t)^{3/2}}{\lambda(s)^{3/2}}\right)\\
&\leq 2 \sqrt{C_{\rho}} \int_{t}^{\infty} ds (s-t) ||\lambda(s) F_{1}(s,\omega)||_{L^{2}(\rho(\omega\lambda(s)^{2})d\omega)} \left(\frac{s}{t}\right)^{\frac{3 C_{l}}{2}}\end{split}\end{equation}
Using Lemma \ref{rhslemma}, we get
\begin{equation}\begin{split} &\lambda(t) ||T(\begin{bmatrix}y_{0}\\
y_{1}\end{bmatrix})_{1}(t,\omega)||_{L^{2}(\rho(\omega\lambda(t)^{2}) d\omega)} < \frac{2 \sqrt{C_{\rho}}}{t^{\epsilon-\frac{3}{2}C_{l}}} \left(\frac{C \log^{100}(t)}{t^{3-\frac{5}{2}C_{u}-\frac{3}{2}C_{l}}}+\frac{C \log^{180}(t)}{t^{2+\epsilon-3 C_{l}}}+\frac{\log^{90}(t)}{12 t^{2} \sqrt{C_{\rho}}}\right)\end{split}\end{equation}
We remark that the constant 2 multiplying $\sqrt{C_{\rho}}$ is not sharp. For instance, the integral was, in part, estimated by using the fact that $s \mapsto \frac{\log^{180}(s)}{s}$ is decreasing on $[T_{0},\infty)$. We get
\begin{equation}\label{Tl2est}\frac{t^{2+\epsilon-\frac{3}{2}C_{l}} \lambda(t)}{\log^{90}(t)} ||T(\begin{bmatrix}y_{0}\\
y_{1}\end{bmatrix})_{1}(t,\omega)||_{L^{2}(\rho(\omega\lambda(t)^{2}) d\omega)} < \frac{1}{6}+C\left(\frac{\log^{10}(t)}{t^{1-\frac{5}{2}C_{u}-\frac{3}{2}C_{l}}}+\frac{\log^{90}(t)}{t^{\epsilon-3C_{l}}}\right)\end{equation}
Also,
\begin{equation}\lambda(t) ||\partial_{t}T(\begin{bmatrix}y_{0}\\
y_{1}\end{bmatrix})_{1}(t,\omega)||_{L^{2}(\rho(\omega\lambda(t)^{2}) d\omega)} \leq 2 \sqrt{C_{\rho}} \int_{t}^{\infty} \lambda(s) ||F_{1}(s,\omega)||_{L^{2}(\rho(\omega\lambda(s)^{2})d\omega)} \left(\frac{s}{t}\right)^{\frac{3}{2}C_{l}} ds\end{equation}
Therefore, the identical argument used to obtain \eqref{Tl2est} gives
$$\frac{t^{3+\epsilon-\frac{3}{2}C_{l}} \lambda(t)}{\log^{90}(t)} ||\partial_{t}T(\begin{bmatrix}y_{0}\\
y_{1}\end{bmatrix})_{1}(t,\omega)||_{L^{2}(\rho(\omega\lambda(t)^{2}) d\omega)} < \frac{1}{6} + C \left(\frac{\log^{10}(t)}{t^{1-\frac{5}{2}C_{u}-\frac{3}{2}C_{l}}} + \frac{\log^{90}(t)}{t^{\epsilon-3 C_{l}}}\right)$$
Next, using Lemmas \ref{rhoscalinglemma} and \ref{rhslemma}, we get
\begin{equation}\begin{split}&C_{Z}^{-1} \lambda(t) ||\sqrt{\omega\lambda(t)^{2}} \partial_{t}T(\begin{bmatrix}y_{0}\\
y_{1}\end{bmatrix})_{1}(t,\omega)||_{L^{2}(\rho(\omega\lambda(t)^{2})d\omega)}\\
&\leq \sqrt{C_{\rho}} C_{Z}^{-1} \int_{t}^{\infty} \lambda(s) ||\sqrt{\omega\lambda(s)^{2}} F_{1}(s,\omega)||_{L^{2}(\rho(\omega\lambda(s)^{2})d\omega)} \left(\left(\frac{\lambda(t)}{\lambda(s)}\right)^{5/2} + \left(\frac{\lambda(t)}{\lambda(s)}\right)^{3/2}\right) ds\\
&\leq \frac{2 \sqrt{C_{\rho}} \log^{90}(t)}{t^{3+\epsilon-\frac{3}{2}C_{l}}} \left(\frac{C_{1}}{C_{Z}} + 2 M (1+||\mathcal{K}||_{\mathcal{L}(L^{2,\frac{1}{2}}_{\rho})})\right)\\
&< \frac{2 \sqrt{C_{\rho}} \log^{90}(t)}{t^{3+\epsilon-\frac{3}{2}C_{l}}} \left(\frac{C_{1}}{C_{Z}}+\frac{1}{24 \sqrt{C_{\rho}}}\right)\end{split}\end{equation}
where we remind the reader that the constant $C_{1}$ from Lemma \ref{rhslemma} is \emph{independent} of $C_{Z}, T_{0}$. Finally, with the same procedure,
\begin{equation} \begin{split} &\lambda(t) ||\omega\lambda(t)^{2} T(\begin{bmatrix}y_{0}\\
y_{1}\end{bmatrix})_{1}(t,\omega)||_{L^{2}(\rho(\omega\lambda(t)^{2}) d\omega)}\leq 2 \sqrt{C_{\rho}} \lambda(t) \int_{t}^{\infty} \lambda(s) ||\sqrt{\omega\lambda(s)^{2}} F_{1}(s,\omega)||_{L^{2}(\rho(\omega\lambda(s)^{2}) d\omega)} \left(\frac{s}{t}\right)^{\frac{5}{2}C_{l}} ds\\
&\leq \frac{2 \lambda(t) \sqrt{C_{\rho}} \log^{90}(t)}{t^{3+\epsilon-\frac{3}{2}C_{l}}} \left(C_{1}+2 M C_{Z}(1+||\mathcal{K}||_{\mathcal{L}(L^{2,\frac{1}{2}}_{\rho})})\right)\\
&< \frac{2 \lambda(t) \sqrt{C_{\rho}} \log^{90}(t)}{t^{3+\epsilon-\frac{3}{2}C_{l}}}\left(C_{1}+\frac{C_{Z}}{24\sqrt{C_{\rho}}}\right)\end{split}\end{equation}
We therefore get that there exists $C>0, C_{1}>0$, independent of $C_{Z}, T_{0}$, such that, for all $y \in \overline{B_{1}(0)} \subset Z$, 
\begin{equation}\begin{split}||T(y)||_{Z} &\leq \frac{5}{12} + \frac{2 \sqrt{C_{\rho}} C_{1}}{C_{Z}}\\
&+\sup_{t \geq T_{0}} \left(C(\frac{\log^{10}(t)}{t^{1-\frac{5}{2}C_{u}-\frac{3}{2}C_{l}}}+\frac{\log^{90}(t)}{t^{\epsilon-3C_{l}}})+\frac{2 \sqrt{C_{\rho}}}{t} \lambda(t)(C_{1}+\frac{C_{Z}}{24 \sqrt{C_{\rho}}})+\frac{C \lambda(t)}{t}\right)\end{split}\end{equation}
Using the fact that $F_{2}$ is linear in $\begin{bmatrix} y_{0}\\
y_{1}\end{bmatrix}$ and Lemmas \ref{f2estlemma}, \ref{f3estlemma}, we get that  there exists $C>0, C_{3}>0$, independent of $C_{Z}, T_{0}$, such that, for all $y,z \in \overline{B_{1}(0)} \subset Z$, 
\begin{equation}\begin{split}||T(y)-T(z)||_{Z} &\leq C ||y-z||_{Z} \begin{aligned}[t]&\left(\sup_{t \geq T_{0}}\left(\frac{1}{t^{1-C_{u}}}\right) + \sup_{t \geq T_{0}}\left(\frac{\log^{90}(t)}{t^{\epsilon-3C_{l}}}\right)+\sup_{t \geq T_{0}}\left(\frac{\log^{10}(t)}{t^{1-\frac{5}{2}C_{u}-\frac{3}{2}C_{l}}}\right)\right)\end{aligned}\\
&+||y-z||_{Z}\left(\frac{5}{12}+\frac{2 \sqrt{C_{\rho}}C_{3}}{C_{Z}}\right) + \frac{||y-z||_{Z} C_{Z}}{12} \sup_{t \geq T_{0}}\left(\frac{1}{t^{1-C_{u}}}\right)\end{split}\end{equation}
Now, we fix $C_{Z} = 24 \sqrt{C_{\rho}} \text{max}\{C_{1},C_{3}\}$. Since $\epsilon>3C_{l}$, there exists $T_{4}>0$ so that, for all $T_{0}\geq T_{4}$, $T$ is a strict contraction on $\overline{B_{1}(0)} \subset Z$. By Banach's fixed point theorem, there thus exists a fixed point of $T$, say $\overline{y}  \in \overline{B_{1}(0)} \subset Z$. By the definition of $T$, $\overline{y}$ is a solution to \eqref{yeqn}.
\subsection{Decomposition of the solution as in Theorem \ref{mainthm}}
We write $\overline{y} =\begin{bmatrix} \overline{y_{0}}\\
\overline{y_{1}}\end{bmatrix}$ and define $v_{f}$ by \begin{equation}v_{f}(t,r) = \frac{\lambda(t)}{r}\left(\overline{y_{0}}(t) \phi_{d}(\frac{r}{\lambda(t)})+\int_{0}^{\infty} \phi(\frac{r}{\lambda(t)},\xi) \overline{y_{1}}(t,\frac{\xi}{\lambda(t)^{2}}) \rho(\xi) d\xi\right), \quad r >0\end{equation}
By our derivation of \eqref{yeqn} from \eqref{veqn}, $v_{f}$ is a solution to \eqref{veqn}. Let
$$v_{rad}(t,r):= v_{2}(t,r) + v_{3}(t,r) + v_{4}(t,r) + \widetilde{v_{2}}(t,r)$$
so that
$$(-\partial_{t}^{2}+\partial_{r}^{2}+\frac{2}{r}\partial_{r})v_{rad}(t,r) =0$$
and let
$$v_{e}(t,r):= u_{ansatz}(t,r)-v_{rad}(t,r)$$
Defining $u_{e}(t,r) = v_{e}(t,r) + v_{f}(t,r)$, we get that $u$ given by
$$u(t,r):= Q_{\lambda(t)}(r) + v_{rad}(t,r) + u_{e}(t,r)$$
is a solution to \eqref{slw}. We conclude the proof of Theorem \ref{mainthm} by showing 
\begin{equation}\label{enconds}E_{SLW}(u,\partial_{t}u)< \infty, \quad E(u_{e},\partial_{t}(Q_{\lambda(t)}+u_{e})) \leq \frac{C \left(\sup_{x \in [T_{\lambda},t]} \sqrt{\lambda(x)}\right)^{2}}{t}\end{equation}
We use Lemmas \ref{uellminusuellmainestlemma} (for $u_{ell}$), \ref{v2estlemma} (for $v_{2}$), \ref{v30estlemma}, \ref{v3minusv3mainestlemma} (for $v_{3}$), \ref{uell2minusuell2mainestlemma} (for $u_{ell,2}$), \ref{uell3estlemma} (for $u_{ell,3}$), \ref{w1estlemma} (for $\widetilde{v_{2}}, w_{1}$), \ref{vexestlemma} (for $v_{ex}$), \ref{uw2estlemma}(for $u_{w,2}$), \ref{u3estlemma} (for $u_{3}$), \ref{un0estlemma} (for $u_{N_{0}},v_{4}$), \ref{u4estlemma} (for $u_{4}$), \ref{un0ellestlemma} (for $u_{N_{0},ell}$), and \ref{un2estlemma} (for $u_{N_{2}}$) to get
$$||\partial_{r}v_{e}(t,r)||_{L^{2}(r^{2} dr)} + ||v_{e}(t,r)||_{L^{6}(r^{2} dr)} \leq \frac{C \sup_{x \in [T_{\lambda},t]}(\sqrt{\lambda(x)})}{\sqrt{t}},\quad ||v_{rad}(t,r)||_{L^{6}(r^{2} dr)} \leq \frac{C \sup_{x \in [T_{\lambda},t]}(\sqrt{\lambda(x)})\log(t)}{\sqrt{t}}$$
We also need to capture an exact cancellation between $\partial_{t}w_{1,0}(t,r)$ and $\partial_{t}Q_{\lambda(t)}(r)$ for large $r$, since each of these functions are not separately in $L^{2}(r^{2} dr)$ (unless $\lambda'(t) =0$). For this, we recall the definitions \eqref{qlambdadef} and \eqref{v20hattildedef}, and note that
\begin{equation}\label{dtw10qlambdacancel}\partial_{t}(w_{1,0}(t,r)+Q_{\lambda(t)}(r)) = \frac{\sqrt{3}}{2}\left(\frac{\lambda'(t)}{\sqrt{\lambda(t)} r} \left(\frac{1}{(1+\frac{3 \lambda(t)^{2}}{r^{2}})^{3/2}}-1\right) -\frac{3 \lambda(t)^{3/2}\lambda'(t)}{(r^{2}+3\lambda(t)^{2})^{3/2}} + \frac{\lambda'(t+r)}{r \sqrt{\lambda(t+r)}}\right)\end{equation}
When $\lambda'(t) \neq 0$, $\partial_{t}Q_{\lambda(t)}(r)$ and $\partial_{t}w_{1,0}(t,r)$ each decay like $\frac{1}{r}$ for large $r$, but, $\partial_{t}(Q_{\lambda(t)}(r) + w_{1,0}(t,r))$ decays faster for large $r$, as the above expression shows. Then, using the above lemmas, we get
$$||\partial_{t}(Q_{\lambda(t)}(r)+v_{e}(t,r))||_{L^{2}(r^{2} dr)} \leq \frac{C \sqrt{h(t)} \sup_{x \in [T_{\lambda},t]}\left(\sqrt{\lambda(x)}\right)}{t}$$
Next, we have
$$v_{f}(t,r) = \frac{\lambda(t)}{r} \mathcal{F}^{-1}(\begin{bmatrix} \overline{y_{0}}(t)\\
\overline{y_{1}}(t,\frac{\cdot}{\lambda(t)^{2}})\end{bmatrix})(\frac{r}{\lambda(t)}), \quad r >0$$
Then, by the transference identity of \cite{kstslw}, namely,
$$\mathcal{F}(r \partial_{r}u)(\xi) = \begin{bmatrix}0\\
-2 \xi \mathcal{F}(u)_{1}'(\xi)\end{bmatrix} + \mathcal{K}(\mathcal{F}(u))(\xi), \quad \mathcal{F}(u)(\xi) = \begin{bmatrix}\mathcal{F}(u)_{0}\\
 \mathcal{F}(u)_{1}(\xi)\end{bmatrix}$$
we get
\begin{equation}\begin{split}||\partial_{t}v_{f}||_{L^{2}(r^{2}dr)} &\leq C \frac{|\lambda'(t)|}{\lambda(t)} ||v_{f}||_{L^{2}(r^{2}dr)} + ||\frac{\lambda(t)}{r} \mathcal{F}^{-1}(\begin{bmatrix} \overline{y_{0}}'(t)\\
\partial_{1}\overline{y_{1}}(t,\frac{\cdot}{\lambda(t)^{2}})\end{bmatrix})(\frac{r}{\lambda(t)})||_{L^{2}(r^{2}dr)} \\
&+ ||\frac{\lambda(t)}{r} \left(\frac{-\lambda'(t)}{\lambda(t)}\right) \mathcal{F}^{-1}\left(\mathcal{K}\left(\begin{bmatrix} \overline{y_{0}}(t)\\
\overline{y_{1}}(t,\frac{\cdot}{\lambda(t)^{2}})\end{bmatrix}\right)\right)\left(\frac{r}{\lambda(t)}\right)||_{L^{2}(r^{2}dr)}\\
&\leq \frac{C \lambda(t)^{3/2} \log^{90}(t)}{t^{3+\epsilon-\frac{3}{2}C_{l}}}\end{split}\end{equation}
By Lemma \ref{vytranslemma}, 
$$||v_{f}||_{L^{2}(r^{2} dr)}+\lambda(t) ||\partial_{r}v_{f}||_{L^{2}(r^{2} dr)}\leq \frac{C \lambda(t)^{3/2} \log^{90}(t)}{t^{2+\epsilon-\frac{3}{2}C_{l}}}$$
In particular, using the Sobolev embedding $H^{1}(\mathbb{R}^{3}) \subset L^{6}(\mathbb{R}^{3})$, we get
$$||v_{f}(t,r)||_{L^{6}(r^{2} dr)} \leq \frac{C \sqrt{\lambda(t)} \log^{90}(t)}{t^{2+\epsilon-\frac{3}{2}C_{l}-C_{u}}}$$
From \eqref{v20tildeests}, \eqref{v40ests}, Lemma \ref{v30estlemma}, and \eqref{v20ests}, we get that $E(v_{rad},\partial_{t}v_{rad})<\infty$, which finishes the verification of \eqref{enconds}.
\printbibliography
\end{document}